\documentclass[12pt]{article}
\usepackage{a4wide}
\usepackage{amsmath, amsthm, amsfonts, amssymb, bbm}
\usepackage{graphics}
\usepackage{xypic}
\usepackage[all]{xy}
\usepackage[english]{babel}
\usepackage[font=small,format=plain,labelfont=bf,up]{caption}

\usepackage{ifpdf} 
\ifpdf
\usepackage{epstopdf}
\usepackage{hyperref}
\else
\usepackage[hypertex]{hyperref}
\fi

\theoremstyle{plain}

\newtheorem{thm}{Theorem}
\newtheorem{lem}[thm]{Lemma}

\theoremstyle{definition}

\newtheorem{rem}[thm]{Remark}
\newtheorem{defn}[thm]{Definition}

 \numberwithin{equation}{section}
 \numberwithin{thm}{section}

\DeclareMathOperator{\Hom}{Hom}

\DeclareMathOperator{\tr}{tr}

\newcommand\be            {\begin{equation}}
\newcommand\ee            {\end{equation}}

\newcommand\nxt{\noindent\raisebox{.08em}{\rule{.44em}{.44em}}\hspace{.4em}}
\newcommand\dwt[1] {\underline{\smash #1}}  
\newcommand\bord{\mathrm{Bord}}
\newcommand\Tcw{T^\mathrm{cw}}
\newcommand\cytens[1]{\circlearrowleft_{#1}\!}

\newcommand\cospd[5]{\raisebox{1.2em}{\small \xymatrix@R=0.8em@C=1.5em{ & #3 \\ #1 \ar[ru]^(.38){#2} && #5 \ar[lu]_(.38){#4}}}}
\newcommand\cospdm[9]{\raisebox{2.2em}{\small \xymatrix@R=0.8em@C=1.5em{ & #3 \ar[dd]^{#9} \\ #1 \ar[ur]^(.35){#2} \ar[dr]_(.35){#6} && #5 \ar[ul]_(.35){#4} \ar[dl]^(.35){#8} \\ & #7 }}}

\newcommand\cospddauxa[2]{#1 \ar[d]_{#2}}
\newcommand\cospddauxb[2]{#1 \ar[u]_{#2}}

\newcommand\cospdd[9]{\raisebox{2.9em}{\small \xymatrix@R=1.5em@C=1.5em{& \cospddauxa#3 \\ #1 \ar@/^4pt/[ur]^(.35){#2} \ar@/_4pt/[dr]_(.35){#6} & #9 & #5 \ar@/_4pt/[ul]_(.35){#4} \ar@/^4pt/[dl]^(.35){#8} \\ & \cospddauxb#7 }}}

\newcommand\vect{\mathcal{V}\hspace{-.5pt}ect}
\newcommand{\cosp}{\mathcal{C}\hspace{-1pt}osp}
\newcommand{\tdiag}{\mathcal{D}\hspace{-.5pt}iag}
\newcommand{\alg}{\mathcal{A}\hspace{.2pt}lg}
\newcommand{\Mor}{\mathcal{M}or}

\newcommand{\Alg}{\mathbf{Alg}}				
\newcommand{\Cosp}{\mathbf{Cosp}}				
\newcommand{\CAlg}{\mathbf{C\hspace{-1pt}Alg}}		

\newcommand{\CALG}{\mathrm{C\hspace{-.8pt}A\hspace{-.3pt}L\hspace{-.4pt}G}}		

\newcommand{\Cospu}{{\underline{\mathbf{Cos\phantom{a}}}}\hspace*{-0.5em}\mathbf{p}}
\newcommand{\CALGu}{{\underline\CALG}}		

\newcommand\eps           {\varepsilon}
\newcommand\id            {{\rm id}}
\newcommand\one           {{\bf1}}

\newcommand\Cb            {\mathbb{C}}

\newcommand\Rb            {\mathbb{R}}
\newcommand\Zb            {\mathbb{Z}}

\newcommand\bfI            {\mathbf{I}}
\newcommand\bfB            {\mathbf{B}}
\newcommand\bfD            {\mathbf{D}}
\newcommand\bfE            {\mathbf{E}}
\newcommand\bfZ            {\mathbf{Z}}

\begin{document}

\thispagestyle{empty}
\def\thefootnote{\fnsymbol{footnote}}
\begin{flushright}
ZMP-HH/11-8\\
Hamburger Beitr\"age zur Mathematik 407
\end{flushright}
\vskip 3em
\begin{center}\LARGE
Field theories with defects and the centre functor
\end{center}

\vskip 2em
\begin{center}
{\large 
Alexei Davydov$^{a}$,\, Liang Kong$^{b}$,\,  Ingo Runkel$^{c}$,\, ~\footnote{Emails: {\tt alexei1davydov@gmail.com}, {\tt kong.fan.liang@gmail.com}, {\tt ingo.runkel@uni-hamburg.de}}}
\\[1em]
\it$^a$ 
Department of Mathematics and Statistics\\
University of New Hampshire, Durham NH, USA
\\[1em]
$^b$ Institute for Advanced Study (Science Hall) \\ 
Tsinghua University, Beijing 100084, China
\\[1em]
$^c$ Fachbereich Mathematik, Universit\"at Hamburg\\
Bundesstra\ss e 55, 20146 Hamburg, Germany
\end{center}

\vskip 2em
\begin{center}
  July 2011
\end{center}
\vskip 2em

\begin{abstract}
This note is intended as an introduction to the functorial formulation of quantum field theories with defects. After some remarks about models in general dimension, we restrict ourselves to two dimensions -- the lowest dimension in which interesting field theories with defects exist. 

We study in some detail the simplest example of such a model, namely a topological field theory with defects which we describe via lattice TFT. Finally, we give an application in algebra, where the defect TFT provides us with a functorial definition of the centre of an algebra. This involves changing the target category of commutative algebras into a bicategory. 

Throughout this paper, we emphasise the role of higher categories -- in our case bicategories -- in the description of field theories with defects.
\end{abstract}

\setcounter{footnote}{0}
\def\thefootnote{\arabic{footnote}}

\newpage

\tableofcontents

\section{Introduction}

One way to think about quantum field theory -- motivated by conformal field theory and string theory \cite{Friedan:1986ua,Vafa:1987ea} -- is as functors from bordisms to vector spaces \cite{Segal:1988,Atiyah:1989vu}; here, each of the terms `functor', `bordism', `vector space' has to be supplemented with the appropriate qualifiers for the application in mind. In its most basic form, the bordisms for an $n$-dimensional quantum field theory form a symmetric monoidal category whose objects are $(n{-}1)$ dimensional manifolds equipped with `collars' and whose morphisms are equivalence classes of $n$-dimensional manifolds with parametrised boundary. 

To study quantum field theories beyond this basic functorial definition, it is often appropriate to employ higher categories. There are two natural ways in which such higher categories enter.
\begin{enumerate}
\item
The $(n{-}1)$-manifolds which form the objects in the above bordism category could in turn be obtained by gluing $(n{-}1)$-manifolds along $(n{-}2)$-manifolds, and so on, down to $0$-manifolds, i.e.\ points. This process is called `extending the field theory down to points' \cite{Freed:1994ad,Lawrence:1993,Baez:1995xq,Lurie:2009aa}. One obtains a higher category whose objects are now points (with extra structure) and which has $1$-morphisms, $2$-morphisms, \dots, up to $n$-morphisms. The resulting field theories are most studied in the case of topological field theories \cite{Lurie:2009aa}.
\item
One can let the bordisms remain a (1-)category but equip them with extra structure, namely with `defects'. These are submanifolds embedded in the $(n{-}1)$- and $n$-dimensional bordisms, decorated with labels which describe different possible `defect conditions'. A field theory on bordisms with defects equips the set of defect conditions with the structure of a higher category. 
\end{enumerate}
Here we want to elaborate on the second point. Some other works which also stress the appearance of higher categories in field theories with defects, and which the reader could consult for further references, are \cite{Schweigert:2006af,Lurie:2009aa,Bartels:2009ts,Kapustin:2010ta,Kitaev:2011a}. In the present paper, we will concentrate on the simplest interesting class of models, namely two-dimensional field theories with defects. In section \ref{sec:field-def} we will see how a field theory with a particular type of defects -- so called topological defects -- gives rise to a 2-category defined in terms of the set of defect conditions. In section \ref{sec:ex-latticeTFT} we use lattice topological field theory to construct a very simple but still non-trivial example of a field theory with defects. This example will motivate -- in section \ref{sec:centre} -- a nice mathematical construction, namely a method to make the assignment which maps an algebra to its centre functorial. Section \ref{sec:outlook} contains an outlook on further developments.

These three constructions -- the 2-category of topological defects (section \ref{sec:2-cat-from-defect-QFT}), lattice topological field theory with defects (theorem \ref{thm:lattice-defect-TFT}), and the centre functor (theorem \ref{thm:alg->CAlg-laxfun} and remark \ref{rem:centre-2nd-version}) -- are the main points of this paper. We hope that they provide some intuition on how to work with field theories containing defects and illustrate their usefulness. 

\section{Field theory with defects}\label{sec:field-def}

\subsection{Bordisms with defects}\label{sec:bord-with-def}

It is beyond the scope of this article (and the present abilities of the authors) to develop an all-purpose formalism for field theories with defects. In this subsection we briefly sketch the basic features of the functorial formulation of field theory with defects.\footnote{
In doing so will omit most details. For those who are familiar with the functorial formulation, some of these details are: we should equip our object-$(n{-}1)$-manifolds with collars to ensure a well-defined gluing operation; objects and morphisms could carry extra geometric data such as a metric, a spin structure, etc.; we should work with families to have a natural notion of continuous or smooth dependence of the functor on the bordism; the functor from bordisms to vector spaces is symmetric monoidal; the target category of the functor consists of topological vector spaces with an appropriate tensor product. These issues are treated carefully in \cite{Stolz:unpubl}.}

As usual, a field theory will be a functor from a bordism category to a category of vector spaces. In the presence of defects, the target category of the functor remains unchanged. However, we do modify the source category. The category of $n$-{\em dimensional bordisms with defects} contains the following ingredients. 
\begin{itemize}
\item 
{\em Sets of defect conditions:} 
the bordism category will depend on a choice of $n{+}1$ (possibly empty) sets, $D_k$, $k=0,\dots,n$. The elements of $D_k$ serve as defect conditions for $k$-dimensional defects.
\item 
{\em Objects:} 
the objects are $(n{-}1)$-dimensional compact oriented manifolds $U$ with empty boundary, together with a disjoint decomposition into submanifolds. That is, $U = \bigcup_{i=0}^{n-1} U_i$, where each $U_i$ is an $i$-dimensional oriented submanifold of $U$ and $U_k \cap U_l = \emptyset$ for $k\neq l$.\footnote{\label{fn:defect-decomp}
  We also demand that the partial union $\bigcup_{i=0}^{k} U_i$ is a closed subset of $U$ for $k=0,\dots,n{-}1$; this ensures that $\bar U_k \setminus U_k$ (the difference of $U_k$ and its closure) is contained in the union $\bigcup_{i=0}^{k-1} U_i$ of lower dimensional pieces. Let us give a non-example in $U=S^3$, which we present as the one-point compactification of $\Rb^3$. Take $U_0 = \emptyset$, $U_1 = (-1,1) \times \{(0,0)\}$, $U_2= S^2 \subset \Rb^3$, $U_3 = U \setminus (U_1 \cup U_2)$. In this case, all $U_i$ are submanifolds, but $U_0 \cup U_1$ is not closed, which is not allowed. (But $U_0 \cup U_1 \cup U_2$ is closed). To turn this into an allowed decomposition, take instead $U'_0 = \{ (\pm 1,0,0) \}$ and $U'_2 = S^2 \setminus \{ (\pm 1,0,0) \}$. Then $U_0'$, $U_1$, $U_2'$, $U_3$ is an allowed decomposition of $U$.
} 
The orientation of $U_{n-1}$ is induced by that of $U$.

For example, $U_{n-1}=U$ and $U_k = \emptyset$ for $k<n{-}1$ would be a possibility, or, if $U_{i}$ ($i<n{-}1$) is a closed submanifold of $U$, then we can take $U_{n-1} = U \setminus U_{i}$ and $U_k=\emptyset$ for $k \neq i,n{-}1$. 

Finally, each connected component of $U_k$ is decorated with a defect condition from $D_{k+1}$, i.e.\ we have a collection of maps $d_{k+1} : \pi_0(U_k) \to D_{k+1}$; the reason for the shift in $k$ is that the $U_k$ will appear as boundaries of $(k+1)$-dimensional submanifolds in the $n$-dimensional manifold making up a morphism.
\item 
{\em Morphisms:} a morphism $M : U \to V$ has a structure analogous to objects, except in one dimension higher. In more detail, $M$ is an $n$-dimensional compact oriented manifold, together with a decomposition $M = \bigcup_{i=0}^{n} M_i$, where each $M_i$ is an $i$-dimensional oriented submanifold, possibly with non-empty boundary $\partial M_k$, and $M_k \cap M_l = \emptyset$ for $k\neq l$ (and footnote \ref{fn:defect-decomp} applies analogously). The orientation of $M_n$ is induced by that of $M$. Each connected component of $M_k$ is labelled by a defect condition, but this time from $D_k$, that is, we have maps $\hat d_k : \pi_0(M_k) \to D_k$.

The boundary $\partial M$ is identified via an orientation preserving diffeomorphism (which is part of the data of a morphism) with the disjoint union ${-}U \sqcup V$; we require that $\partial M_k \subset \partial M$, and that the resulting decomposition and labelling of $\partial M$ agrees with the one induced by ${-}U \sqcup V$.
\end{itemize}
For example, in $n=3$ dimensions, a generic morphism would look like a foam, where the interior of each bubble is `coloured' by an element of $D_3$, the walls between two bubbles by elements of $D_2$, lines along which the walls between bubbles meet by elements of  $D_1$, and points where these lines meet by elements of  $D_0$ (somewhat problematic in the foam analogy, but nonetheless allowed, are 1- and 0-dimensional submanifolds not attached to any walls).

While we have given the overall name `defect conditions' to elements of the sets $D_k$, more descriptive names in the various dimensions $0 \le k \le n$ would be that they are conditions for
\begin{itemize}
\item $D_n$: domains (or phases of the field theory)
\item $D_{n-1}$: domain walls (or phase boundaries)
\item $D_{n-2},\dots,D_0$: junctions
\end{itemize}
The sets $D_k$ are equipped with additional structure describing in which geometric configurations the domains can occur. This is complicated in general, but it is easy to state for domain walls. Since the $n$-dimensional manifold underlying the morphism and the $(n-1)$-dimensional submanifold underlying the domain wall are oriented, we can speak of the `left and right side' of the domain wall. Accordingly, there are two maps 
\be
  s,t ~:~ D_{n-1} \longrightarrow D_{n}
\ee
(for `source' and `target'), and a domain wall of type $x \in D_{n-1}$ must have a domain labelled by $s(x)$ on its left and $t(x)$ on its right. This gives a restriction on the allowed maps $d_n, d_{n-1}$ in objects and $\hat d_n, \hat d_{n-1}$ in morphisms. In this work, we will only discuss the cases $n=1$ and $n=2$ and our orientation conventions are shown in figure \ref{fig:or-convention}.

\begin{figure}[tb] 
\begin{center}
\raisebox{25pt}{a)}
\raisebox{-5pt}{\begin{picture}(85,25)
  \put(0,5){\scalebox{0.80}{\includegraphics{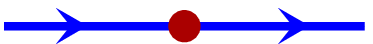}}}
  \put(0,5){
     \setlength{\unitlength}{0.80pt}\put(-0,-0){
     \put(48, 15)   {\scriptsize $+$ }
     \put(49, -10)   {\scriptsize $x$ }
     \put(5, -10)   {\scriptsize $s(x)$ }
     \put(85, -10)   {\scriptsize $t(x)$ }
  }}
\end{picture}}
\hspace{3em}
\raisebox{25pt}{b)}
\raisebox{-5pt}{\begin{picture}(85,25)
  \put(0,5){\scalebox{0.80}{\includegraphics{pic1a.eps}}}
  \put(0,5){
     \setlength{\unitlength}{0.80pt}\put(-0,-0){
     \put(48, 15)   {\scriptsize $-$ }
     \put(49, -10)   {\scriptsize $x$ }
     \put(5, -10)   {\scriptsize $t(x)$ }
     \put(85, -10)   {\scriptsize $s(x)$ }
  }}
\end{picture}}
\hspace{3em}
\raisebox{25pt}{c)}~
\raisebox{-30pt}{\begin{picture}(80,70)
  \put(0,3){\scalebox{1.00}{\includegraphics{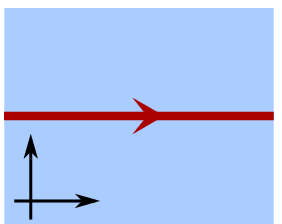}}}
  \put(0,3){
     \setlength{\unitlength}{1.00pt}\put(-261,-384){
     \put(272,403)   {\scriptsize $2$ }
     \put(282,394)   {\scriptsize $1$ }
     \put(290,419)   {\scriptsize $x$ }
     \put(315,428)   {\scriptsize $t(x)$ }
     \put(315,397)   {\scriptsize $s(x)$ }
  }}
\end{picture}}
\end{center}
\vspace*{-1em}
\caption{Figures a)--c) show open subsets of a bordism in dimension $n=1$ and $n=2$. 
They give our orientation convention in the compatibility condition for the assignment of defect conditions in the case $n=1$ (figs.\,a,\,b) and $n=2$ (fig.\,c). The arrows represent positively oriented ordered bases. }
\label{fig:or-convention}
\end{figure}

\subsection{One-dimensional topological field theory with domain walls} \label{sec:1-dim-tft-def}

Before passing to the more interesting two-dimensional situation, let us briefly discuss the simplest one-dimensional field theory with domain walls, namely the case where the field theory is topological. 

We fix two sets $D_1$ and $D_0$, together with two maps $s, t : D_0 \to D_1$. The objects in the bordism category are finite sets of oriented points $U$, together with a map $d_1 : U \to D_1$. The morphisms $M : U \to V$ are (diffeomorphism classes of) 1-dimensional manifolds $M$ with a finite set $W$ of marked points in the interior of $M$. Each connected component of $M \setminus W$ is labelled by an element of $D_1$, and each element of $W$ by an element of $D_0$. On the boundary $\partial M$ the $D_1$-labels have to agree with those of $U$, resp.\ $V$.

A symmetric monoidal functor from this bordism category to (necessarily finite dimensional) $k$-vector spaces for some field $k$ is then determined
\begin{itemize}
\item
{\em on objects:} by a collection of vector spaces $(V_i)_{i \in D_1}$. The value of the functor on a point with orientation `$+$' and label $i$ is given by $V_i$, while a point with orientation `$-$' gets mapped to $V_i^*$. On 0-dimensional manifolds with more than one point the functor is fixed by the monoidal structure as usual.
\item
{\em on morphisms:} by two collections of linear maps $(L_x^+)_{x \in D_0}$ and $(L_x^-)_{x \in D_0}$, where $L_x^+ : V_{s(x)} \to V_{t(x)}$ and $L_x^- : V_{t(x)} \to V_{s(x)}$. Let $\eps  \in \{\pm 1\}$. The map $L^\eps_x$ is the value of the functor on the interval $[-1,1]$ with standard orientation, together with the 0-dimensional submanifold $\{0\}$ with orientation $\eps$ and label $x \in D_0$. If $\eps=+$, the sub-interval $[-1,0)$ is labelled by $s(x) \in D_1$ and $(0,1]$ by $t(x) \in D_1$, while for $\eps=-$, the label of $[-1,0)$ and $(0,1]$ is $t(x)$ and $s(x)$, respectively, as in figure\,\ref{fig:or-convention}\,a,\,b). An arbitrary morphism can be obtained by composing and tensoring the above maps, as well as the cup and cap bordisms, which the functor maps to evaluation and co-evaluation.
\end{itemize}
We can collect this data in a category $\mathcal{D}$, together with a distinguished subset of arrows, as follows. Take $D_1$ as objects of $\mathcal{D}$. As space of morphisms $i \to j$, for $i,j \in D_1$, take $\mathcal{D} := \Hom_k(V_i,V_j)$, the linear maps from $V_i$ to $V_j$. Finally, fix a map $D_0 \times \{ \pm\} \to \mathrm{Mor}(\mathcal{D})$, which assigns to $(x,\pm)$ the arrow $L^\pm_x$, with source and target as described above.

\subsection{Two-dimensional metric bordisms with defects}\label{sec:2d-bord-def}

Let us look in more detail at an instance of a bordism category with defects in two dimensions; the exposition essentially follows \cite[Sec.\,3]{Runkel:2008gr}. A note on convention: by manifold we mean smooth manifold, and by a map between manifolds we mean a smooth map; a finite or countable disjoint union has an ordering of its factors, so that for two sets $A$, $B$ the disjoint unions $A \sqcup B$ and $B \sqcup A$ are isomorphic but not equal.

\subsubsection*{Sets of defect conditions}
We start with the three sets $D_2$, $D_1$, and $D_0$, which are the sets of world sheet phases, domain wall conditions, and junction conditions, respectively. As above we have two maps $s,t : D_1 \to D_2$ giving the phase to the left and right of a domain wall; our orientation conventions are shown in figure \ref{fig:or-convention}\,c). For a junction in $D_0$ we need to specify which domain walls can meet with which orientations at a junction point. 

The combinatorial description thereof is a bit lengthy: Let $D_1^{(n)}$ be the set of tuples of $n$ {\em cyclically composable domain walls}. By this we mean the subset of $n$-fold cartesian product $(D_1 \times \{ \pm \})^{\times n}$ selected by the following condition: For $( (x_1,+), \dots , (x_n,+) )$ we require $t(x_{i+1}) = s(x_{i})$ and $t(x_1) = s(x_n)$. If some of the `$+$' are changed for `$-$', the role of $s$ and $t$ is exchanged as in figure \ref{fig:cycl-comp-example}. The group $C_n$ of cyclic permutations acts on the $n$-tuples in $D_1^{(n)}$. The set $D_0$ is equipped with a map
\be \label{eq:j-definition}
  j : D_0 \longrightarrow \bigsqcup_{n=0}^\infty \Big( D_1^{(n)} / C_n \Big) \ .
\ee
In words, for each element $u$ of $D_0$, the map $j$ determines how many domain walls can end at a junction labelled by $u$ and what their orientations and domain wall conditions are, up to cyclic reordering. 

The map $j$ is similar in spirit to the relation between $D_1$ and $D_2$. There, we can combine the `source' and `target' maps into a single map $(s,t) : D_1 \to D_2 \times D_2$, which determines the world sheet phases that must lie on the two sides of a domain wall labelled by a given element of $D_1$. There is no need to divide by the symmetric group in two elements, because the orientations allow one to distinguish the `left' and `right' side of a domain wall.

\begin{figure}[tb] 
\begin{center}
\raisebox{-43pt}{\begin{picture}(120,100)
  \put(8,8){\scalebox{1.00}{\includegraphics{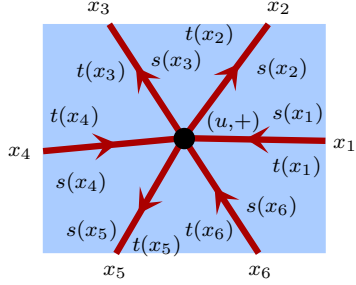}}}
  \put(8,8){
     \setlength{\unitlength}{1.00pt}\put(-15,-18){
     \put(77,67)   {\scriptsize $(u,\!{+})$ }
     \put(125, 59)   {\scriptsize $x_1$ }
     \put(102, 70)   {\scriptsize $s(x_1)$ }
     \put(102, 50)   {\scriptsize $t(x_1)$ }
     \put(100,111)   {\scriptsize $x_2$ }
     \put( 96, 87)   {\scriptsize $s(x_2)$ }
     \put( 69,100)   {\scriptsize $t(x_2)$ }
     \put(32,111)   {\scriptsize $x_3$ }
     \put(56,90)   {\scriptsize $s(x_3)$ }
     \put(28,86)   {\scriptsize $t(x_3)$ }
     \put(  2, 56)   {\scriptsize $x_4$ }
     \put(18,69)   {\scriptsize $t(x_4)$ }
     \put(20,44)   {\scriptsize $s(x_4)$ }
     \put( 38, 11)   {\scriptsize $x_5$ }
     \put(24,26)   {\scriptsize $s(x_5)$ }
     \put(49,20)   {\scriptsize $t(x_5)$ }
     \put( 93, 11)   {\scriptsize $x_6$ }
     \put(69,26)   {\scriptsize $t(x_6)$ }
     \put(92,35)   {\scriptsize $s(x_6)$ }
  }}
\end{picture}}
\vspace*{-1em}
\end{center}
\caption{Illustration of the condition of cyclic composability of domain walls. Given the $n$-tuple $( (x_1,\eps_1), \dots , (x_n,\eps_n) )$, the $i$'th domain wall (counted anti-clockwise) is labelled by $x_i$ and is pointing towards the junction point if $\eps_i=+$ and away from the junction point if $\eps_i=-$. In the present example the 6-tuple is $((x_1,+),(x_2,-),(x_3,-),(x_4,+),(x_5,-),(x_6,+))$. The images under the maps $s$ and $t$ to $D_2$ have to agree as shown, e.g.\ $s(x_1)=s(x_2)$ and $t(x_2)=s(x_3)$. The junction point has orientation `$+$' and is labelled by $u \in D_0$.}
\label{fig:cycl-comp-example}
\end{figure}

\subsubsection*{Objects}

In short, an object is a disjoint union of a finite number of unit circles $S^1$ with marked points, together with a germ of a collar.\footnote{
  This is more restrictive than allowing general one-dimensional manifolds as in section \ref{sec:bord-with-def} but does not loose any generality and has the advantage that objects form a set, and that the connected components of an object are already ordered by our convention on disjoint unions.}

In more detail, for a single $S^1$ the structure is as follows. Take $U=S^1$ to be the unit circle in $\Cb$, decorated as in section \ref{sec:bord-with-def}: a $0$-dimensional submanifold $U_0 \subset S^1$ (i.e.\ a set of points decorated by signs $\pm$), a map $d_1 : \pi_0(U_0) = U_0 \to D_1$, and a map $d_2 : \pi_0(U_1) \to D_2$, where $U_1 = S^1 \setminus U_0$. The maps $d_1$, $d_2$ have to be compatible with $s,t$ as in section \ref{sec:bord-with-def} (cf.\ figure \ref{fig:or-convention}\,a,\,b)).

\begin{figure}[tb] 
\begin{center}
\raisebox{-71pt}{\begin{picture}(170,160)
  \put(0,0){\scalebox{1.00}{\includegraphics{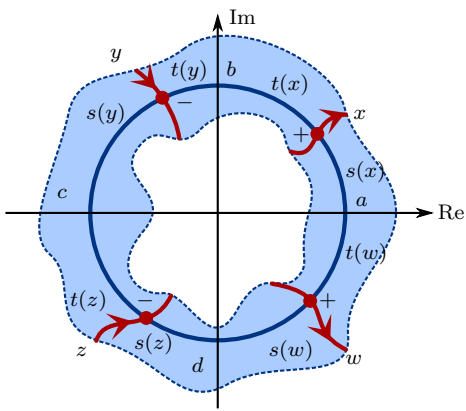}}}
  \put(0,0){
     \setlength{\unitlength}{1.00pt}\put(-7,-4){
     \put(126,43)   {\scriptsize $+$ }
     \put(116,106)   {\scriptsize $+$ }
     \put(72,118)   {\scriptsize $-$ }
     \put(57,43)   {\scriptsize $-$ }
     \put(136,21)   {\scriptsize $w$ }
     \put(107,24)   {\scriptsize $s(w)$ }
     \put(136,61)   {\scriptsize $t(w)$ }
     \put(139,114)   {\scriptsize $x$ }
     \put(136,92)   {\scriptsize $s(x)$ }
     \put(108,125)   {\scriptsize $t(x)$ }
     \put(47,137)   {\scriptsize $y$ }
     \put(70,130)   {\scriptsize $t(y)$ }
     \put(38,114)   {\scriptsize $s(y)$ }
     \put(34,24)   {\scriptsize $z$ }
     \put(32,43)   {\scriptsize $t(z)$ }
     \put(56,27)   {\scriptsize $s(z)$ }
     \put(171,76)   {\scriptsize Re}
     \put(92,150)   {\scriptsize Im}
     \put(140,81)   {\scriptsize $a$}
     \put(91,130)   {\scriptsize $b$}
     \put(27,84)   {\scriptsize $c$}
     \put(78,18)   {\scriptsize $d$}
  }}
\end{picture}}
\end{center}
\vspace*{-1em}
\caption{Illustration of a collar which forms part of the data for an object in the bordism category. In the notation of the text, the solid (blue) circle is a unit circle $U=S^1$, the shaded area is an open neighbourhood $A$, the solid (red) short lines form the oriented submanifold $A_1$ of $A$ which intersects $S^1$ in $U_0$. Our convention for the orientation of $A_1$ induced by that of $U_0$ (the signs `$\pm$') is as shown. The elements $a,b,c,d \in D_2$ label connected components of $U_1$ and their extension $A_2$; these labels have to agree with the source and target maps of the domain wall labels $w,x,y,z \in D_1$ as shown. E.g.\ $t(w)=a=s(x)$.}
\label{fig:collar}
\end{figure}

A {\em collar} is, in short, an extension of the above structure to an open neighbourhood of $S^1$ in $\Cb$, see figure \ref{fig:collar}. Let $A$ be an open neighbourhood of $S^1$, and let $A_1$ be a one-dimensional submanifold, closed in $A$, which intersects $S^1$ transversally (the tangents to $A_1$ and $S^1$ are linearly independent at intersection points). Set $A_2 = A \setminus A_1$. There are maps $\hat d_i : \pi_0(A_i) \to D_i$, $i=1,2$, compatible with $s,t$ as in figure \ref{fig:or-convention}\,c). The restriction of $A_2$ and $A_1$ to $S^1$ has to reproduce $U_1$ and $U_0$ with labelling and orientation, with conventions as in figure \ref{fig:collar}. Finally, $A$ carries a metric in conformal gauge, i.e.\ $g(z)_{ij} = e^{\sigma(z)} \delta_{ij}$ for a real-valued function $\sigma$ on $A$.

Two collars are equivalent if they agree in some open neighbourhood of $S^1$; an equivalence class is called a {\em germ of collars}.

For a disjoint union $U$ of such $S^1$ with collars, write $U_\text{in}$ for the subset obtained by taking only points $|z| \ge 1$ in the collar of each $S^1$, and $U_\text{out}$ when taking only points with $|z| \le 1$.

\subsubsection*{Morphisms}

Morphisms $M : U \to V$ are equivalence classes of surfaces with extra structure as in section \ref{sec:bord-with-def}, together with a metric. Thus we have a decomposition $M = M_2 \cup M_1 \cup M_0$, maps $\hat d_i : \pi_0(M_i) \to D_i$, $i=0,1,2$. The map $\hat d_1$ is compatible with $s,t$ as in figure \ref{fig:or-convention}\,c). Going beyond the level of detail in section \ref{sec:bord-with-def}, we also require the following:

\smallskip\noindent
\nxt {\em compatibility condition for} $\hat d_0$: For a point $p \in M_0$ labelled by $u \in D_0$ (i.e.\ $\hat d_0(p)=u$), let $((x_1,\eps_1),\dots,(x_n,\eps_n))$ denote the domain wall conditions and orientations in anti-clockwise order (with arbitrary starting point). We require that $j(u)$ is the cyclic permutation equivalence class of $((x_1,\eps_1),\dots,(x_n,\eps_n))$ if the junction point has orientation `$+$', cf.\ figure \ref{fig:cycl-comp-example}, and that it is in the class of $((x_n,-\eps_n),\dots,(x_1,-\eps_1))$ if the junction orientation is `$-$'. Junctions with opposite orientation are dual in the following sense (figure \ref{fig:junction-orientation}): if the bordism is a 2-sphere with two antipodal junction points both labelled by $u$ but with orientations `$+$' and `$-$', the domain walls starting at the two junctions can be joined up (intersection-free) by half-circles around the $S^2$.

\begin{figure}[tb] 
\begin{center}
\raisebox{-42pt}{\begin{picture}(155,85)
  \put(0,0){\scalebox{.60}{\includegraphics{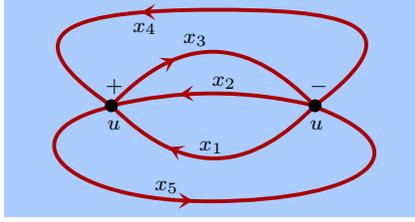}}}
  \put(0,0){
     \setlength{\unitlength}{.60pt}\put(-17,-15){
     \put(82,70)   {\scriptsize $u$ }
     \put(81,94)   {\scriptsize $+$ }
     \put(210,70)   {\scriptsize $u$ }
     \put(210,94)   {\scriptsize $-$ }
     \put(140,57)   {\scriptsize $x_1$ }
     \put(148,98)   {\scriptsize $x_2$ }
     \put(130,125)   {\scriptsize $x_3$ }
     \put(98,132)   {\scriptsize $x_4$ }
     \put(112,32)   {\scriptsize $x_5$ }
  }}
\end{picture}}
\end{center}
\vspace*{-1em}
\caption{Some domain walls and two junctions placed on $S^2$; we only display a fragment after projection to the plane. The two junctions are labelled by the same junction condition $u \in D_0$ but with opposite orientation `$\pm$'. Here, $j(u)$ is the cyclic permutation equivalence class of $((x_1,+),(x_2,+),(x_3,-),(x_4,+),(x_5,-))$. Thus, the junction labelled by $u$ with orientation `$+$' must have domain walls $(x_1,+),(x_2,+),(x_3,-),(x_4,+),(x_5,-)$ attached in anti-clockwise order, where for $(x_i,+)$ the domain wall is oriented towards the junction and for $(x_i,-)$ it is oriented away from the junction. The junction labelled by $u$ with orientation `$-$' must have domain walls $(x_5,+),(x_4,-),(x_3,+),(x_2,-),(x_1,-)$ attached in anti-clockwise order.}
\label{fig:junction-orientation}
\end{figure}

\smallskip\noindent
\nxt {\em boundary parametrisation}: A choice $U'$, $V'$ of collars representing the germs $U$, $V$, together with injective maps $f_\text{in} : U_\text{in} \to M$ and $f_\text{out} : V_\text{out} \to M$ which preserve the orientation, metric, boundary, 1-dimensional submanifold (with orientation) and labelling. The images of the factors $S^1$ in $U'$ and $V'$ are disjoint and cover the boundary of $M$.

\smallskip

Two surfaces are equivalent if they are isometric and the isometry preserves the decomposition $M = M_2 \cup M_1 \cup M_0$ together with orientations and labelling, and commutes with the boundary parametrisation in some open neighbourhood of $\partial M$.

Composition of morphisms is defined by choosing representatives and gluing via the boundary parametrisation; the collars ensure that this does not introduce `corners' and results again in a surface as described above. The equivalence class of the glued surface is independent of the choice of representatives.

\subsubsection*{Identities and symmetric structure}

So far there are no identity morphisms. We will add these by hand by extending the morphisms to include permutations of the $S^1$ factors in the disjoint union of a given object $U$. If we denote the permutation by $\sigma$ and the permuted disjoint union by $\sigma(U)$, we add morphisms $\sigma : U \to \sigma(U)$. Each morphism $U \to V$ is either a permutation (only possible if $V = \sigma(U)$) or a bordism; composing $\sigma^{-1}(U) \xrightarrow{\sigma} U \xrightarrow{M} V \xrightarrow{\tau} \tau(V)$ produces a bordism $\sigma^{-1}(U) \xrightarrow{M'} \tau(V)$, where $M'$ differs from $M$ only in the boundary parametrisation maps.

This endows the bordism category with a symmetric structure (the tensor product is disjoint union as usual).

\subsubsection*{Topological domain walls and junctions}

Denote the symmetric monoidal category described above by 
\be \label{eq:Bord21-def-symbol}
  \bord_{2,1}^\mathrm{def}(D_2,D_1,D_0) \ ,
\ee
or $\bord_{2,1}^\mathrm{def}$ for short. A two-dimensional quantum field theory with defects can now be defined as a symmetric monoidal functor $Q$ from $\bord_{2,1}^\mathrm{def}$ to topological vector spaces, which depends continuously on the moduli (namely, the metric on a morphism $M$, the decomposition $M = M_2 \cup M_1 \cup M_0$, and the boundary parametrisation). 

\begin{rem}
It is also easy to say when such a functor $Q$ describes a {\em conformal field theory with defects}. Namely, the vector space $Q(U)$ assigned to an object $U$ has to be independent of the conformal factor $e^{\sigma(z)}$ giving the metric $g(z)_{ij} = e^{\sigma(z)} \delta_{ij}$ on $U$, and the linear map $Q(M)$ assigned to a morphism $M$ changes by at most a scalar factor if the metric on $M$ is changed by a conformal factor $g \leadsto e^f g$ for some $f : M \to \Rb$. Thus $Q$ would in general only give a projective functor if one passes to conformal equivalence classes of manifolds (it would be a true functor if the so-called central charge vanishes).
\end{rem}

With respect to a chosen $Q$, we can define an interesting subset of domain walls and junctions:
\begin{itemize}
\item {\em topological domain walls} are elements $x$ of $D_1$ such that 
\begin{enumerate}
\item 
for all objects $U$, the vector space $Q(U)$ is unchanged under isotopies moving components of $U_0$ labelled by $x$ (and their extension into the collars with them)  such that no point of $U_0$ crosses the point  $-1 \in S^1$. This condition renders the space of such isotopies contractible (on germs of collars). In particular, a full $2\pi$-rotation is excluded, as it would in general induce a non-trivial endomorphism of $Q(U)$. The metric on $U$ stays fixed.
\item 
for all morphisms $M$, $Q(M)$ is invariant under isotopies moving components of $M_1$ labelled by $x$ while leaving $M_0$ fixed and restricting on $\partial M$ to isotopies respecting the condition in 1.
\end{enumerate}
\item {\em topological junctions} are elements $u$ of $D_0$ such that $j(u)$ only contains elements of $D_1$ labelling topological domain walls, and such that $Q(M)$ is invariant under isotopies moving components of $M_1$ labelled by topological domain wall conditions and points in $M_0$ labelled by $u$.
\end{itemize}
From now on we will concentrate on topological domain walls and junctions. We will denote the corresponding subsets by $D_i^\mathrm{top}(Q)$, $i=0,1$, or just $D_i^\mathrm{top}$.

\subsection{2-categories of defect conditions}\label{sec:2-cat-from-defect-QFT}

Let us fix a two-dimensional field theory with defects as above, i.e.\ a  continuous symmetric monoidal functor $Q$ from $\bord_{2,1}^\mathrm{def}(D_2,D_1,D_0)$ to an appropriate category of topological vector spaces. The construction below will only make use of the sets $D_2$ and $D_1$, but not of $D_0$. To emphasise this, we take $D_0 = \emptyset$ (i.e.\ no junctions are allowed). 

\medskip

Consider the topological domain walls $s,t : D_1^\mathrm{top}(Q) \to D_2$. This is a pre-category (which is nothing but a graph, see e.g.\ \cite[Ch.\,II.7]{MacLane-book}), and the aim of this section is to show that $Q$ turns the free category (with conjugates) generated by this pre-category into a 2-category. This 2-category can be thought of as capturing some of the genus-0 information of the field theory $Q$. Our conventions for bicategories are collected in appendix \ref{app:bicategories}.
 
Recall (e.g.\ from \cite[Ch.\,II.7]{MacLane-book}) that the free category is generated by tuples of composable arrows. By the free category with conjugates we mean the category whose objects are $D_2$ and whose morphisms $a \to b$ (for $a,b \in D_2$) are tuples
\be
  \dwt x \equiv ( (x_1,\eps_1),\dots,(x_n,\eps_n))
\ee
where $x_i \in D_1^\mathrm{top}$, $\eps_i \in \{\pm\}$. As for cyclically composable domain walls, if all signs $\eps_i=+$, then we require $s(x_i) = t(x_{i+1})$ and $s(x_n) = a$, $t(x_1) = b$. If some $\eps_i=-$, the role of $s$ and $t$ changes as in figure \ref{fig:cycl-comp-example}. Composition of $\dwt x : a \to b$ and $\dwt y : b \to c$ is by concatenation,
\be
  \dwt y \circ \dwt x = ((y_1,\nu_1),\dots,(y_m,\nu_n),(x_1,\eps_1),\dots,(x_n,\eps_n)) \ .
\ee
Let us denote this category by $\bfD\equiv \bfD[D_2,D_1^\mathrm{top}]$. Morphism spaces $a \to b$ are written as $\bfD(a,b)$. The conjugation is the involution $* : \bfD(a,b) \to \bfD(b,a)$ given by
\be
  ( (x_1,\eps_1),\dots,(x_n,\eps_n))^* = ( (x_n,-\eps_n),\dots,(x_1,-\eps_1)) \ .
\ee
Note that endomorphisms $\dwt x : a \to a$ in $\bfD$ are precisely the tuples of cyclically composable domain walls. In particular, for any $\dwt x, \dwt y : a \to b$, the morphism $\dwt y \circ \dwt x^*$ is cyclically composable. 

\medskip

To define the 2-category structure on domain walls, we need the notion of translation and scale invariant families of states. Their definition will take us a few paragraphs. 

We will only need to know $Q$ on a subset of bordisms, each of which consists of a single disc in $\Rb^2$ from which a number of smaller discs have been removed (if there were no domain walls, these bordisms would form the little discs operad, see e.g. \cite{May-book}). The metric on the bordism is the one induced by $\Rb^2$ and the boundaries are parametrised by linear maps that are a combination of a translation and a scale transformation $x \mapsto r x + v$, where $x,v \in \Rb^2$ and $r \in \Rb_{>0}$. 

The objects which serve as source and target of these bordisms are described as follows. Let $\dwt x$ be cyclically composable and denote by $O(\dwt x;r)$ an object in the bordism category consisting of a single $S^1$ with 0-dimensional submanifold given by $n$ points not containing $-1 \in S^1$. These are clockwise cyclically labelled $x_1,\dots,x_n$, such that $x_1$ labels the first point in clockwise direction from $-1 \in S^1$. The collar around $S^1$ is obtained by taking concentric copies to fill a small neighbourhood. The conformal factor defining the metric on the collar is $e^\sigma = r^2$, so that the parametrising map $x \mapsto r x + v$, which takes the $S^1$ (with radius 1) to a circle of radius $r$ is an isometry. As all $x_i$ are in $D_1^\mathrm{top}$, by definition the vector space $Q(O(\dwt x;r))$ does not depend on the precise position of the $n$ marked points on $S^1$, as long as they are in the prescribed ordering. However, the vector space $Q(O(\dwt x;r))$ may still depend on $r$.\footnote{
In many examples (but not always), the spaces $Q(O(\dwt x;r))$ for different values of $r$ are isomorphic, with a preferred isomorphism given by evaluating $Q$ on an annulus with the two radii. But even in this case we do not demand that one passes to a formulation of the theory where these state spaces are actually {\em equal}.}
  
\begin{figure}[tb] 
$$
\psi_{\dwt x;R} = 
Q\hspace{-1pt}\Bigg(\hspace{-1pt}
\raisebox{-33pt}{\begin{picture}(71,71)
  \put(0,0){\scalebox{.60}{\includegraphics{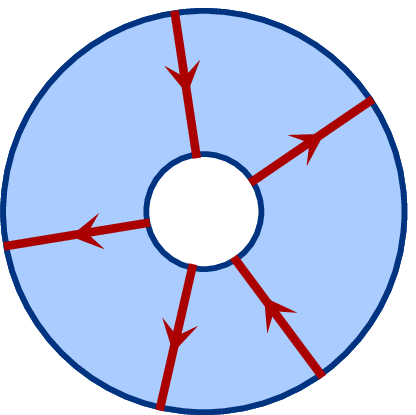}}}
  \put(0,0){
     \setlength{\unitlength}{.60pt}\put(-17,-15){
  }}
\end{picture}}
\hspace{-1pt}\Bigg)\hspace{-1pt}(\psi_{\dwt x;r_1})
=
Q\hspace{-1pt}\Bigg(\hspace{-1pt}
\raisebox{-33pt}{\begin{picture}(71,71)
  \put(0,0){\scalebox{.60}{\includegraphics{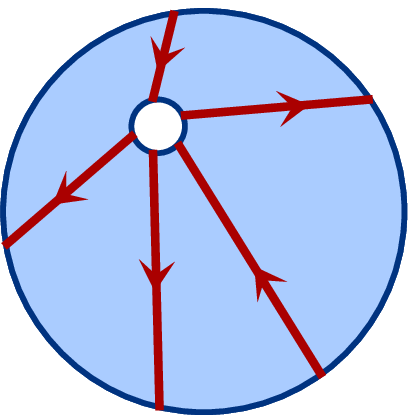}}}
  \put(0,0){
     \setlength{\unitlength}{.60pt}\put(-17,-15){
  }}
\end{picture}}
\hspace{-1pt}\Bigg)\hspace{-1pt}(\psi_{\dwt x;r_2})
=
Q\hspace{-1pt}\Bigg(\hspace{-1pt}
\raisebox{-33pt}{\begin{picture}(71,71)
  \put(0,0){\scalebox{.60}{\includegraphics{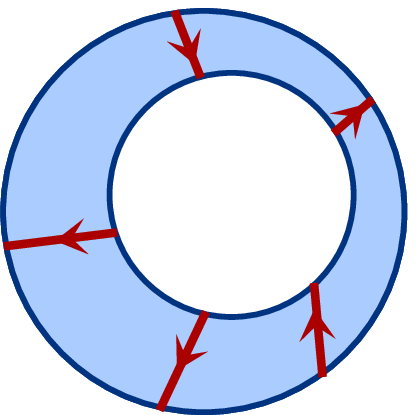}}}
  \put(0,0){
     \setlength{\unitlength}{.60pt}\put(-17,-15){
  }}
\end{picture}}
\hspace{-1pt}\Bigg)\hspace{-1pt}(\psi_{\dwt x;r_3})
$$
\vspace*{-1em}
\caption{Illustration of the condition for scale and translation invariant family of states: Let $\psi_{\dwt x}$ be such a family. The figure shows $Q$ applied to three annuli, understood as bordisms $O(\dwt x,r_i) \to O(\dwt x,R)$, for $i=1,2,3$, where $r_i$ denotes the radius of the inner disc of the $i$'th annulus shown above. All three annuli have the same outer radius $R$. Applying $Q$ to the bordism and evaluating the resulting linear map on $\psi_{\dwt x,r_i} \in Q(O(\dwt x,r_i))$, $i=1,2,3$, results always in the vector $\psi_{\dwt x;R} \in Q(O(\dwt x,R))$ of the same family.}
\label{fig:scale-trans-inv}
\end{figure}

Let $D(R;r,v) : O(\dwt x;r) \to O(\dwt x;R)$ be the bordism given by a disc of radius $R$ in $\Rb^2$ centred at the origin, from which a smaller disc of radius $r$ and centre $v$ has been removed. The domain walls are straight lines and the boundary parametrisation is given by scaling and translation as above (see figure \ref{fig:scale-trans-inv}). A {\em scale and translation invariant family} $\psi_{\dwt x}$ is a family of vectors $\{ \psi_{\dwt x;r} \}_{r \in \Rb_{>0}}$ with $\psi_{\dwt x;r} \in Q(O(\dwt x;r))$ such that
\be \label{eq:scale-trans-inv-def}
  \psi_{\dwt x;R} = Q\big(D(R;r,v)\big)( \psi_{\dwt x;r} ) 
  \quad \text{ for all } r,R>0 , \, v \in \Rb^2 \text{ with } r + |v| < R \ .
\ee
This condition is illustrated in figure \ref{fig:scale-trans-inv}. The {\em space of scale and translation invariant families},
\be
  H^\mathrm{inv}(\dwt x) \ ,
\ee   
is defined to be the vector space of all scale and translation invariant families $\psi_{\dwt x} \equiv \{ \psi_{\dwt x;r} \}_{r \in \Rb_{>0}}$ for fixed $\dwt x$. The space $H^\mathrm{inv}(\dwt x)$ may be zero-dimensional.

Scale and translation invariant families have the following important property: all amplitudes $Q(M)$ -- with $M$ a disc in $\Rb^2$ with smaller discs removed --  are independent of the position and size of an in-going boundary circle $O(\dwt x;r)$ for which $\psi_{\dwt x;r}$ is inserted as the corresponding argument. This can be seen by using functoriality of $Q$ to cut out a disc $D(R;r,v)$ from $M$ containing such an in-going boundary, then moving the in-going boundary circle using the defining property \eqref{eq:scale-trans-inv-def}, and finally gluing the resulting disc $D(R;r',v')$ back.

\medskip

\begin{figure}[tb] 
\begin{center}
\raisebox{50pt}{a)} \hspace{.5em}
\raisebox{-64pt}{\begin{picture}(60,135)
  \put(0,37){\scalebox{.80}{\includegraphics{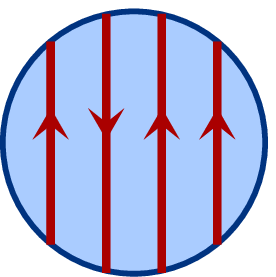}}}
  \put(0,37){
     \setlength{\unitlength}{.80pt}\put(-262,-377){
     \put(265,414)   {\scriptsize $b$ }
     \put(331,414)   {\scriptsize $a$ }
     \put(250,378)   {\scriptsize $(x_1,\!-)$ }
     \put(266,366)   {\scriptsize $(x_2,\!+)$ }
     \put(303,366)   {\scriptsize $(x_3,\!-)$ }
     \put(320,378)   {\scriptsize $(x_4,\!-)$ }     
     \put(250,450)   {\scriptsize $(x_1,\!+)$ }
     \put(266,460)   {\scriptsize $(x_2,\!-)$ }
     \put(303,460)   {\scriptsize $(x_3,\!+)$ }
     \put(320,450)   {\scriptsize $(x_4,\!+)$ }
  }}
\end{picture}}
\hspace{4em} \raisebox{50pt}{b)} \hspace{.0em}
\raisebox{-64pt}{\begin{picture}(87,135)
  \put(0,10){\scalebox{.50}{\includegraphics{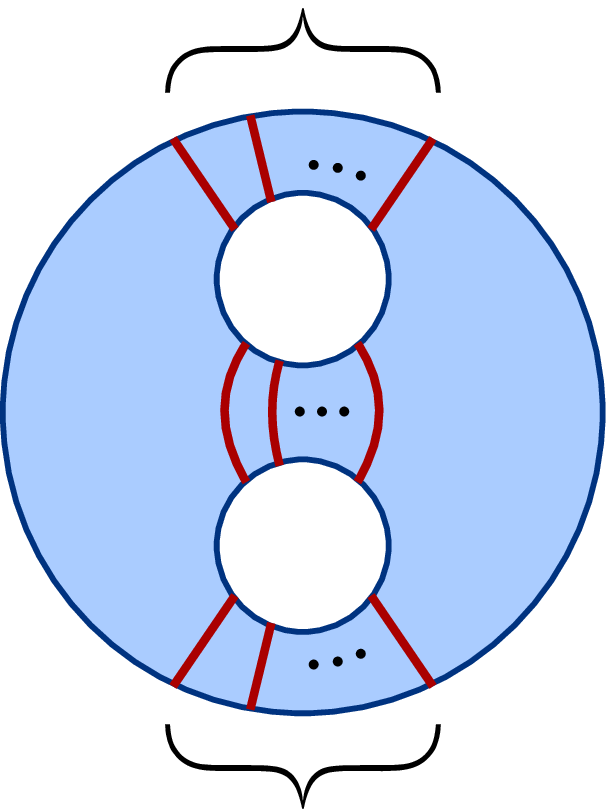}}}
  \put(0,10){
     \setlength{\unitlength}{.50pt}\put(-213,-300){
     \put(233,408)   {\scriptsize $b$ }
     \put(358,408)   {\scriptsize $a$ }
     \put(295,372)   {\scriptsize $u$ }
     \put(295,449)   {\scriptsize $v$ }
     \put(296,288)   {\scriptsize $\dwt x^*$ }
     \put(296,539)   {\scriptsize $\dwt z$ }
     \put(255,352)   {\scriptsize $x_1$ }
     \put(332,352)   {\scriptsize $x_p$ }
     \put(260,410)   {\scriptsize $y_1$ }
     \put(328,412)   {\scriptsize $y_q$ }
     \put(255,472)   {\scriptsize $z_1$ }
     \put(332,473)   {\scriptsize $z_r$ }
  }}
\end{picture}}
\hspace{4em} \raisebox{50pt}{c)} \hspace{.5em}
\raisebox{-64pt}{\begin{picture}(90,135)
  \put(-5,10){\scalebox{.50}{\includegraphics{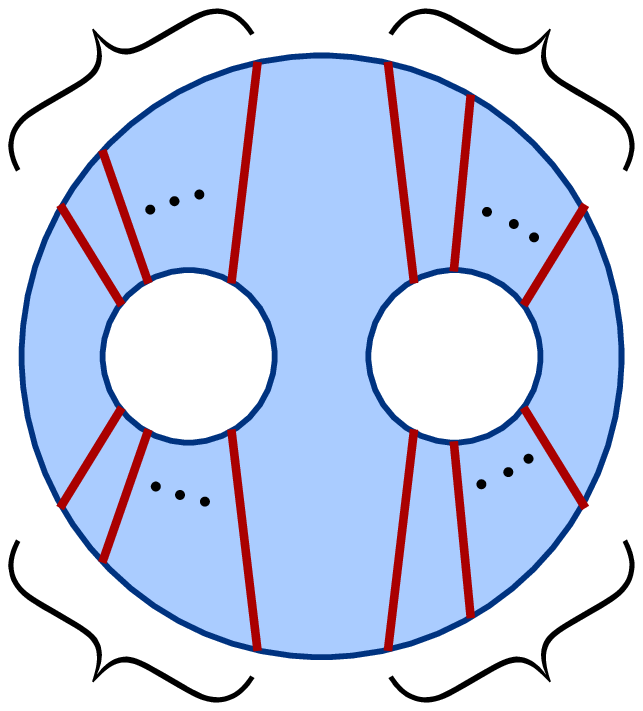}}}
  \put(-5,10){
     \setlength{\unitlength}{.50pt}\put(-200,-301){
     \put(217,410)   {\scriptsize $c$ }
     \put(297,410)   {\scriptsize $b$ }
     \put(374,410)   {\scriptsize $a$ }
     \put(363,307)   {\scriptsize $\dwt x^*$ }
     \put(227,307)   {\scriptsize $\dwt y^*$ }
     \put(363,516)   {\scriptsize $\dwt x'$ }
     \put(227,516)   {\scriptsize $\dwt y'$ }
     \put(307,375)   {\scriptsize $x_1$ }
     \put(366,392)   {\scriptsize $x_p$ }
     \put(281,375)   {\scriptsize $y_q$ }
     \put(220,392)   {\scriptsize $y_1$ }
     \put(307,452)   {\scriptsize $x_1'$ }
     \put(366,429)   {\scriptsize $x_r'$ }
     \put(281,452)   {\scriptsize $y_s'$ }
     \put(220,429)   {\scriptsize $y_1'$ }
     \put(256,411)   {\scriptsize $q$ }
     \put(335,411)   {\scriptsize $p$ }
  }}
\end{picture}}
\end{center}
\vspace*{-1.5em}
\caption{Bordisms defining the structure maps for the 2-category. a) the identity on a 1-morphism $\dwt x: a \to b$; b) vertical composition of 2-morphisms $u \in \bfD_2(\dwt x,\dwt y)$ and $v \in \bfD_2(\dwt y,\dwt z)$, where $\dwt x, \dwt y, \dwt z \in \bfD(a,b)$ (drawing the vector inside the cut-out disc means that this vector is to be used as the corresponding argument after applying $Q$); c) horizontal composition of 2-morphisms $p \in \bfD_2(\dwt x,\dwt x')$ and $q \in \bfD_2(\dwt y,\dwt y')$, where $\dwt x,\dwt x' : a \to b$ and $\dwt y,\dwt y' : b \to c$.}
\label{fig:id-hor-vert-comp}
\end{figure}

Given $\dwt x, \dwt y : a \to b$, we define the space of 2-morphisms from $\dwt x$ to $\dwt y$ to be
\be \label{eq:defect-2-cat-2morphs}
  \bfD_2(\dwt x,\dwt y) := H^\mathrm{inv}(\dwt y \circ \dwt x^*) \ .
\ee
The identity 2-morphisms, and the horizontal and vertical composition are defined by the bordisms shown in figure \ref{fig:id-hor-vert-comp}. 
The identity 1-morphism $\one_a : a \to a$, for $a \in D_2$, is the empty tuple $\one_a = ()$.

\begin{rem} \label{rem:in-image-of-proj}
(i) In order to obtain {\em families} of states from the bordisms shown in figure \ref{fig:id-hor-vert-comp}, one uses that there is an $\Rb_{>0}$-action on metric bordisms given by rescaling the metric. For the disc shaped bordisms in $\Rb^2$ relevant here, this amounts to a rescaling by some $R>0$. Consider the bordism in figure \ref{fig:id-hor-vert-comp}\,b) as an example. Call this bordism $M$ and assume that the radius of its outer disc is 1. Thus
\be
  M ~:~ O(\dwt z \circ \dwt y^*;r_1) \,\sqcup\, O(\dwt y \circ \dwt x^*;r_2) ~\longrightarrow~ O(\dwt z \circ \dwt x^*;1) \ .
\ee  
For each $R>0$ this produces a bordism $RM : O(\dwt z \circ \dwt y^*;R r_1) \sqcup O(\dwt y \circ \dwt x^*;R r_2) \to O(\dwt z \circ \dwt x^*;R)$. Given two families $\phi_1 \in H^\mathrm{inv}(\dwt z \circ \dwt y^*)$ and $\phi_2 \in H^\mathrm{inv}(\dwt y \circ \dwt x^*)$, we obtain a family of vectors
\be
  \psi_R := Q(RM)\big(\,\phi_{1;Rr_1}\,,\,\phi_{2;Rr_2}\,\big) ~\in~ Q\big(O(\dwt z \circ \dwt x^*;R)\big) \ .
\ee
The family $\{\psi_R\}_{R \in \Rb_{>0}}$ is again scale and translation invariant. This follows by substituting into the defining property \eqref{eq:scale-trans-inv-def} and using that $\phi_1$ and $\phi_2$ are scale and translation invariant families.
\\[.3em]
(ii) For all $r>0$, $v \in \Rb^2$ with $r+|v|<1$, the bordism $D(1;r,v) : O(\dwt x;r) \to O(\dwt x;1)$ induces the identity map on $H^\mathrm{inv}(\dwt x)$. This is just a reformulation of the defining property \eqref{eq:scale-trans-inv-def} using the prescription in (i). In other words, cylinders give the identity map on $H^\mathrm{inv}$, not just idempotents. This is important when verifying that the bordism in figure \ref{fig:id-hor-vert-comp}\,a) is indeed the unit for the vertical composition in figure \ref{fig:id-hor-vert-comp}\,b).
\end{rem}

The fact that we are working with topological domain walls and with scale and translation invariant families of states ensures that the properties of a 2-category are satisfied (as composition of 1-morphisms is strictly associative, and the unit 1-morphisms are strict, we indeed have a 2-category and not only a bicategory).
Let us denote this 2-category as 
\be \label{eq:def-defect-2cat}
  \bfD[Q]\equiv \bfD[D_2,D_1^\mathrm{top};Q] \ . 
\ee  
As was to be expected, moving one dimension up from the example in section \ref{sec:1-dim-tft-def} also increased the categorial level: in this construction the pre-category $D_1^\mathrm{top} \rightrightarrows D_2$ gets extended to a 2-category. 

\begin{figure}[tb] 
\begin{center}
\raisebox{55pt}{a)} \hspace{-1.5em}
\raisebox{-40pt}{\begin{picture}(85,88)
  \put(0,5){\scalebox{.50}{\includegraphics{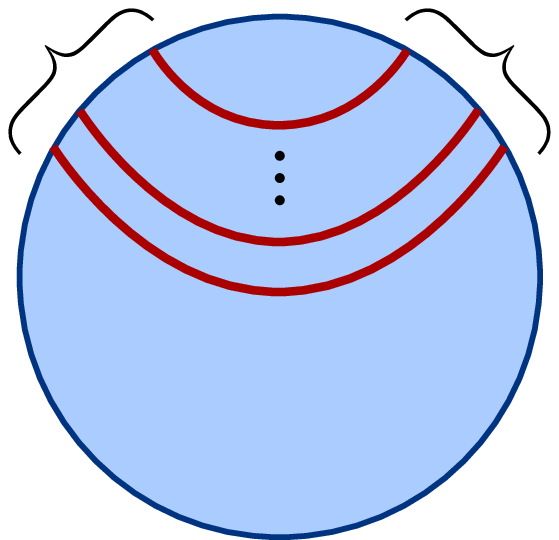}}}
  \put(0,5){
     \setlength{\unitlength}{.50pt}\put(-213,-332){
     \put(236,418)   {\scriptsize $x_1$ }
     \put(258,443)   {\scriptsize $x_2$ }
     \put(278,465)   {\scriptsize $x_m$ }
     \put(220,479)   {\scriptsize $\dwt x$ }
     \put(370,479)   {\scriptsize $\dwt x^*$ }
     \put(307,470)   {\scriptsize $a$ }
     \put(307,366)   {\scriptsize $b$ }
  }}
\end{picture}}
\hspace{2em} 
\raisebox{55pt}{b)} \hspace{-1.5em}
\raisebox{-40pt}{\begin{picture}(85,88)
  \put(0,0){\scalebox{.50}{\includegraphics{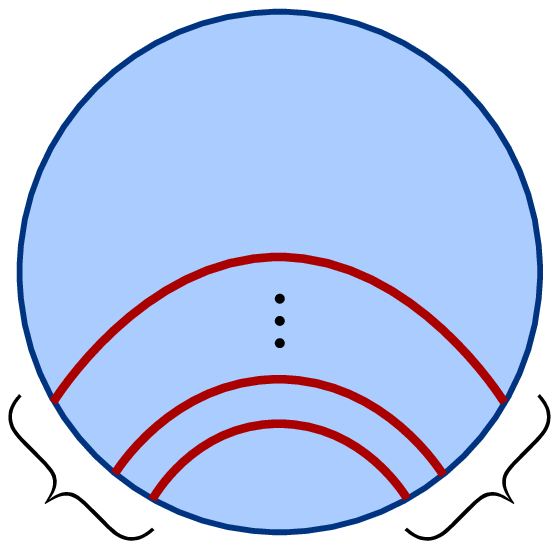}}}
  \put(0,0){
     \setlength{\unitlength}{.50pt}\put(-213,-332){
     \put(316,355)   {\scriptsize $x_1$ }
     \put(339,375)   {\scriptsize $x_2$ }
     \put(354,401)   {\scriptsize $x_m$ }
     \put(222,340)   {\scriptsize $\dwt x$ }
     \put(367,340)   {\scriptsize $\dwt x^*$ }
     \put(288,350)   {\scriptsize $b$ }
     \put(288,456)   {\scriptsize $a$ }
  }}
\end{picture}}
\hspace{2em} 
\raisebox{55pt}{c)} \hspace{-1.5em}
\raisebox{-40pt}{\begin{picture}(85,88)
  \put(0,5){\scalebox{.50}{\includegraphics{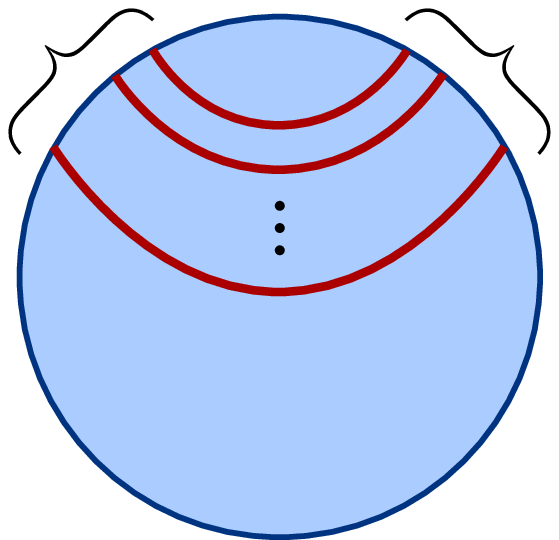}}}
  \put(0,5){
     \setlength{\unitlength}{.50pt}\put(-213,-332){
     \put(311,467)   {\scriptsize $x_1$ }
     \put(335,440)   {\scriptsize $x_2$ }
     \put(349,413)   {\scriptsize $x_m$ }
     \put(220,479)   {\scriptsize $\dwt x^*$ }
     \put(370,479)   {\scriptsize $\dwt x$ }
     \put(290,466)   {\scriptsize $b$ }
     \put(290,366)   {\scriptsize $a$ }
  }}
\end{picture}}
\hspace{2em} 
\raisebox{55pt}{d)} \hspace{-1.5em}
\raisebox{-40pt}{\begin{picture}(85,88)
  \put(0,0){\scalebox{.50}{\includegraphics{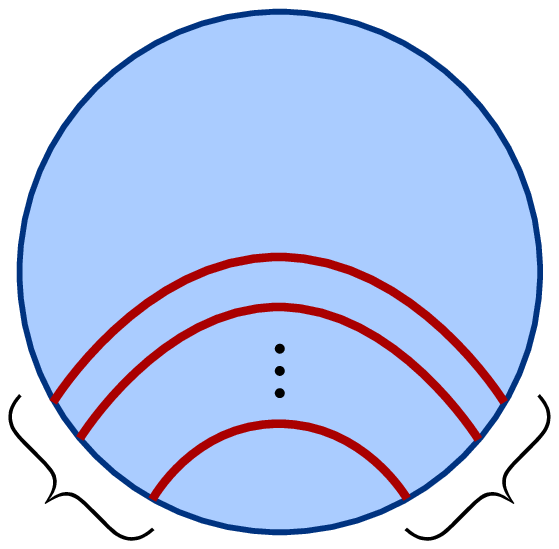}}}
  \put(0,0){
     \setlength{\unitlength}{.50pt}\put(-213,-332){
     \put(232,406)   {\scriptsize $x_1$ }
     \put(252,372)   {\scriptsize $x_2$ }
     \put(278,356)   {\scriptsize $x_m$ }
     \put(212,340)   {\scriptsize $\dwt x^*$ }
     \put(367,340)   {\scriptsize $\dwt x$ }
     \put(308,352)   {\scriptsize $a$ }
     \put(308,456)   {\scriptsize $b$ }
  }}
\end{picture}}
\end{center}
\vspace*{-1em}
\caption{Bordisms defining the adjunction maps in $\bfD[Q]$: The left and right adjoint of $\dwt x : a \to b$ is $\dwt x^* : b \to a$ and applying $Q$ to the bordisms shown gives the adjunction maps
a) $b_{\dwt x} : \one_b \to \dwt x \circ \dwt x^*$; 
b) $d_{\dwt x} : \dwt x^* \circ \dwt x \to \one_a$; 
c) $\tilde b_{\dwt x} : \one_a \to \dwt x^* \circ \dwt x$; 
d) $\tilde d_{\dwt x} : \dwt x \circ \dwt x^* \to \one_b$.}
\label{fig:duality-bordisms}
\end{figure}

\begin{rem} \label{rem:top-def-interesting-for-qft}
(i) Actually, $\bfD[Q]$ carries more structure. For example, each 1-morphism $\dwt x : a \to b$ has a left and a right adjoint, namely $\dwt x^* : b \to a$, together with adjunction maps as shown in figure \ref{fig:duality-bordisms} (see \cite[Sec.\,I.6]{Gray:1974} for more on adjunctions in bicategories). Such rigid and related structures on the category of defects were discussed already in \cite{Frohlich:2006ch,McNamee:2009a,Carqueville:2010hu}. Rigid and pivotal structures on the 2-category $\bfD[Q]$ were studied in the context of planar algebras\footnote{
  A planar algebra \cite{Jones:1999a} can be understood as a two-dimensional theory with exactly two world sheet phases $D_2 = \{a,b\}$, exactly one topological domain wall type $D_1 = \{a \xrightarrow{x} b\}$, and no junctions, $D_0= \emptyset$. Furthermore, the theory is only defined on genus zero surfaces with exactly one out-going boundary circle, that is, on discs with smaller discs removed, see figures \ref{fig:scale-trans-inv}--\ref{fig:duality-bordisms}.
} 
in \cite{Ghosh:2008a,Das:2010a}.
\\[.3em]
(ii) The 2-category $\bfD[Q]$ is an invariant attached to a quantum field theory with defects. Interestingly, even though only `topological data' enters its definition (topological domain walls and scale and translation invariant families of states), general quantum field theories do produce more general 2-categories $\bfD[Q]$ than topological field theories. In a nutshell, the reason is that the rigid structure mentioned in (i) will tend to produce integer quantum dimensions for topological field theories, while for example in rational conformal field theories non-integer quantum dimensions occur.\footnote{
A notion of 2d TFT with domain walls was also studied in \cite{Kodiyalam:2005a} in relation to subfactor planar algebras. Apart from there being exactly two world sheet phases and one type of domain wall -- as is usual in the planar algebra setting -- there is one important difference: in \cite{Kodiyalam:2005a} a bordism is in addition equipped with a decomposition into `genus 0 components', and bordisms with different such decompositions are considered distinct, unless they have a common refinement (see \cite[Def.\,2.7]{Kodiyalam:2005a}). This excludes for example the decomposition of a torus along different cycles to be used below (the `Cardy condition'). Thus the functors constructed in \cite{Kodiyalam:2005a} are in general not defect TFTs in our sense (cf.\ eqn.\,\eqref{eq:tft-with-defects-functor} below).
}

In slightly more detail, fix a topological field theory with defects and take $k=\Cb$. Consider the bordism $M : \emptyset \to \emptyset$ given by a torus (say $[0,1] \times [0,1]$ with opposite edges identified) with a single defect line labelled $x \in D_1$ wrapping a non-contractible cycle (say $[0,1] \times \{ \tfrac12 \}$). We label the unique connected domain of the bordism by $a \in D_2$, so that $x : a \to a$. Then $Q(M) : \Cb \to \Cb$ is just a number. This number can be computed in two ways. Let $M_| : O(x) \to O(x)$ be the annulus obtained by cutting $M$ along $\{0\} \times [0,1]$, i.e.\ by not identifying the vertical edges of $M$. Then $M_|$ is just the cylinder over $O(x)$. We will learn later in remark \ref{rem:scale-trans-inv-in-TFT}, that for a topological field theory the space of scale and translation invariant states $H^\mathrm{inv}(x)$ can be identified with the image of $Q(M_|)$ (which is an idempotent in topological field theory) in $Q(O(x))$. Thus, using also functoriality of $Q$,
\be \label{eq:torus-as-2-traces-1}
  Q(M) = \tr_{Q(O(x))} Q(M_|) = \tr_{H^\mathrm{inv}(x)} \id = \dim(H^\mathrm{inv}(x)) \in \Zb_{\ge 0} \ .
\ee
On the other hand, we can consider $M_- : O(a) \to O(a)$, which is obtained by cutting $M$ along $[0,1] \times \{0\}$, i.e.\ by not identifying the horizontal edges. The endomorphism $Q(M_-) : Q(O(a)) \to Q(O(a))$ is called {\em defect operator} for the defect $x : a \to a$, we will return to this briefly in section \ref{sec:T^cw-construct} below. By the same reasoning as above, $Q(M) = \tr_{Q(O(a))} Q(M_-)$. Let $C_{O(a)}$ be the cylinder over $O(a)$. By functoriality (and again only for topological field theory) we have $Q(M_-) = Q(C_{O(a)}) \circ Q(M_-) \circ Q(C_{O(a)})$, so that the image of $Q(M_-)$ lies in the image of the idempotent $Q(C_{O(a)})$ and $Q(M_-)$ acts trivially on the kernel of $Q(O(a))$. By restriction we obtain an endomorphism $Q(M_-) : H^\mathrm{inv}(\one_a) \to H^\mathrm{inv}(\one_a)$ and the trace can be computed in this restriction,
\be \label{eq:torus-as-2-traces-2}
  Q(M) = \tr_{Q(O(a))} Q(M_-) = \tr_{H^\mathrm{inv}(\one_a)} Q(M_-) \  .
\ee
The fact that the traces \eqref{eq:torus-as-2-traces-1} and \eqref{eq:torus-as-2-traces-2} agree is known as the Cardy condition (because of the paper \cite{Cardy:1989ir}) and was first investigated for topological defects in \cite{Petkova:2000ip} in the context of rational conformal field theory. 

In summary, for topological field theories, the trace of a defect operator over the space of scale and translation invariant states is equal to the dimension of a vector space, and is thus a non-negative integer. In rational conformal field theories with non-degenerate vacuum (this means that the space $H^\mathrm{inv}(\one_a)$ is one-dimensional), the defect operator acts by multiplying with a number -- the (left or right) quantum dimension of $\dwt x$ -- and the trace $\tr_{H^\mathrm{inv}(\one_a)} Q(M_-)$ is then equal to this number. In many examples, this quantum dimension is not an integer (and not even rational, though still algebraic). This is the case for the examples studied in \cite{Petkova:2000ip} and \cite{tft1,Frohlich:2006ch}.
\end{rem}

\section[A simple example: 2d lattice topological field theory]{A simple example: 2d lattice top.\ field theory}\label{sec:ex-latticeTFT}

It is difficult to find functors from $\bord_{2,1}^\mathrm{def}(D_2,D_1,D_0)$ to topological vector spaces which depend non-trivially on the metric. On the other hand, it is easy to construct examples where the functors are independent of the metric and boundary parametrisation. In this section we describe such a construction.

For the remainder of this section we fix a field $k$.

\subsection{Category of smooth bordisms with defects} \label{sec:smooth-bord-cat}

Instead of posing restrictions on the functor one can modify the bordism category accordingly. This leads us to define a symmetric monoidal category of {\em smooth bordisms with defects}, which we denote as
\be \label{eq:Bord21-def,top-symbol}
  \bord_{2,1}^\mathrm{def,top}(D_2,D_1,D_0) \ .
\ee
The modifications relative to the definition in section \ref{sec:2d-bord-def} are as follows. 
\begin{itemize}
\item {\em Objects:} The collar around an $S^1$ no longer carries a metric.
\item {\em Morphisms:} The manifold does not carry a metric and the parametrising maps are only required to be smooth (rather than isometric). Two bordisms are equivalent if there is a diffeomorphism between them which preserves orientation, decomposition, labelling, as well as the image of the point $-1 \in S^1$ in each connected component of the boundary (rather than commuting with the parametrising maps in some neighbourhood of the boundary).
\end{itemize}
The symmetric monoidal structure is as in section \ref{sec:2d-bord-def}. 

Note that this definition is different form the standard 2-bordism category for topological field theories even in the case without domain walls ($D_2 = \{*\}$ and $D_1 = D_0 = \emptyset$) because we have still added the identities (and $S_n$-action) by hand to the space of morphisms; in particular, the cylinder over a given object $U$ is {\em not} the identity morphism (but it is still an idempotent).

\medskip

Denote by $\vect_f(k)$ the symmetric monoidal category of finite-dimensional $k$-vector spaces. Fix furthermore sets $D_i$, $i=0,1,2$, of defect labels with maps as required. The aim of this section is to construct examples of symmetric monoidal functors
\be \label{eq:tft-with-defects-functor}
  T : \bord_{2,1}^\mathrm{def,top}(D_2,D_1,D_0) \longrightarrow \vect_f(k) \ .
\ee
We will do this via a lattice TFT construction which is a straightforward generalisation of the original lattice TFT without domain walls \cite{Bachas:1992cr,Fukuma:1993hy} and of the lattice construction of homotopy TFTs in \cite[Sec.\,7]{Turaev:1999yf}. A construction of field theory correlators on arbitrary world sheets in the presence of domain walls first appeared in \cite{tft1,tft4,Frohlich:2006ch}, where it was carried out in the context of two-dimensional rational conformal field theory. The construction for TFTs given in this section can be extracted from this in the special case that the modular category underlying the CFT is that of vector spaces. 

Note, however, that this will not give the most general such functor $T$ (as it does not even in the case without domain walls); a classification of functors \eqref{eq:tft-with-defects-functor} akin to the one in the situation without domain walls in \cite{Dijkgraaf:1989,Abrams:1996ty} is at present not known. A classification of 2d TFTs with defects that can be extended down to points has been reported in \cite{Schommer-Pries:2009}; related results are given in \cite[Ex.\,4.3.23]{Lurie:2009aa}. Two dimensional homotopy TFTs over $X=K(G,1)$, which in the present language correspond to defect TFTs with invertible defects labelled by a group $G$,\footnote{
  By a defect TFT with only invertible defects we mean the case that $D_2= \{*\}$, $D_0 = \emptyset$ and that $D_1=G$ is a group (more generally we can take $s,t : D_1 \to D_2$ to be a groupoid). We demand that the linear map assigned by $T$ in \eqref{eq:tft-with-defects-functor} to a bordism does not change if we replace two parallel defect lines with opposite orientation ``$\rightleftarrows$'' labelled by the same element of $D_1$ by the `reconnected' defect lines ``$\supset\subset$''. Given a bordism $M$ with defect lines labeled by group elements in $G$, we can construct a principal $G$-bundle on the bordism by taking the trivial $G$-bundle over each component of $M_2$ and choosing the transition functions across $M_1$ to be multiplication with the group element labelling that component of $M_1$. This  gives a functor from $\bord_{2,1}^\mathrm{def,top}(\{*\},G,\emptyset)$ to principal $G$-bundles with two-dimensional base.
}
have been classified in \cite[Thm.\,4.1]{Turaev:1999yf}. The classification in the case of simply connected $X$ is given in \cite[Thm.\,4.1]{Brightwell:1999a}.

\medskip

The construction works in two steps. In the first step, we introduce a larger category, $\bord_{2,1}^\mathrm{def,top,cw}$, where objects and morphisms are in addition endowed with a cell decomposition (section \ref{sec:bord-with-celldecomp}). It comes with a forgetful functor $F : \bord_{2,1}^\mathrm{def,top,cw} \to  \bord_{2,1}^\mathrm{def,top}$, which is surjective (not just essentially surjective) and full (but not faithful). We then construct a symmetric monoidal functor $\Tcw : \bord_{2,1}^\mathrm{def,top,cw} \to \vect_f(k)$ (section \ref{sec:T^cw-construct}). In the second step, we show that $\Tcw$ is independent of the cell-decomposition in the sense that there exists a symmetric monoidal functor $T$ which makes the diagram
\be \label{eq:Tcw-T-comm-diag}
\raisebox{2em}{\xymatrix{
\bord_{2,1}^\mathrm{def,top,cw} \ar[d]_F \ar[rr]^{\Tcw} && \vect_f(k) 
\\
\bord_{2,1}^\mathrm{def,top} \ar@{-->}[urr]_{\exists! \, T}
}}
\ee
commute on the nose (section \ref{sec:T^cw-cell-indep}). Since $F$ is surjective and full, this diagram defines $T$ uniquely.

\subsection{Category of bordisms with cell decomposition} \label{sec:bord-with-celldecomp}

We will use a class of cell decompositions introduced in \cite[Def.\,5.1]{Kirillov:2010a}, called PLCW decompositions there. These are less general than CW-complexes, which for example allow one to decompose $S^2$ into a 0-cell (a point on $S^2$) and 2-cell ($S^2$ minus that point). However, they are more general than regular CW-complexes, where one cannot identify different faces of one given cell, or even triangulations, where in addition each cell is a simplex. It is shown in \cite{Kirillov:2010a} that any two PLCW decompositions are related by a simple collection of local moves.

Given a compact $n$-dimensional manifold $M$, possibly with non-empty boundary, we will consider decompositions of $M$ into a finite number of mutually disjoint open $k$-cells, $k=0,\dots,n$, with the following properties. 
Let $B^k$ be the closed unit ball in $\Rb^k$ and let $\mathring B^k$ be its interior. For each $k$-cell $C$ there has to exist a continuous\footnote{
  To match \cite{Kirillov:2010a} we should work in the PL setting. We thus impose the additional condition
  that the manifold $M$ with the decomposition as defined is homeomorphic to a PLCW complex. This is
  automatically satisfied if the manifold and the cell maps are already PL.
}
map $\varphi : B^k \to M$ such that 
\\[-1.7em]
\begin{itemize}
\itemsep 0.1em
\parskip 0pt
\item[-] $C$ is the homeomorphic image of $\mathring B^k$,
\item[-] there exists a decomposition of the boundary $S^{k-1}$ of $B^k$ which gets mapped by $\varphi$ to the decomposition in $M$,
\item[-] $\varphi$ is a homeomorphism when restricted to the interior of each cell on $S^{k-1}$. 
\end{itemize}
For more details we refer to \cite{Kirillov:2010a}. We just note that this last condition excludes the decomposition of $S^2$ into a single 0-cell and a single 2-cell mentioned above.

By abuse of terminology, we will refer to a decomposition of $M$ as just described simply as a {\em cell decomposition} $C(M)$. The set of $i$-dimensional cells ($i$-cells for short) will be called $C_i(M)$. 

\medskip
 
\begin{figure}[tb] 
\begin{center}
\raisebox{30pt}{a)} \hspace{.5em}
\raisebox{-38pt}{\begin{picture}(70,80)
  \put(0,0){\scalebox{.8}{\includegraphics{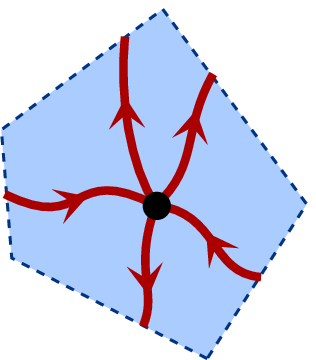}}}
\end{picture}}
\hspace{5em}
\raisebox{30pt}{b)} \hspace{.5em}
\raisebox{-38pt}{\begin{picture}(70,80)
  \put(0,0){\scalebox{.8}{\includegraphics{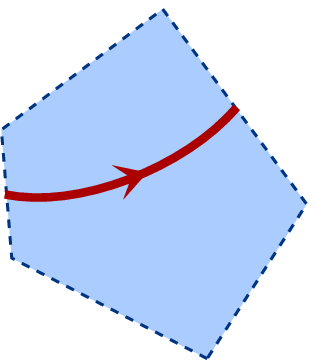}}}
\end{picture}}
\end{center}
\vspace*{-1em}
\small
\caption{Allowed configurations of 2-cells and defects: (a) 2-cell containing a point from $M_0$; here only a star-shaped pattern of domain walls is allowed and each edge has to be crossed by precisely one domain wall. (b) 2-cell containing no point from $M_0$ but part of $M_1$; here only one segment of $M_1$ may lie in the 2-cell, and it must enter and leave the 2-cell via distinct edges.}
\label{fig:allowed-cells}
\end{figure}

The category $\bord_{2,1}^\mathrm{def,top,cw}(D_2,D_1,D_0)$ is the same as $\bord_{2,1}^\mathrm{def,top}(D_2,D_1,D_0)$, except that objects and morphisms are equipped with the following extra structure:
\begin{itemize}
\item {\em Objects:} The 1-dimensional part of an object $U$ (the disjoint union of circles, not their collars) are equipped with a  cell decomposition $C(U)$ such that each point of the set $U_0$ of marked points lies in a $1$-cell (recall that a $1$-cell is homeomorphic to an open interval), and such that each $1$-cell contains at most one point of $U_0$.
\item {\em Morphisms:} By a {\em bordism with cell decomposition} we mean a bordism $M = M_2 \cup M_1 \cup M_0$ in $\bord_{2,1}^\mathrm{def,top}$, which is equipped with a cell decomposition $C(M)$ such that
\begin{itemize}
\item the 1-dimensional submanifold $M_1$ only intersects 1-cells and 2-cells, but not 0-cells. Each 1-cell intersects $M_1$ in at most one point.
\item each point of $M_0$ lies in a 2-cell, and each 2-cell contains at most one point of $M_0$. A 2-cell containing a point of $M_0$ must be homeomorphic to one of the type shown in figure \ref{fig:allowed-cells}\,a), i.e.\ it may only contain a `star-shaped' configuration of domain walls, such that each edge of this 2-cell is traversed by exactly one domain wall.
\item a 2-cell containing no point from $M_0$ but a part of $M_1$ must be homeomorphic to one of the type shown in figure \ref{fig:allowed-cells}\,b), i.e.\ its intersection with $M_1$ is a single open interval.
\end{itemize}
If $M$ has non-empty boundary $\partial M$, we demand that $\partial M$ is a union of 0-cells and 1-cells (and hence does not intersect 2-cells). Furthermore, the decomposition of $\partial M$ has to be the image of the  cell decomposition of the source and target object under the parametrising maps. 

Morphisms are equivalence classes of bordisms with cell decomposition.
Two bordisms with cell decomposition are equivalent if their underlying bordisms are equivalent in $\bord_{2,1}^\mathrm{def,top}$ and if the two cell decompositions are related by an isotopy (which at each instance has to give a bordism with cell decomposition).
\end{itemize}
All cells in $C(U)$ and $C(M)$ are a priori unoriented. However, the 2-cells of $C(M)$ have an orientation induced by that of $M$.

\subsection{Algebraic preliminaries} \label{sec:alg-prel}

By an {\em algebra} we shall always mean a unital, associative algebra over $k$. The {\em centre} $Z(A)$ of an algebra $A$ is the (commutative, unital) subalgebra
\be
  Z(A) = \big\{\,z \in A \,\big|\, za = az \text{ for all } a \in A \,\big\} \ . 
\ee
Given a right $A$-module $M$ and a left $A$-module $N$, the tensor product $M \otimes_A N$ is defined as the cokernel 
\be
  M \otimes A \otimes N \xrightarrow{l-r} M \otimes N \xrightarrow{\pi_\otimes} M \otimes_A N \ ,
\ee  
where 
$l(m \otimes a \otimes n) = m \otimes (a.n)$ 
and 
$r(m \otimes a \otimes n) = (m.a) \otimes n$ denote the left and right action.
Given an $A$-$A$-bimodule $X$, we shall also require the `cyclic tensor product' which identifies the left and right action of $A$ on $X$. We denote this tensor product by $\cytens A X$; it is defined as the cokernel
\be
  A \otimes X \xrightarrow{l-r} X \xrightarrow{\pi_\otimes} \, \cytens A X \ ,
\ee
where $l(a \otimes x) = a.x$ and $r(a \otimes x) = x.a$. Note that $\cytens A X$ is in general no longer an $A$-$A$-bimodule; it does, however, carry a coinciding left and right action of $Z(A)$. Denote by $[A,A]$ the linear subspace of $A$ (not the ideal) generated by all elements of the form $ab-ba$. Then by definition $\cytens A A = A/[A,A]$.\footnote{
  The quotient $A/[A,A]$ is not necessarily isomorphic to $Z(A)$ -- just take $A$ to be the 3-dimensional algebra of upper triangular $2{\times}2$-matrices in which case $Z(A) \cong k$ while $A / [A,A] \cong k \oplus k$.}

A {\em Frobenius algebra} is a finite-dimensional algebra $A$ together with a linear map $\eps : A \to k$, called the {\em counit}, such that the bilinear pairing 
\be\label{eq:pairing-via-eps}
  \langle a,b \rangle := \eps(ab) 
\ee
on $A \times A$ is non-degenerate. We denote by $\beta : k \to A \otimes A$ the unique linear map such that $( (\eps \circ m) \otimes \id_A) \circ (\id \otimes \beta) = \id_A$. In other words, if $a_i$ is a basis of $A$ and $a_i'$ the dual basis in the sense that $\langle a_i,  a_j'\rangle = \delta_{i,j}$, then $\beta(1) = \sum_i a_i' \otimes a_i$. The defining property can be written as
\be \label{eq:Frob-tracepair-completeness}
  \sum_i \langle x, a_i' \rangle \,a_i = x 
  \quad \text{for all $x \in A$} \ .
\ee
By associativity of $A$, the pairing \eqref{eq:pairing-via-eps} is always {\em invariant}, i.e.\ $\langle a,bc \rangle = \langle ab,c \rangle$.

For $a \in A$ let $L_a : A \to A$ be the left multiplication by $a$, $L_a(b) = ab$. We say that $A$ is a {\em Frobenius algebra with trace pairing} if it is a Frobenius algebra whose counit is given by\footnote{
  Let $R_a(b) = ba$ denote the right multiplication by $a$ and set $\eps(a) = \tr_A(L_a)$, $\eps'(a) = \tr_A(R_a)$. Then $(A,\eps)$ is a Frobenius algebra if and only if $(A,\eps')$ is a Frobenius algebra and in this case $\eps(a) = \eps'(a)$. For a proof see e.g.\ \cite[Lem.\,3.9]{tft1} in the special case that the tensor category is $\vect_f(k)$. Note that if $(A,\eps)$ and $(A,\eps')$ are not Frobenius, the statement may be false; for example, if $A$ is the 3-dimensional algebra of upper triangular $2{\times}2$-matrices and $a = \big(\begin{smallmatrix} 0&0\\0&1 \end{smallmatrix}\big)$, then $\tr_A(L_a) = 1$ and $\tr_A(R_a) = 2$.
} 
$\eps(a) = \tr_A(L_a)$.  
Note that `Frobenius'  is extra structure on an algebra, while `Frobenius with trace pairing' is a property of an algebra.

In the following we list some of the special properties of Frobenius algebras with trace pairing. Firstly,
it is automatically {\em symmetric}: $\langle a,b \rangle = \langle b,a \rangle$. Consequently also $\beta$ is symmetric: $\beta(1) = \sum_i a_i' \otimes a_i =  \sum_i a_i \otimes a_i'$. Next, by the definition of the counit we have the following identity:
$\langle x, 1 \rangle = \eps(x) = \tr_A(L_x) = \sum_i a_i^*(x a_i) = \sum_i \langle x a_i, a_i' \rangle =  \langle x , \sum_i a_i  a_i' \rangle$. Since this holds for all $x \in A$, we conclude
\be \label{eq:beta-special-property}
  \sum_i a_i a_i' = 1 \ .
\ee
For all $a,b \in A$ we have
\be \label{eq:beta-move-factors}
  \sum_i (a a_i' b) \otimes a_i = \sum_i a_i' \otimes (b a_i a) 
  \quad \in A \otimes A \ ; 
\ee
this can be proved by pairing both sides with an arbitrary $x$. On the left hand side, pairing with $x$ produces 
$\sum_i \langle x, a a_i' b\rangle a_i = \sum_i \langle b x a,  a_i' \rangle a_i = b x a$ by \eqref{eq:Frob-tracepair-completeness}. The right hand side gives the same result.

If we interpret $\beta(1)$ as an element of the algebra $A^\mathrm{op} \otimes A$, the two equations above lead to the following result:

\begin{lem} \label{lem:beta(1)-prop}
In the algebra $A^\mathrm{op} \otimes A$ we have
\\[.3em]
(i) 
$\beta(1)\cdot (a \otimes 1) = \beta(1)\cdot (1 \otimes a)$
and 
$ (a \otimes 1) \cdot \beta(1) =  (1 \otimes a) \cdot \beta(1)$,
\\[.3em]
(ii) $\beta(1)\cdot \beta(1) = \beta(1)$.
\end{lem}

\begin{proof}
By definition, the product of $A^\mathrm{op} \otimes A$ is $(a \otimes a') \cdot (b \otimes b') = (ba) \otimes (a'b')$. Statement (i) is nothing but \eqref{eq:beta-move-factors}, while (ii) follows from (i) and \eqref{eq:beta-special-property} via 
$
  \beta(1) \cdot \beta(1) 
  = \sum_i \beta(1) \cdot (a_i' \otimes 1) \cdot (1 \otimes a_i)
  = \sum_i \beta(1) \cdot (1 \otimes a_i') \cdot (1 \otimes a_i)
  = \sum_i \beta(1) \cdot (1 \otimes (a_i' a_i))
  = \beta(1) \cdot (1 \otimes 1)
$.
\end{proof}

Consider a right $A$-module $M$ and a left $A$-module $N$ as above. For Frobenius algebras with trace pairing, the tensor product $M \otimes_A N$ can be canonically identified with a subspace of $M \otimes N$ as follows. The tensor product $M \otimes N$ carries an $A^\mathrm{op} \otimes A$-action. Define the linear map $p_\otimes : M \otimes N \to M \otimes N$ to be the action of $\beta(1)$ on $M\otimes N$:
\be
  p_\otimes(m \otimes n) = \beta(1).(m \otimes n) = \sum_i (m.a_i') \otimes (a_i.n) \ .
\ee

\begin{lem}
Let $A$ be a Frobenius algebra with trace pairing. 
\\
(i) $p_\otimes(m.a \otimes n) = p_\otimes(m \otimes a.n)$ holds for all $a \in A$, $m \in M$, $n \in N$.
\\
(ii) $p_\otimes$ is idempotent.
\\
(iii) $\mathrm{im}(p_\otimes) = M \otimes_A N$.
\end{lem}

\begin{proof}
Parts (i) and (ii) are immediate consequences of lemma \ref{lem:beta(1)-prop}.
For part (iii), denote by
\be
  e_\otimes : \mathrm{im}(p_\otimes) \longrightarrow M \otimes N
  \quad , \quad
  \pi_\otimes : M \otimes N \longrightarrow \mathrm{im}(p_\otimes) \ ,
\ee
the embedding of the image of $p_\otimes$ into $M \otimes N$ and the projection from $M \otimes N$ to the image. They satisfy
\be
  e_\otimes \circ \pi_\otimes = p_\otimes
  \quad  ,  \quad
  \pi_\otimes \circ e_\otimes = \id_{\mathrm{im}(p_\otimes)} \ .
\ee
Consider the diagram
\be \label{eq:ptens-prop-aux1}
\raisebox{2em}{\xymatrix{
M \otimes A \otimes N \ar[r]^(.58){l-r} & 
  M \otimes N \ar[r]^{\pi_\otimes} \ar[d]^f & 
  \mathrm{im}(p_\otimes) \ar@{-->}[dl]^{f \circ e_\otimes}
\\
& V
}}
\qquad .
\ee
Here $V$ is a vector space and $f$ satisfies $f \circ (l-r) = 0$. By part (i) the map $\pi_\otimes$ satisfies $\pi_\otimes \circ (l-r)= 0$. From $f(m.a \otimes n) = f(m \otimes a.n)$ it is easy to see that $f \circ p_\otimes = f$, and so $f = (f \circ e_\otimes) \circ \pi_\otimes$. Thus the above diagram commutes and we have verified the universal property of the cokernel.
\end{proof}

The construction of $p_\otimes$, $e_\otimes$, $\pi_\otimes$ works similarly for the $\cytens A$ tensor product of an $A$-$A$-bimodule $X$. In this case, 
\be \label{eq:p-otimes-cytens}
  p_\otimes(x) = \sum_i a_i.x.a_i' \ . 
\ee
We use the same notation $e_\otimes : \cytens A X \to X$ and $\pi_\otimes : X \to  \cytens A X$ as above. For multiple tensor products, the idempotents can be combined. For the state spaces of the TFT with defects to be constructed below, the maps
\be \label{eq:er-def-multiple-tensor}
\raisebox{0em}{\xymatrix{
  X_1 \otimes X_2 \otimes \cdots \otimes X_n  \ar@<1ex>[rr]^(.4){\pi_\otimes} &&
  \cytens{A_{n,1}} X_1 \otimes_{A_{1,2}} X_2 \otimes_{A_{2,3}} \cdots \otimes_{A_{n-1,n}} X_n
  \ar@<1ex>[ll]^(.6){e_\otimes}
}}
\ee
will be useful. Here the $X_i$ are bimodules with left/right actions of algebras as indicated. 
As above, the maps $e_\otimes$ and $\pi_\otimes$ satisfy $\pi_\otimes \circ e_\otimes = \id$ and $e_\otimes \circ \pi_\otimes = p_\otimes$. The projector $p_\otimes$ in this case is
\be
  p_\otimes(x_1 \otimes x_2 \otimes \cdots \otimes x_n) 
  = \sum_{i_1,i_2,\dots,i_n}
  (a_{i_1}.x_1.a_{i_2}') \otimes (a_{i_2}.x_2.a_{i_3}') \otimes \cdots \otimes (a_{i_{n}}.x_n.a_{i_1}')
   \ ,
\ee
where the $a_{i_k}$ are bases of the corresponding algebras.

\begin{lem} \label{lem:ho-hom=ho-cohom}
Let $A$ be a Frobenius algebra with trace pairing. Then 
$$
  Z(A) \,\cong~ \cytens A A \,=\, A / [A,A] \ .
$$
\end{lem}

\begin{proof}
Here $p_\otimes : A \to A$ is given by $p_\otimes(x) = \sum_i a_i x a_i'$. By the same reasoning as in \eqref{eq:ptens-prop-aux1} we conclude that the cokernel $A / [A,A]$ is isomorphic to $\mathrm{im}(p_\otimes)$. It remains to show that $\mathrm{im}(p_\otimes) = Z(A)$. For $z \in Z(A)$ we have $p_\otimes(z) =  \sum_i a_i z a_i' =  z(\sum_i a_i  a_i') = z$, where we used \eqref{eq:beta-special-property}. Thus $Z(A) \subset \mathrm{im}(p_\otimes)$. Conversely, if $x \in \mathrm{im}(p_\otimes)$ then $x = \sum_i a_i x a_i'$. Then for all $y \in A$ we have 
\be
  xy 
  = \sum_i a_i x a_i' y 
  = m \circ ( \sum_i a_i \otimes (x a_i' y) )
  = m \circ ( \sum_i (y a_i x) \otimes a_i' )
  =  \sum_i y a_i x a_i' 
  = y x \ ,
\ee
where we used \eqref{eq:beta-move-factors} together with symmetry of $\beta$. Thus $\mathrm{im}(p_\otimes) \subset Z(A)$.
\end{proof}

\begin{rem}
Recall that for a unital associative $R$-algebra (for $R$ a commutative ring), the $0$'th Hochschild homology $H\!H_0(A)$ and cohomology $H\!H^0(A)$ are given by
\be
  H\!H_0(A) = A / [A,A] \qquad \text{and} \qquad H\!H^0(A) = Z(A) \ ,
\ee
see \cite[Sec.\,1.1,\,1.5]{Loday:1992}. Similarly, for an $A$-$A$-bimodule $X$ one has $H_0(A,X) = \,\cytens A X$ and $H^0(A,X) = \{ x \in X \,|\, a.x=x.a \text{ for all } a \in A \}$. In the situation of lemma \ref{lem:ho-hom=ho-cohom}, that is if $A$ is a Frobenius algebra with trace pairing, one finds $H_0(A,X) \cong H^0(A,X)$. We will see in \eqref{eq:def-TFT-T-on-S1} below that the $0$'th Hochschild (co)homology provides the state space which the 2d lattice TFT with defects assigns to a circle with a single marked point.
\end{rem}

\subsection{Data for lattice TFT with defects}\label{sec:lattice-TFT-data}

Fix sets $D_2$, $D_1$, $D_0$ with maps $s,t : D_1 \to D_2$ and a map $j$ as in \eqref{eq:j-definition}. The data that serves as input to the lattice TFT construction is as follows.
\begin{enumerate}
\item
For each $a \in D_2$ a Frobenius algebra $A_a$ with trace pairing.
\item 
For each $x \in D_1$ a finite-dimensional $A_{t(x)}$-$A_{s(x)}$-bimodule $X_x$. 
\end{enumerate}
There is also a piece of data associated to $D_0$, but we need a bit of preparation before we present it. For an $A$-$B$-bimodule $X$ write $X^+ \equiv X$ and write $X^-$ for the $B$-$A$-bimodule $X^*$. Recall the free category with conjugates $\bfD \equiv \bfD[D_2,D_1]$ defined in section \ref{sec:2-cat-from-defect-QFT}. For $\dwt x \in \bfD(a,b)$ we define the $A_b$-$A_a$-bimodule 
\be \label{eq:X-dwtx-def}
  X_{\dwt x} = X_{x_1}^{\eps_1} \otimes_{A_{1,2}} X_{x_2}^{\eps_2} \otimes_{A_{2,3}} \cdots \otimes_{A_{n-1,n}} X_{x_n}^{\eps_n} \ ,
\ee 
where $\dwt x = ((x_1,\eps_1),\dots,(x_n,\eps_n))$ and $A_{i,i+1}$ denotes the algebra which acts from the right on $X_i^{\eps_i}$ and from the left on $X_{i+1}^{\eps_{i+1}}$. Note that 
\be
  X_{\dwt y \circ \dwt x} = X_{\dwt y} \otimes_{A_b} X_{\dwt x} 
  \qquad \text{for} \quad
  c \xleftarrow{\dwt y} b \xleftarrow{\dwt x} a \ .
\ee
Let $\chi$ be an element of $D_1^{(n)} / C_n$ as in \eqref{eq:j-definition}. Let $\dwt x \in D_1^{(n)}$ be a representative of $\chi$, i.e.\ $\chi = [\dwt x]$, and let $\mathcal{O}_\chi = C_n.\dwt x$ be the $C_n$ orbit of $\dwt x$ in $D_1^{(n)}$ (which is independent of the choice of $\dwt x$). Let 
\be
  \pi : J_\chi \to \mathcal{O}_\chi
\ee
be the vector bundle (with discrete base) whose fibres are given by the dual vector space
\be
  \pi^{-1}(\dwt y) = \Hom_k( \cytens{A_{s(\dwt y)=t(\dwt y)}} X_{\dwt y},k) \ ,
\ee
where $\Hom_k(U,V)$ stands for the space of linear maps from $U$ to $V$. Since $\dwt y \in \bfD(s(\dwt y),t(\dwt y))$ is cyclically composable by assumption, we indeed have $s(\dwt y) = t(\dwt y)$. An element $\sigma$ of the cyclic group $C_n$ acts on $J_\chi$ by taking a vector $\varphi \in \pi^{-1}(\dwt y)$ to $\varphi \circ \sigma^{-1} \in \pi^{-1}(\sigma(\dwt y))$. By abuse of notation, here we also denoted by $\sigma$ the linear isomorphism $\cytens{A_{s(\dwt y)}} X_{\dwt y} \to \cytens{A_{s(\sigma \dwt y)}} X_{\sigma \dwt y}$ obtained by shifting tensor factors. Denote by $\Gamma(J_\chi)^\mathrm{inv}$ the space of $C_n$-invariant sections of the bundle $J_\chi$. The value of the section at $\dwt y$ is then invariant under the action of the stabiliser of $\dwt y$ in $C_n$.

For example, if $\chi = [\dwt x]$ is such that all $(x_i,\eps_i)$ in $\dwt x$ are mutually distinct, then the orbit has lengths $n$, all stabilisers are trivial, and $\Gamma(J_\chi)^\mathrm{inv}$ is isomorphic to any one of the fibres of $J_\chi$. If all $(x_i,\eps_i)$ are identical, the orbit has length one and $\Gamma(J_\chi)^\mathrm{inv}$ consists of the $C_n$ invariant vectors in $\Hom_k(\cytens{A_{s(\dwt x)}} X_{\dwt x},k)$.

With this preparation, we can finally state the third piece of data for the lattice TFT construction.
\begin{enumerate}
\item[3.] For each $u \in D_0$, a vector $\varphi_u \in \Gamma(J_{j(u)})^\mathrm{inv}$.
\end{enumerate}
This complicated construction will later ensure that the junction-condition is unchanged under `rotations which leave the attached domain walls invariant', or in other words, it has no preferred `starting edge'.

\subsection{Functor on bordisms with cell decomposition} \label{sec:T^cw-construct}

\begin{figure}[tb] 
\begin{center}
\raisebox{-47pt}{\begin{picture}(120,115)
  \put(0,0){\scalebox{.90}{\includegraphics{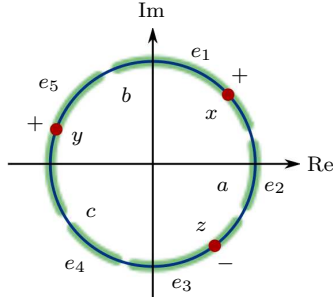}}}
  \put(0,0){
     \setlength{\unitlength}{.90pt}\put(-239,-358){
     \put(333,449)   {\scriptsize $+$ }
     \put(327,370)   {\scriptsize $-$ }
     \put(247,429)   {\scriptsize $+$ }
     \put(365,411)   {\scriptsize Re }
     \put(294,476)   {\scriptsize Im }
     \put(322,433)   {\scriptsize $x$ }
     \put(318,386)   {\scriptsize $z$ }
     \put(266,424)   {\scriptsize $y$ }
     \put(287,440)   {\scriptsize $b$ }
     \put(327,403)   {\scriptsize $a$ }
     \put(272,392)   {\scriptsize $c$ }
     \put(316,459)   {\scriptsize $e_1$ }
     \put(347,403)   {\scriptsize $e_2$ }
     \put(308,362)   {\scriptsize $e_3$ }
     \put(263,370)   {\scriptsize $e_4$ }
     \put(252,446)   {\scriptsize $e_5$ }
  }}
\end{picture}}
\end{center}
\vspace*{-1em}
\caption{Assignment of vector spaces $R_e$ to edges $e \in C_1(O)$: 
The figure shows a circle with three domains labelled $a,b,c \in D_2$ and three marked points labelled $x,y,z \in D_1$. The circle is decomposed into 5 edges, and the corresponding vector spaces are $R_{e_1} = X_x$, an $A_b$-$A_a$-bimodule; $R_{e_2} = A_a$; $R_{e_3} = X_z^*$, an $A_a$-$A_c$-bimodule (while $X_z$ itself is an $A_c$-$A_a$-bimodule); $R_{e_4} = A_c$; $R_{e_5} = X_y$, an $A_c$-$A_b$-bimodule. 
}
\label{fig:object+morph-to-vsp-for-Tcw}
\end{figure}

Fix $D_i$, $i=0,1,2$ and the data described in the previous subsection. We proceed to define a symmetric monoidal functor 
\be
  \Tcw : \bord_{2,1}^\mathrm{def,top,cw} \longrightarrow \vect_f(k) \ .
\ee
The action of $\Tcw$ on objects is as follows. Denote by $O$ an object of the bordism category $\bord_{2,1}^\mathrm{def,top,cw}$ consisting of a single $S^1$. Recall that $C_1(O)$ is the set of 1-cells (i.e.\ edges) of the cell decomposition of $O$. To each edge $e \in C_1(O)$ we assign the vector space
\be \label{eq:Re-def}
  R_e = \begin{cases} 
    A_a &; \text{ $e$ contains no marked point and carries label $a \in D_2$,} \\
    X_x^{\eps} &; \text{ $e$ contains a marked point with orientation $\eps$ and label $x\in D_1$,} 
  \end{cases}
\ee
see figure \ref{fig:object+morph-to-vsp-for-Tcw} for an illustration. We set\footnote{
  Here, the tensor product stands for the tensor product over $k$ of a family of vector spaces indexed by some set $I$ (here $C_1(O)$). To define this tensor product it is not necessary to choose an ordering of $I$, i.e.\ a preferred way to write out the tensor product in a linear order. The same applies to similar tensor products below.
}
\be\label{eq:Tcw-on-single-S1}
 \Tcw(O) \, = \!\!\bigotimes_{e \in C_1(O)}\!\! R_e \ .
\ee
For an object $U = O_1 \sqcup O_2 \sqcup \cdots \sqcup O_n$ we take\footnote{
  Of course we could have taken the definition \eqref{eq:Tcw-on-single-S1} also for a general object $U$ instead of writing $\Tcw(U)$ as a tensor product with implied ordering of the factors. However, in this way it is easier to see that the functor is symmetric.}
$\Tcw(U) = \Tcw(O_1) \otimes \cdots \otimes \Tcw(O_n)$.

\medskip

There are two types of morphisms in the bordism category: permutations of objects, and bordisms. A permutation $\sigma : U \to \sigma(U)$ is mapped by $\Tcw$ to the corresponding permutation of tensor factors. The description of $\Tcw$ for bordisms will take a little while. Given a bordism $M : U \to V$ (we assume that a representative of the equivalence class has been chosen and use the same symbol), we will write the functor as a composition of two linear maps
\be \label{eq:Tcw-as-two-maps}
  \Tcw(M) ~:~ \Tcw(U) \xrightarrow{\id_{\Tcw(U)} \otimes P(M)} \Tcw(U) \otimes Q(M) \otimes \Tcw(V) \xrightarrow{E(M) \otimes \id_{\Tcw(V)}} \Tcw(V) \ .
\ee
We will now describe the vector space $Q(M)$, and the maps $P(M)$ (`propagator') and $E(M)$ (`evaluation').

\medskip

\begin{figure}[tb] 
\begin{center}
\raisebox{50pt}{a)} \hspace{.5em}
\raisebox{-63pt}{\begin{picture}(128,128)
  \put(0,0){\scalebox{.7}{\includegraphics{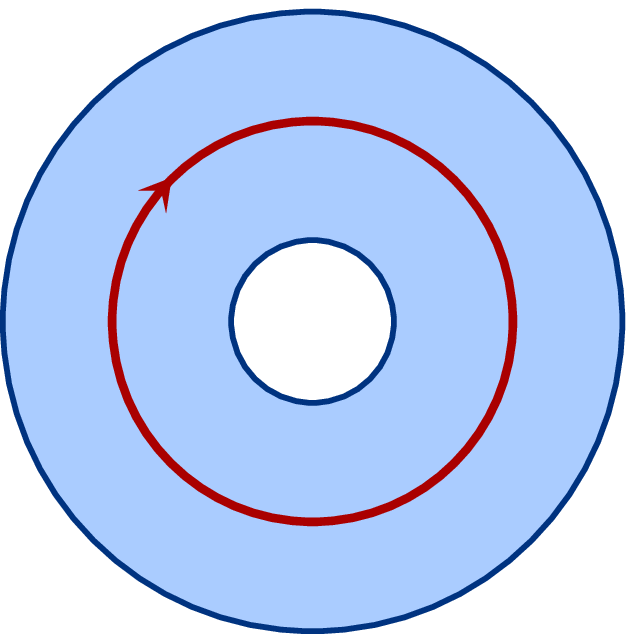}}}
  \put(0,0){
     \setlength{\unitlength}{.7pt}\put(-211,-324){
     \put(242,445)   {\scriptsize $x$ }
     \put(280,482)   {\scriptsize $t(x)=b$ }
     \put(280,450)   {\scriptsize $s(x)=a$ }
     \put(302,400)   {\scriptsize in }
     \put(356,332)   {\scriptsize out }
  }}
\end{picture}}
\hspace{5em}
\raisebox{50pt}{b)} \hspace{.5em}
\raisebox{-63pt}{\begin{picture}(128,128)
  \put(0,0){\scalebox{.74}{\includegraphics{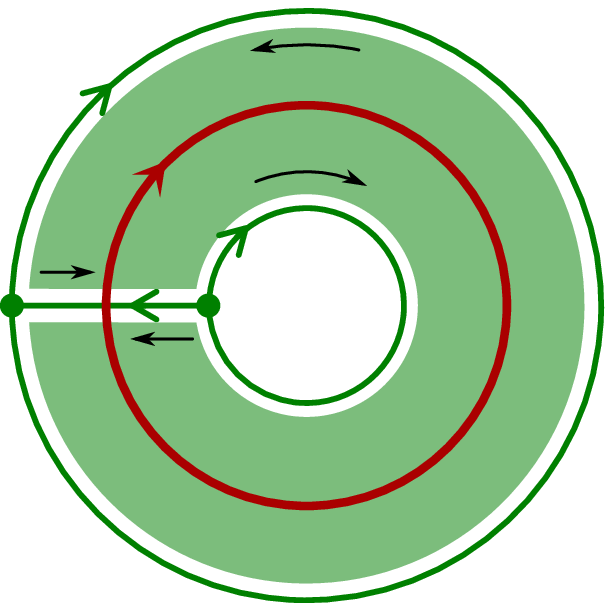}}}
  \put(0,0){
     \setlength{\unitlength}{.74pt}\put(-215,-328){
     \put(242,445)   {\scriptsize $x$ }
     \put(338,433)   {\scriptsize $a$ }
     \put(362,439)   {\scriptsize $b$ }
     \put(288,428)   {\scriptsize $e_\alpha$ }
     \put(226,481)   {\scriptsize $e_ \gamma$ }
     \put(228,403)   {\scriptsize $e_ \beta $ }
     \put(247,364)   {\scriptsize $p$ }
  }}
\end{picture}}
\end{center}
\vspace*{-1em}
\caption{Denote by $O(a)$ the circle labelled by $a \in D_2$; figure a) shows a bordism $A(x) : O(a) \to O(b)$ with a single circular domain wall labelled $x$. Figure b) gives a possible cell decomposition of $A(x)$ into a 2-cell $p$, 1-cells $e_\alpha$, $e_\beta$, $e_\gamma$ and two unnamed 0-cells. The edges are oriented for convenience, so that we can describe the two possible orientations simply by $\pm$ (the orientation is not part of the data of the cell decomposition); the orientation of the boundary of $p$ induced by the orientation of $A(x)$ is indicated by the arrows placed in $p$. There are 4 allowed triples: $(p,e_\alpha,+)$, $(p,e_\beta,+)$, $(p,e_\gamma,-)$, $(p,e_\beta,-)$. The corresponding vector spaces are 
$Q_{p,e_\alpha,+} = A_a$,
$Q_{p,e_\beta,+} = X_x^*$, 
$Q_{p,e_\gamma,-} = A_b$, 
$Q_{p,e_\beta,-} = X_x$. 
As $e_\alpha$ lies on $\partial_\text{in} A(x)$, we have $Q(A(x)) = Q_{p,e_\beta,+} \otimes Q_{p,e_\gamma,-} \otimes Q_{p,e_\beta,-}$.
}
\label{fig:annulus-cell-decomp}
\end{figure}

We start with the vector space $Q(M)$.
Denote by $\partial_\text{in}M$ the in-going part of the boundary of $M$, that is, the part parametrised by $U$, and by $\partial_\text{out}M$ the out-going part of the boundary of $M$, parametrised by $V$.

Consider triples $(p,e,\mathrm{or})$, where $p \in C_2(M)$ is a 2-cell (i.e.\ an open polygon), $e \in C_1(M)$ is a 1-cell, and `$\mathrm{or}$' is an orientation of $e$. We only allow triples which satisfy the following condition: the orientation of $M$ also orients $p$, and this in turn induces an orientation of the boundary $\partial p$; we demand that $(e,\mathrm{or})$ is part of $\partial p$ as an oriented edge. This is illustrated in figure \ref{fig:annulus-cell-decomp}.
To each allowed triple $(p,e,\mathrm{or})$ we assign a vector space:
\be\label{eq:Q_pe-def}
Q_{p,e,\mathrm{or}} = \begin{cases}
  A_a &; \text{ $(e,\mathrm{or})$ does not intersect $M_1$ and is in a component}\\[-.3em]
           &  \phantom{;}\, \text{ of $M_2$ labelled $a$,} \\
  X_x &; \text{ $(e,\mathrm{or})$ intersects a component of $M_1$ which is labelled $x$}\\[-.3em]
           & \phantom{;}\, \text{ and is oriented into the polygon $p$,} \\
  X_x^* &; \text{ $(e,\mathrm{or})$ intersects a component of $M_1$ which is labelled $x$}\\[-.3em]
           & \phantom{;}\, \text{ and is oriented out of the polygon $p$.} 
\end{cases}
\ee
Here, the pair $(e,\mathrm{or})$ is understood as part of the boundary $\partial p$; this is important if the same edge $e$ occurs twice in $\partial p$, see the example in figure \ref{fig:annulus-cell-decomp}. The vector space $Q(M)$ is given by
\be \label{Q(M)-def}
  Q(M) ~= \hspace{-1em}\bigotimes_{(p,e,\mathrm{or}) ,\, e \notin\partial_\text{in} M} \hspace{-1em}  Q_{p,e,\mathrm{or}} \quad ,
\ee
where the tensor product is taken over all allowed triples $(p,e,\mathrm{or})$ for which $e$ does not lie in $\partial_\text{in} M$.

We now turn to the description of the map $P(M) : k \to Q(M) \otimes \Tcw(V)$. 
Each edge $e \in C_1(M)$ in the interior of $M$ occurs in two allowed triples, let us call them $(p(e)_1,e,\mathrm{or}_1)$ and $(p(e)_2,e,\mathrm{or}_2)$,  where $\mathrm{or}_1$ and $\mathrm{or}_2$ are the two possible orientations of $e$, and $p(e)_i$ is the polygon which contains the oriented edge $(e,\mathrm{or}_i)$ in its boundary. Note that it may happen that $p(e)_1=p(e)_2$, as it does in figure \ref{fig:annulus-cell-decomp}. For each interior edge $e$ define the linear map 
\be \label{eq:Pe-def}
  P_e \,:\, k \longrightarrow Q_{p(e)_1,e,\mathrm{or}_1} \otimes Q_{p(e)_2,e,\mathrm{or}_2} 
\ee  
according to the following two cases.
\begin{enumerate}
\item If $M_1$ does not intersect $e$ then according to \eqref{eq:Q_pe-def} we have $Q_{p(e)_1,e,\mathrm{or}_1} =  Q_{p(e)_2,e,\mathrm{or}_2} = A_a$, where $a$ is the label of the component of $M_2$ containing $e$. We take $P_e = \beta_{A_a}$, with $\beta$ the dual of the Frobenius pairing as in section \ref{sec:alg-prel}. Since $A_a$ has trace paring, $\beta$ is symmetric and the map $P_e$ is independent of the choice of order of $(p(e)_1,e,\mathrm{or}_1)$ and $(p(e)_2,e,\mathrm{or}_2)$. 
\item Suppose a component of $M_1$ labelled by $x$ intersects $e$. Let $u_i$ be a basis of $X_x$ and let $u_i^*$ be the dual basis of $X_x^*$. Choose the numbering `$1$' and `$2$' of  $(p(e)_1,e,\mathrm{or}_1)$ and $(p(e)_2,e,\mathrm{or}_2)$ so that the orientation of the domain wall $M_1$ is such that it points into the polygon $p(e)_1$ at $(e,\mathrm{or}_1)$ and out of the polygon $p(e)_2$ at $(e,\mathrm{or}_2)$. Then $Q_{p(e)_1,e,\mathrm{or}_1} = X_x$ and $Q_{p(e)_2,e,\mathrm{or}_2}=X_x^*$ (see figure \ref{fig:annulus-cell-decomp}) and we set $P_e(\lambda) = \lambda \sum_i u_i \otimes u_i^*$. 
\end{enumerate}
If $e$ is an edge on the out-going boundary $\partial_\text{out} M$, i.e.\ the boundary component parametrised by $V$, then there is exactly one allowed triple which contains $e$. Let $(p,e,\mathrm{or})$ be that triple. The parametrisation identifies $e$ with an edge of $C(V)$ which we also call $e$. In this case $P_e$ is defined as in \eqref{eq:Pe-def}, but with $Q_{p(e)_1,e,\mathrm{or}_1}$ and $Q_{p(e)_2,e,\mathrm{or}_2}$ replaced by $Q_{p(e),e,\mathrm{or}}$ and $R_e$. Comparing \eqref{eq:Re-def} and \eqref{eq:Q_pe-def} (and using the conventions in figure \ref{fig:collar}), one checks that cases 1 and 2 above still apply.

Altogether, the map $P(M) : k \to Q(M) \otimes \Tcw(V)$ is defined as
\be \label{eq:P(M)-def}
  P(M) ~= \hspace{-1em}\bigotimes_{e \in C_1(M),\,e \notin \partial_\text{in} M} \hspace{-1em} P_e \ .
\ee

Finally, we need to define $E(M) : \Tcw(U) \otimes Q(M) \to k$. Note that $\Tcw(U) \otimes Q(M)$ contains one factor $Q_{p,e,\mathrm{or}}$ for each $p \in C_2(M)$ and $(e,\mathrm{or}) \in \partial p$, even if $e \subset \partial M$, such that
\be
  \Tcw(U) \otimes Q(M) ~= \hspace{-1em}\bigotimes_{p \in C_2(M), \, (e,\mathrm{or}) \in \partial p} \hspace{-1em} Q_{p,e,\mathrm{or}} \ .
\ee
For each polygon $p \in C_2(M)$ we define a linear map 
\be \label{eq:Ep-def}
  E_p \,: \hspace{-.5em}\bigotimes_{ (e,\mathrm{or}) \in \partial p }\hspace{-.5em} Q_{p,e,\mathrm{or}} \longrightarrow k \ .
\ee
Fix $p \in C_2(M)$. There are three cases to distinguish, depending on whether $p$ intersects $M_1$ and/or $M_0$. 
\begin{enumerate}
\item 
Suppose $p$ intersects neither $M_1$ nor $M_0$. Let $a$ be the label of the component of $M_2$ containing $p$. 
Choose an edge $(e_1,\mathrm{or}_1) \in \partial p$ and denote by $(e_1,\mathrm{or}_1), (e_2,\mathrm{or}_2), \dots, (e_m,\mathrm{or}_m)$ all oriented edges of $\partial p$ in anti-clockwise ordering. Let further 
\be\label{EM-def-aux1}
  q_1 \otimes q_2 \otimes \cdots \otimes q_m \in Q_{p,e_1,\mathrm{or}_1} \otimes Q_{p,e_2,\mathrm{or}_2} \otimes \cdots \otimes Q_{p,e_m,\mathrm{or}_m} \ .
\ee  
Each $Q_{p,e_i,\mathrm{or}_i}$ is equal to $A_a$ and we set $E_p(q_1 \otimes \cdots \otimes q_m) = \eps_{A_a}(q_1 \cdots q_m)$, where $\eps_{A_a}$ is the counit of $A_a$.
By symmetry of the pairing of the Frobenius algebra $A_a$, the result is independent of the choice of starting edge $(e_1,\mathrm{or}_1)$.
\item 
Suppose $p$ intersects $M_1$ but not $M_0$. In this case there is one oriented edge where $M_1$ leaves $p$, which we take to be $(e_1,\mathrm{or}_1)$. Then we order the oriented edges of $\partial p$ anti-clockwise as in 1. Let $(e_i,\mathrm{or}_i)$ be the edge where $M_1$ enters $p$ and let $x$ be the label of the component of $M_1$ in $p$.
In the notation from \eqref{EM-def-aux1}, we have $q_1 \in X_x^*$, $q_i \in X_x$, and $q_2, \dots, q_{i-1} \in A_{t(x)}$ and $q_{i+1},\dots,q_m \in A_{s(x)}$. We set 
$E_p(q_1 \otimes \cdots \otimes q_m) = q_1\big( (q_2 \cdots q_{i-1}). q_i . (q_{i+1} \cdots q_m) \big)$. Unlike case 1., case 2.\ did not involve an arbitrary choice, and there is no invariance condition to check.
\item 
Suppose $p$ contains a point $u$ from $M_0$ of orientation $\nu_u \in \{\pm\}$ and with label $\hat d_0(u) = t \in D_0$. As in 1., we choose an arbitrary starting edge $(e_1,\mathrm{or}_1) \in \partial p$ and order the remaining edges anti-clockwise. Each edge $e_i$, $i=1,\dots,m$ is transversed by a domain wall. For each $i=1,\dots,m$ we thereby obtain a pair $(x_i,\eps_i)$ where $x_i \in D_1$ is the label of the domain wall crossing $e_i$, and $\eps_i = +$ if this domain wall is oriented into the polygon at $(e_i,\mathrm{or}_i)$ and $\eps_i=-$ otherwise. Let $\dwt x = ((x_1,\eps_1),\dots,(x_m,\eps_m))$. 

If $\nu_u=+$, the labelling has to satisfy $j(t) = [\dwt x] =: \chi$. According to the construction in section \ref{sec:lattice-TFT-data}, $\dwt x \in \mathcal{O}_\chi$. Evaluating the section $\varphi_t \in \Gamma(J_{j(t)})^\mathrm{inv}$ at $\dwt x$ gives an element $\psi \in \Hom_k( \cytens{A_{s(\dwt x)=t(\dwt x)}} X_{\dwt x},k)$. Precomposing with the projection $\pi_\otimes : X_{\dwt x} \to  \cytens{A_{s(\dwt x)=t(\dwt(x)}} X_{\dwt x}$ we obtain a linear form $\psi \circ \pi_\otimes : X_{x_1}^{\eps_1} \otimes \cdots \otimes X_{x_m}^{\eps_m}\to k$. We set $E_p(q_1 \otimes \cdots \otimes q_m) = \psi \circ \pi_\otimes(q_1 \otimes  \cdots \otimes q_m)$. Independence of the choice of $(e_1,\mathrm{or}_1)$ follows since $\Gamma(J_{j(t)})^\mathrm{inv}$ consists of elements invariant under cyclic permutations.

If $\nu_u=-$, the labelling has to satisfy $j(t) = [\dwt x^*]$, and the above construction is repeated with $\dwt x^*$ instead of $\dwt x$.
\end{enumerate}
Figure \ref{fig:annulus-cell-decomp} gives an example of case 2. There, 
$(e_1,\mathrm{or}_1) = (e_\beta,+)$, 
$(e_2,\mathrm{or}_2) = (e_\gamma,-)$, 
$(e_3,\mathrm{or}_3) = (e_\beta,-)$, 
$(e_4,\mathrm{or}_4) = (e_\alpha,+)$, and $E_p(q_1 \otimes q_2 \otimes q_3 \otimes q_4) = q_1(q_2.q_3.q_4)$, where $q_2.q_3.q_4$ is the left/right action of $q_2 \in A_b$ and $q_4 \in A_a$ on $q_3 \in X_x$. 

Altogether, for $E(M)$ we take
\be \label{eq:E(M)-def}
 E(M) = \hspace{-.5em}\bigotimes_{p \in C_2(M)}\hspace{-.5em} E_p \ .
\ee
This completes the definition of $\Tcw$. 

\medskip

Let us briefly illustrate the construction in two related examples;  more examples will be computed in section \ref{sec:example-amplitude}.
For the bordism $A(x) : O(a) \to O(b)$ considered in figure \ref{fig:annulus-cell-decomp}, the composition of maps in \eqref{eq:Tcw-as-two-maps} reads
\be
 \Tcw(A(x)) : 
 \Tcw(O(a)) 
 \xrightarrow{\id \otimes P} 
 R_{e_\alpha} \otimes Q_{p,e_\beta,+} \otimes Q_{p,e_\gamma,-} \otimes Q_{p,e_\beta,-} \otimes R_{e_\gamma}
 \xrightarrow{E \otimes \id} 
 \Tcw(O(b)) \ ,
\ee
where the edge on $O(a)$ is identified with $e_\alpha$ via the parametrisation, and the edge on $O(b)$ with $e_\gamma$. Substituting the definition of these vector spaces and maps gives
\be \label{eq:Tcw-ann-example}
\begin{array}{rccccl}
\Tcw(A(x)) : &
A_a &
\xrightarrow{\id \otimes P}& 
A_a \otimes X_x^* \otimes A_b \otimes X_x \otimes A_b &
\xrightarrow{E \otimes \id} &
A_b
\\[.5em]
& q & 
\longmapsto &
\sum_{i,j} q \otimes u_i^* \otimes b_j' \otimes u_i \otimes b_j &
\longmapsto &
\sum_{i,j} u_i^*(b_j'.u_i.q) \, b_j \ .
\end{array}
\ee
This map can be defined for any two Frobenius algebras with trace pairing $A,B$ and a finite-dimensional $B$-$A$-bimodule $X$. One can check that the image of this map lies in $Z(B)$, and that the kernel of the projector $p_\otimes$ onto $Z(A)$ is contained in the kernel of $T(A(x))$. We therefore lose nothing if we restrict ourselves to $Z(A)$ and $Z(B)$ from the start:
\be
   D(X) : Z(A) \longrightarrow Z(B) 
   \quad  , \quad z \mapsto \sum_{i,j} u_i^*(b_j'.u_i.z) \, b_j \ .
\ee   
This is an example of a defect operator, which we already briefly mentioned in remark \ref{rem:top-def-interesting-for-qft}\,(ii). Such defect operators have some nice properties\footnote{
  These properties are all easily checked directly with the methods of section \ref{sec:alg-prel}. They have also been shown in arbitrary modular categories (instead of just the category $\vect_f(k)$) in \cite[Lem.\,2]{Fuchs:2007vk} and \cite[Lem.\,3.1]{Kong:2007yv}.}:
 if $X \cong X'$ as bimodules, then $D(X) = D(X')$; if $Y$ is a $C$-$B$-bimodule, then $D(Y)D(X) = D(Y \otimes_B X)$; and for the $A$-$A$-bimodule $A$ one has $D(A) = \id_{Z(A)}$.
 
A related example comes from the annulus as in figure \ref{fig:annulus-cell-decomp}, but without the domain wall $x$, so that necessarily $a=b$. The map in \eqref{eq:Tcw-ann-example} specialises to $q \mapsto \sum_{i,j} \langle a_i',a_j' a_i q\rangle \, a_j$. By \eqref{eq:Frob-tracepair-completeness}, this is equal to $q \mapsto \sum_{i} a_i q a_i' = p_\otimes(q)$, cf.\ lemma \ref{lem:ho-hom=ho-cohom}. Thus, $\Tcw$ maps the cylinder over $O(a)$ to the projector onto the centre of $A_a$.

\medskip

It is fairly straightforward to see from the above construction that $\Tcw$ is compatible with composition and tensor products. Since we imposed that permutations of $S^1$-components in an object $U$ of the bordism category get mapped to permutations of tensor factors in $\Tcw(U)$, the functor respects identities and is symmetric.

\subsection{Independence of cell decomposition} \label{sec:T^cw-cell-indep}

In this subsection we abbreviate $\bord^\mathrm{cw} \equiv \bord_{2,1}^\mathrm{def,top,cw}$ and $\bord \equiv \bord_{2,1}^\mathrm{def,top}$. Objects and morphisms in $\bord^\mathrm{cw}$ will be decorated by a tilde (e.g.\ $\tilde U$, $\tilde M$, \dots). Recall the forgetful functor $F : \bord^\mathrm{cw} \to  \bord$ from section \ref{sec:smooth-bord-cat}. We will show that there exists a symmetric monoidal functor $T$ making the diagram \eqref{eq:Tcw-T-comm-diag} commute (consequently, this functor is unique). This will be done in several steps, the key one being the following lemma.

\begin{figure}[tb] 
\raisebox{30pt}{a)} \hspace{-1em}
\raisebox{-14pt}{\begin{picture}(60,40)
  \put(0,0){\scalebox{.6}{\includegraphics{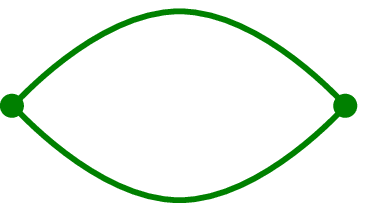}}}
\end{picture}}
$~\longleftrightarrow~$
\raisebox{-14pt}{\begin{picture}(60,40)
  \put(0,0){\scalebox{.6}{\includegraphics{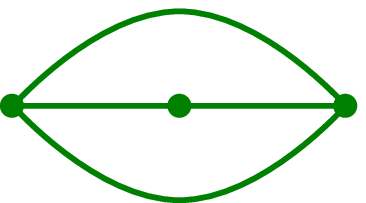}}}
\end{picture}}
\hspace{2em}
\raisebox{30pt}{b)} \hspace{-1em}
\raisebox{-14pt}{\begin{picture}(60,40)
  \put(0,0){\scalebox{.6}{\includegraphics{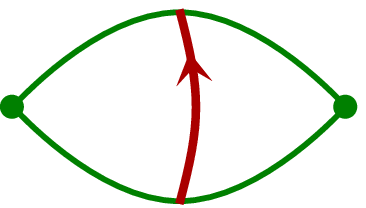}}}
\end{picture}}
$~\longleftrightarrow~$
\raisebox{-14pt}{\begin{picture}(60,40)
  \put(0,0){\scalebox{.6}{\includegraphics{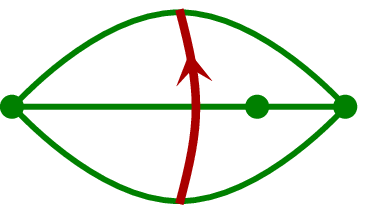}}}
\end{picture}}
$~\longleftrightarrow~$
\raisebox{-14pt}{\begin{picture}(60,40)
  \put(0,0){\scalebox{.6}{\includegraphics{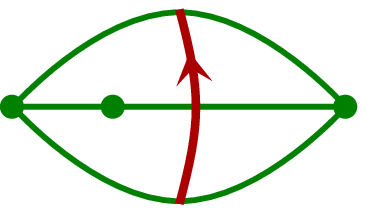}}}
\end{picture}}
\caption{a) A local modification of the cell decomposition which adds an edge and a vertex turning a 2-gon into two triangles, or conversely. The two exterior vertices are allowed to be identical. b) The same in the presence of a domain wall; there are two modifications as the vertex can be added on either side of the domain wall.}
\label{fig:edge-vertex-add}
\end{figure}

\begin{figure}[tb] 
\begin{center}
\raisebox{-32pt}{\begin{picture}(65,75)
  \put(0,0){\scalebox{.7}{\includegraphics{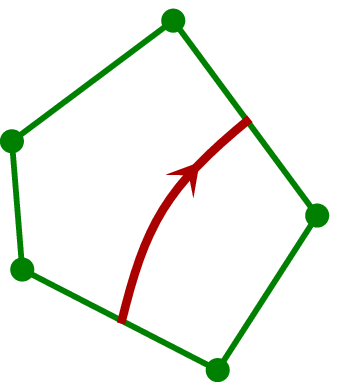}}}
\end{picture}}
$~~\longleftrightarrow~~$
\raisebox{-32pt}{\begin{picture}(65,75)
  \put(0,0){\scalebox{.7}{\includegraphics{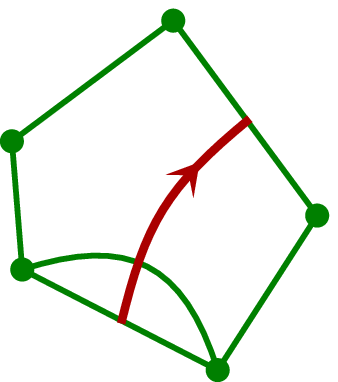}}}
\end{picture}}
$~~\longleftrightarrow~~$
\raisebox{-32pt}{\begin{picture}(65,75)
  \put(0,0){\scalebox{.7}{\includegraphics{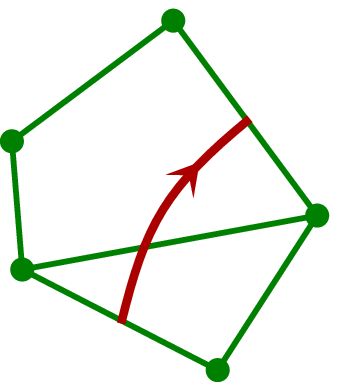}}}
\end{picture}}
$~~\longleftrightarrow~ \cdots ~\longleftrightarrow~~$
\raisebox{-32pt}{\begin{picture}(65,75)
  \put(0,0){\scalebox{.7}{\includegraphics{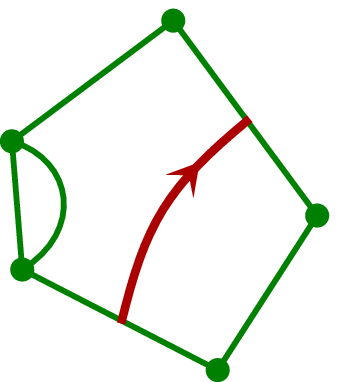}}}
\end{picture}}
\end{center}
\vspace*{-1em}
\caption{A local modification of the cell decomposition which adds a single edge to the interior of a 2-cell. The figure shows an exemplary situation. Alternatively, the 2-cell can be any $n$-gon with $n \ge 2$, the domain wall can run between other edges, or there could be no domain wall at all.}
\label{fig:edge-add}
\end{figure}

\begin{lem} \label{lem:Tcw-local-moves}
Let $\tilde U, \tilde V \in \bord^\mathrm{cw}$ and let  $\tilde M , \tilde M' : \tilde U \to \tilde V$ be morphisms. If $\tilde M'$ is obtained from $\tilde M$ by one of the local modifications of the cell decomposition shown in figures \ref{fig:edge-vertex-add} and \ref{fig:edge-add}, then 
$\Tcw(\tilde M) = \Tcw(\tilde M')$. 
\end{lem}

\begin{figure}[tb] 
\begin{center}
\begin{tabular}{ll}
\raisebox{40pt}{a)} 
\raisebox{-48pt}{\begin{picture}(144,105)
  \put(0,10){\scalebox{.8}{\includegraphics{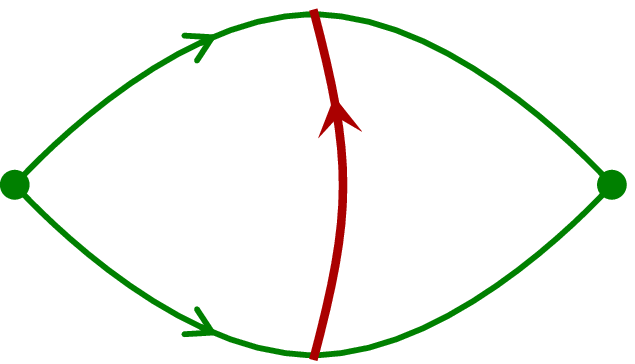}}}
  \put(0,10){
     \setlength{\unitlength}{.8pt}\put(-210,-364){
     \put(296,388)   {\scriptsize $x$ }
     \put(315,442)   {\scriptsize $a$ }
     \put(290,442)   {\scriptsize $b$ }
     \put(233,449)   {\scriptsize $e_1$ }
     \put(233,378)   {\scriptsize $e_2$ }
     \put(270,410)   {\scriptsize $p$ }
     \put(294,355)   {\scriptsize $X_x$ }
     \put(294,472)   {\scriptsize $X_x^*$ }
  }}
\end{picture}} \hspace{4em}
&
\raisebox{40pt}{b)} 
\raisebox{-48pt}{\begin{picture}(144,105)
  \put(0,10){\scalebox{.8}{\includegraphics{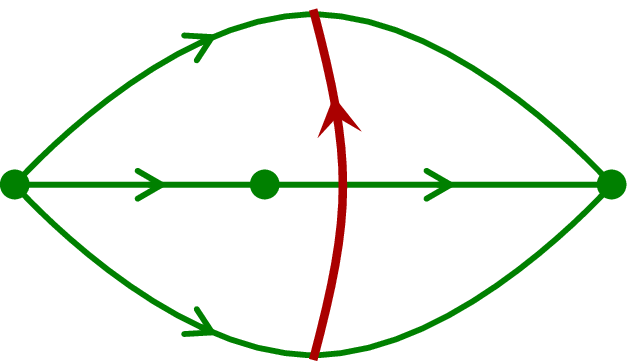}}}
  \put(0,10){
     \setlength{\unitlength}{.8pt}\put(-210,-364){
     \put(296,388)   {\scriptsize $x$ }
     \put(315,442)   {\scriptsize $a$ }
     \put(290,442)   {\scriptsize $b$ }
     \put(233,449)   {\scriptsize $e_1$ }
     \put(233,378)   {\scriptsize $e_2$ }
     \put(236,407)   {\scriptsize $e_3$ }
     \put(351,407)   {\scriptsize $e_4$ }
     \put(270,441)   {\scriptsize $p_1$ }
     \put(270,388)   {\scriptsize $p_2$ }
     \put(315,422)   {\scriptsize $X_x$ }
     \put(315,403)   {\scriptsize $X_x^*$ }
     \put(294,355)   {\scriptsize $X_x$ }
     \put(294,472)   {\scriptsize $X_x^*$ }
     \put(259,403)   {\scriptsize $A_b$ }
     \put(259,422)   {\scriptsize $A_b$ }
  }}
\end{picture}}
\\[5em]
\raisebox{50pt}{c)}  \hspace{1.5em}
\raisebox{-60pt}{\begin{picture}(110,125)
  \put(0,0){\scalebox{.6}{\includegraphics{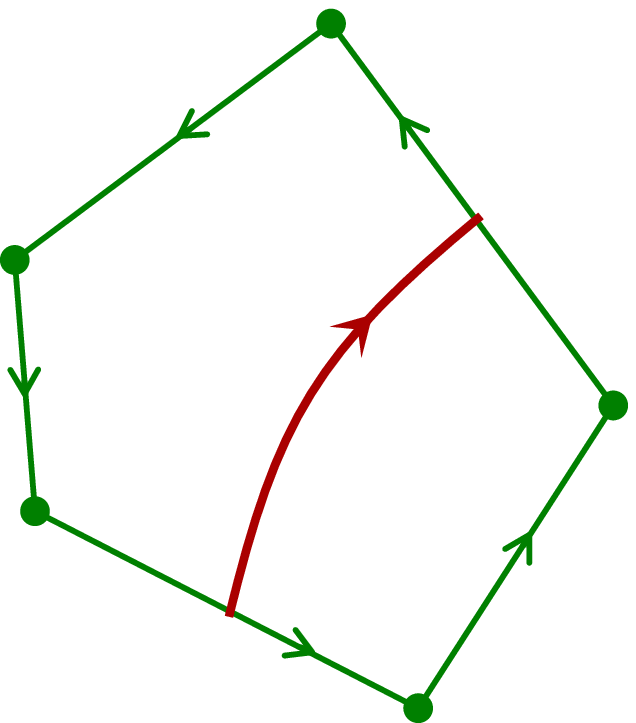}}}
  \put(0,0){
     \setlength{\unitlength}{.6pt}\put(-210,-313){
     \put(245,488)   {\scriptsize $A_b$ }
     \put(195,405)   {\scriptsize $A_b$ }
     \put(362,349)   {\scriptsize $A_a$ }
     \put(266,326)   {\scriptsize $X_x$ }
     \put(353,458)   {\scriptsize $X_x^*$ }
     \put(306,434)   {\scriptsize $x$ }
     \put(329,422)   {\scriptsize $a$ }
     \put(285,436)   {\scriptsize $b$ }
     \put(251,401)   {\scriptsize $p$ }
     \put(312,477)   {\scriptsize $e_1$ }
     \put(280,482)   {\scriptsize $e_2$ }
     \put(221,425)   {\scriptsize $e_3$ }
     \put(299,340)   {\scriptsize $e_4$ }
     \put(347,372)   {\scriptsize $e_5$ }
  }}
\end{picture}}
&
\raisebox{50pt}{d)} \hspace{1.5em}
\raisebox{-60pt}{\begin{picture}(110,125)
  \put(0,0){\scalebox{.6}{\includegraphics{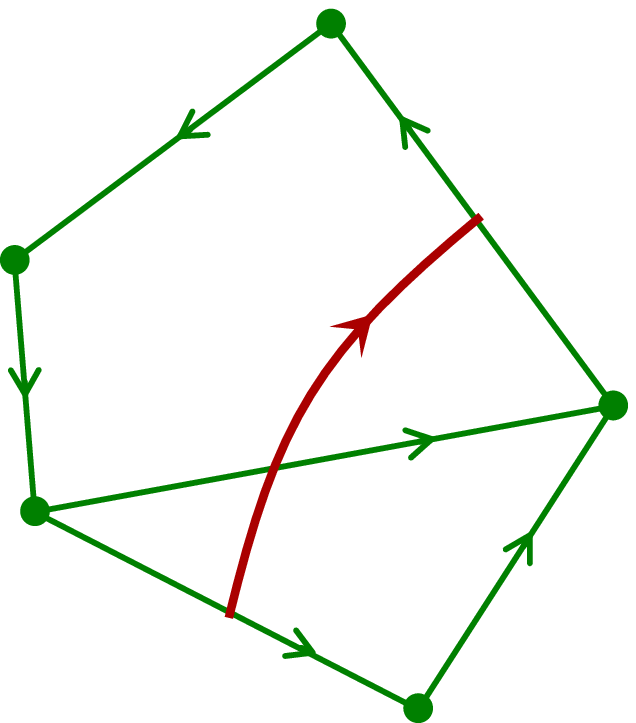}}}
  \put(0,0){
     \setlength{\unitlength}{.6pt}\put(-210,-313){
     \put(245,488)   {\scriptsize $A_b$ }
     \put(195,405)   {\scriptsize $A_b$ }
     \put(362,349)   {\scriptsize $A_a$ }
     \put(266,326)   {\scriptsize $X_x$ }
     \put(353,458)   {\scriptsize $X_x^*$ }
     \put(306,434)   {\scriptsize $x$ }
     \put(329,422)   {\scriptsize $a$ }
     \put(285,436)   {\scriptsize $b$ }
     \put(261,451)   {\scriptsize $p_1$ }
     \put(256,365)   {\scriptsize $p_2$ }
     \put(312,477)   {\scriptsize $e_1$ }
     \put(280,482)   {\scriptsize $e_2$ }
     \put(221,425)   {\scriptsize $e_3$ }
     \put(299,340)   {\scriptsize $e_4$ }
     \put(347,372)   {\scriptsize $e_5$ }
     \put(267,393)   {\scriptsize $X_x$ }
     \put(292,370)   {\scriptsize $X_x^*$ }
     \put(337,405)   {\scriptsize $e_6$ }
  }}
\end{picture}}
\end{tabular}
\end{center}
\vspace*{-1em}
\caption{Two examples. Figures a,b: Adding two edges and a vertex; the vector spaces $Q_{p,e,\mathrm{or}}$ associated to triples $(p,e,\mathrm{or})$ as in section \ref{sec:T^cw-construct} are also shown, e.g.\ $Q_{p_2,e_4,-} = X_x^*$ and $Q_{p_1,e_4,+} = X_x$. Figures c,d: Adding an edge to a 2-cell; also shown are the associated vector spaces.}
\label{fig:edge-vertex-add-example}
\end{figure}

\begin{proof}
We will show the equality $\Tcw(\tilde M) = \Tcw(\tilde M')$ in the two cases displayed in figure \ref{fig:edge-vertex-add-example}, the remaining cases are treated analogously. 

The part of the cell decomposition in figure \ref{fig:edge-vertex-add-example}\,a) contributes the factor $E_p : X_x^* \otimes X_x \to k$, $\varphi \otimes x \mapsto \varphi(x)$ from \eqref{eq:Ep-def} to the map $E$ in \eqref{eq:E(M)-def} and \eqref{eq:Tcw-as-two-maps}. Figure \ref{fig:edge-vertex-add-example}\,b) contributes the factors $E_{p_1} \otimes E_{p_2} : (X_x^* \otimes A_b \otimes X_x) \otimes (X_x^* \otimes A_b \otimes X_x) \to k$, $\varphi \otimes b \otimes x \otimes \varphi' \otimes b' \otimes x' \mapsto \varphi(b.x) \cdot \varphi'(b'.x')$ to $E$, and to $P$ in \eqref{eq:P(M)-def} it contributes the factors $P_{e_3} \otimes P_{e_4} : k \to (A_b \otimes A_b) \otimes (X_x \otimes X_x^*)$, $1 \mapsto \sum_{i,j} b_i' \otimes b_i \otimes u_j \otimes u_j^*$. The composition of $P_{e_3} \otimes P_{e_4} \otimes \id_{X_x^*} \otimes \id_{X_x}$ and $E_{p_1} \otimes E_{p_2}$ (with the appropriate permutation of tensor factors) yields 
\be
  \varphi \otimes x
  \mapsto 
  \sum_{i,j} \varphi(b_i'.u_j) \cdot u_j^*(b_i.x)
  = \sum_{i} \varphi(b_i'.b_i.x) = \varphi(x) = E_p(\varphi \otimes x) \ .
\ee
Thus if $\tilde M$ and $\tilde M'$ differ only in one place as shown in figure \ref{fig:edge-vertex-add-example}\,a,b), we still have $\Tcw(\tilde M) = \Tcw(\tilde M')$.

For figure \ref{fig:edge-vertex-add-example}\,c,d) the argument is the same. The 2-cell $p$ contributes the map $E_p : X_x^* \otimes A_b \otimes A_b \otimes X_x \otimes A_a \to k$, $\varphi \otimes b_1 \otimes b_2 \otimes x \otimes a \mapsto \varphi(b_1.b_2.x.a)$. The two cells $p_1$ and $p_2$ contribute $E_{p_1} \otimes E_{p_2} : (X_x^* \otimes A_b \otimes A_b \otimes X_x) \otimes (X_x^* \otimes X_x \otimes A_a) \to k$, $\varphi \otimes b_1 \otimes b_2 \otimes x' \otimes \varphi' \otimes x \otimes a \mapsto \varphi(b_1.b_2.x') \cdot \varphi'(x.a)$. The new edge $e_6$ gives the map $P_{e_6} : k \to X_x \otimes X_x^*$, $1 \mapsto \sum_j u_j \otimes u_j^*$. Composing the two as $(E_{p_1} \otimes E_{p_2}) \circ (\id \otimes P_{e_6} \otimes \id)$ results in a map
$X_x^* \otimes A_b \otimes A_b \otimes X_x \otimes A_a \to k$ which acts as
\be
\begin{array}{ll}
  \varphi \otimes b_1 \otimes b_2 \otimes x \otimes a ~ \mapsto 
  &\sum_j (E_{p_1} \otimes E_{p_2})( \varphi \otimes b_1 \otimes b_2 \otimes u_j \otimes u_j^* \otimes x \otimes a )
\\[.5em]
  & \quad =
  \sum_j \varphi(b_1.b_2.u_j) \cdot u_j^*(x.a)
  =
  \varphi(b_1.b_2.x.a)
\\[.5em]
  & \quad = E_p( \varphi \otimes b_1 \otimes b_2 \otimes x \otimes a ) \ .
\end{array}
\ee
Hence, if $\tilde M$ and $\tilde M'$ differ only in one place as shown in figure \ref{fig:edge-vertex-add-example}\,c,d), we have $\Tcw(\tilde M) = \Tcw(\tilde M')$.
\end{proof}

\begin{figure}[tb] 
\vspace*{1em}
\begin{center}
\raisebox{-12pt}{\begin{picture}(80,30)
  \put(0,0){\scalebox{.5}{\includegraphics{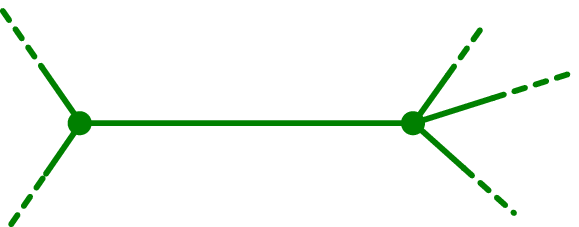}}}
\end{picture}}
$~\underset{\text{fig.\,\ref{fig:edge-add}}}{\longleftrightarrow}$
\raisebox{-12pt}{\begin{picture}(80,30)
  \put(0,0){\scalebox{.5}{\includegraphics{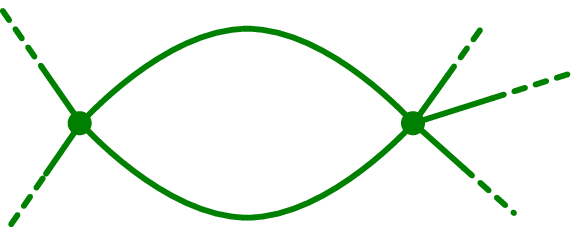}}}
\end{picture}}
$~\underset{\text{fig.\,\ref{fig:edge-vertex-add}}}{\longleftrightarrow}$
\raisebox{-12pt}{\begin{picture}(80,30)
  \put(0,0){\scalebox{.5}{\includegraphics{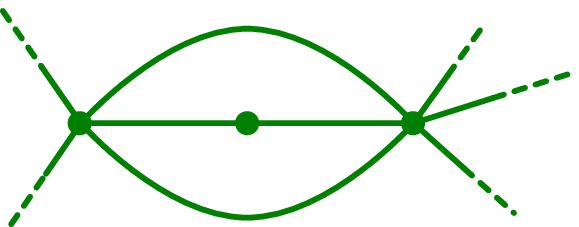}}}
\end{picture}}
$~\underset{\text{2{$\times$}fig.\,\ref{fig:edge-add}}}{\longleftrightarrow}$
\raisebox{-12pt}{\begin{picture}(80,30)
  \put(0,0){\scalebox{.5}{\includegraphics{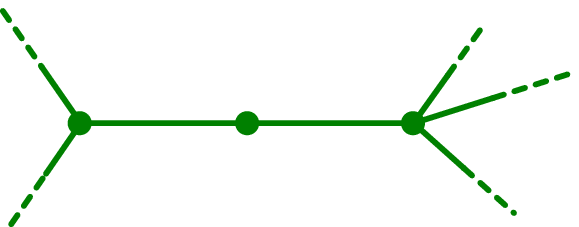}}}
\end{picture}}
\end{center}
\vspace*{-1em}
\caption{The moves in figures \ref{fig:edge-vertex-add} and \ref{fig:edge-add} allow one to split an edge by adding a new vertex.}
\label{fig:split-edge-new-vertex}
\end{figure}

One immediate consequence of the above lemma is that we can insert new vertices on edges which do not belong to the boundary of $M$ via the sequence of moves in figure \ref{fig:split-edge-new-vertex}. (The cell-decomposition of the boundary is fixed by the parametrisation in terms of the source and target objects). In the absence of domain walls, the `elementary subdivisions' (and their inverses) of a 2-cell (figure \ref{fig:edge-add}) and of a 1-cell (figure \ref{fig:split-edge-new-vertex}) have been shown in \cite{Kirillov:2010a} to relate any two cell decompositions (more precisely: two PLCW-decompositions of a compact polyhedron, see \cite{Kirillov:2010a} for details). The next lemma extends this to cell decomposition in the presence of domain walls and junctions; we will only sketch its proof.

\begin{lem} \label{lem:Tcw-cell-indep-fixed-bnd}
Let $\tilde U, \tilde V \in \bord^\mathrm{cw}$ and let  $\tilde M, \tilde M' : \tilde U \to \tilde V$ be two morphisms such that $F(\tilde M) = F(\tilde M')$. Then $\Tcw(\tilde M) = \Tcw(\tilde M')$. 
\end{lem}

\begin{proof}[Sketch of proof]

A 2-cell containing a junction is by construction the same in all cell decompositions, so we will remove such 2-cells from $\tilde M$ and $\tilde M'$ and treat their boundary edges as additional boundary components for the remaining cell complexes. 

Next use the moves in figures \ref{fig:edge-add} and \ref{fig:split-edge-new-vertex} to refine the cell decomposition of $\tilde M$ and $\tilde M'$ to a triangulation. The same moves allow us to make the part of the triangulation touched by the domain walls in $\tilde M$ and $\tilde M'$ agree: each component of the domain wall submanifold $M_1$ defines a string of triangles in $\tilde M$ and $\tilde M'$ and one can pass to a common refinement. (To achieve this refinement it is allowed to change the triangulation away from the domain walls.) We now remove all triangles containing domain walls from $\tilde M$ and $\tilde M'$, giving rise to yet more boundary components. 

The triangulations of $\tilde M$ and $\tilde M'$ still remaining do no longer contain any domain walls or junctions, and the standard proof of triangulation independence applies (see \cite{Kirillov:2010a} or just check that  figures \ref{fig:edge-add} and \ref{fig:split-edge-new-vertex} imply invariance under the Pachner moves). 
\end{proof}

The next step is to study the behaviour of $\Tcw$ on preimages of cylinders under $F$. For $U \in \bord$ denote by $C_U : U \to U$ the morphism given by the cylinder over $U$. That is, $C_U = U \times [-1,1]$ with decomposition induced by that of $U$ via $C_U = (C_U)_2 \cup (C_U)_1 \cup (C_U)_0$ with $(C_U)_i = U_{i-1} \times [-1,1]$, $i=2,1$ and $(C_U)_0=\emptyset$. Orientations and labellings are induced by $U$ as well. Note that because morphisms of $\bord$ are diffeomorphism classes, we have
\be
  C_U \circ C_U = C_U 
\ee
as morphisms in $\bord$. Now pick objects $\tilde U, \tilde U'$ with $F(\tilde U) = F(\tilde U') = U$ and a morphism $\tilde M : \tilde U \to \tilde U'$ with $F(\tilde M) = C_U$. Define
\be \label{eq:zeta_UU-def}
  \zeta_{\tilde U',\tilde U} := \Tcw(\tilde M) : \Tcw(\tilde U) \longrightarrow \Tcw(\tilde U') \ .
\ee
By lemma \ref{lem:Tcw-cell-indep-fixed-bnd}, $\zeta_{\tilde U,\tilde U'}$ is independent of $\tilde M$. As a consequence, given another preimage $\tilde U''$ of $U$, we have
\be \label{eq:zeta.zeta=zeta}
  \zeta_{\tilde U'',\tilde U'} \circ \zeta_{\tilde U',\tilde U} = \zeta_{\tilde U'',\tilde U} \ .
\ee
For an object $O \in \bord$ consisting of a single component $S^1$ write $\dwt x(O)$ for the list $((x_1,\eps_1),\dots,(x_n,\eps_n))$ of the marked points on $O$ together with their orientation, ordered clockwise starting from the point $-1 \in S^1$. If $\tilde O$ is a preimage of $O$, denote by $e_1,\dots, e_m$ the 1-cells, again ordered clockwise starting from $-1 \in S^1$. For example, figure \ref{fig:object+morph-to-vsp-for-Tcw} shows a preimage $\tilde O$ with $m\,{=}\,5$ 1-cells, for $O$ a circle with $\dwt x(O) = ((y,+),(x,+),(z,-))$ so that $n\,{=}\,3$.
Consider the maps
\be \label{eq:e-tildeS-def}
  e(\tilde O) \,:= \Big( \cytens{A_{n,1}}X_{\dwt x(O)} 
  \xrightarrow{~\sim~}
  ~\cytens{A_{m,1}} R_{e_1} \otimes_{A_{1,2}} R_{e_2} \otimes_{A_{2,3}} \cdots \otimes_{A_{m-1,m}}  R_{e_m}
  \xrightarrow{e_\otimes}
  \Tcw(\tilde O) \, \Big)
\ee
and
\be\label{eq:pi-tildeS-def}
  \pi(\tilde O) \, := \Big( \,
  \Tcw(\tilde O) 
  \xrightarrow{\pi_\otimes}
  ~\cytens{A_{m,1}} R_{e_1} \otimes_{A_{1,2}} R_{e_2} \otimes_{A_{2,3}} \cdots \otimes_{A_{m-1,m}}  R_{e_m}
  \xrightarrow{~\sim~}
  \cytens{A_{n,1}}X_{\dwt x(O)}  \Big)
\ee
Here, $X_{\dwt x}$ is the notation introduced in \eqref{eq:X-dwtx-def}, $R_e$ was defined in \eqref{eq:Re-def} and the intermediate algebras $A_{i,i+1}$ are as required by the bimodules. By \eqref{eq:Tcw-on-single-S1}, $\Tcw(\tilde O)$ consists precisely of the tensor factors $R_{e_1} \otimes \cdots  \otimes R_{e_m}$, and the maps $e_\otimes, \pi_\otimes$ are as in \eqref{eq:er-def-multiple-tensor}. 

\begin{lem} \label{lem:e-r-and-Tcw-relation}
Let $\tilde O$, $\tilde O'$ in $\bord^\mathrm{cw}$ be preimages of $O$. Then 
$\pi(\tilde O) \circ e(\tilde O) = \id_{\cytens{A_{n,1}} X_{\dwt x(O)}}$
and 
$e(\tilde O') \circ \pi(\tilde O) = \zeta_{\tilde O',\tilde O}$.
\end{lem}

\begin{proof}
The first equality is the defining property of the maps $\pi_\otimes$ and $e_\otimes$ in \eqref{eq:er-def-multiple-tensor}. 

Let $\tilde C_{\tilde U}$ be the cylinder over $\tilde U$ obtained by equipping $C_U$ with the cell decomposition induced by that of $\tilde U$: each edge $e$ of $\tilde U$ gets extended to the square 2-cell $e \times [-1,1]$. If we apply the functor $\Tcw$ to $\tilde C_{\tilde O}$, a short calculation starting from the definition of $\Tcw$ (illustrated in the first example in section \ref{sec:example-amplitude} below) shows that
\be \label{eq:e-pi-prop-proof-aux2}
  \zeta_{\tilde O,\tilde O} = p_\otimes \ ,
\ee
where $p_\otimes$ is the idempotent on $X_{x_1}^{\eps_1} \otimes \cdots \otimes X_{x_n}^{\eps_n}$ whose image is $\cytens{A_{n,1}}X_{\dwt x(O)}$. Below we will furthermore check that 
\be \label{eq:eq:e-pi-prop-proof-aux3}
  \pi(\tilde O') \circ \zeta_{\tilde O',\tilde O} \circ e(\tilde O) = \id_{\cytens{A_{n,1}} X_{\dwt x(O)}} \ .
\ee
Composing this from the left with $e(\tilde O')$ and from the right with $\pi(\tilde O)$, and using $e_\otimes \circ \pi_\otimes = p_\otimes$ together with \eqref{eq:e-pi-prop-proof-aux2} and \eqref{eq:zeta.zeta=zeta}, proves the second equality of the lemma. 

Let us now sketch the proof of \eqref{eq:eq:e-pi-prop-proof-aux3}. We identify $\cytens{A_{n,1}}X_{\dwt x(O)}$ and $\cytens{A_{m,1}} R_{e_1} \otimes_{A_{1,2}} \cdots \otimes_{A_{m-1,m}}  R_{e_m}$ with the images of the corresponding projectors $p_\otimes$ in $X_{x_1}^{\eps_1} \otimes \cdots \otimes X_{x_n}^{\eps_n}$ and
$R_{e_1} \otimes \cdots \otimes R_{e_m} \equiv \Tcw(\tilde O)$.
Let $\sum_i x_1^{(i)} \otimes \cdots \otimes x_n^{(i)}$ be an element of $X_{x_1}^{\eps_1} \otimes \cdots \otimes X_{x_n}^{\eps_n}$ in the image of $p_\otimes$. The first arrow in \eqref{eq:e-tildeS-def} is the isomorphism mapping this to the element
\be 
  v = p_\otimes \circ \Big(\sum_i 
  1_{A_{n,1}} \otimes \cdots \otimes x_1^{(i)} \otimes 1_{A_{1,2}} \otimes \cdots \otimes x_2^{(i)}  \otimes 1_{A_{2,3}} \otimes \cdots \otimes  x_n^{(i)} \otimes  \cdots  \otimes 1_{A_{n,1}} \Big)
\ee
of $R_{e_1} \otimes \cdots \otimes R_{e_m} \equiv \Tcw(\tilde O)$. Here one unit element has been inserted for each factor $R_{e_k}$ for which $e_k$ does not contain a marked point (in this case $R_{e_k} = A_a$ for an appropriate $a$, cf.\ \eqref{eq:Re-def}). One can convince oneself that $\zeta_{\tilde O',\tilde O}$ maps $v$ to an element $v'$ of the same form in $\Tcw(\tilde O)$ (i.e.\ $v'$ has same factors $x_k^{(i)}$ but possibly a different number of factors $1_{A_{k,k+1}}$). We omit the details of this step. The final isomorphism in \eqref{eq:pi-tildeS-def} maps $v'$ back to $\sum_i x_1^{(i)} \otimes \cdots \otimes x_n^{(i)}$.
\end{proof}

In the last step, we define the sought-after functor $T$. On objects $O \in \bord$ with a single $S^1$ component, we set 
\be \label{eq:def-TFT-T-on-S1}
  T(O) := \, \cytens{A_{m,1}} X_{\dwt x(O)} \ .
\ee   
For $U = O_1 \sqcup \cdots \sqcup O_n$, monoidality then requires $T(U) = T(O_1) \otimes \cdots \otimes T(O_n)$. For a bordism $M : U \to V$ in $\bord$ pick a preimage $\tilde M : \tilde U \to \tilde V$ under the forgetful functor. Extend the definition of $e(\tilde O)$ and $\pi(\tilde O)$ to $\tilde U$ by taking tensor products. Define
\be \label{eq:TM-def}
  T(M) := \Big(
  T(U) \xrightarrow{e(\tilde U)} \Tcw(\tilde U) \xrightarrow{\Tcw(\tilde M)} \Tcw(\tilde V) \xrightarrow{\pi(\tilde V)} T(V)
  \Big) \ .
\ee
The first main result of this paper is:

\begin{thm} \label{thm:lattice-defect-TFT}
(i) $T(M)$ is independent of the choice of preimage $\tilde M$ of $M$.
\\[.3em]
(ii) $T(C_U) = \id_{T(U)}$.
\\[.3em]
(iii) $T : \bord_{2,1}^\mathrm{def,top} \to \vect_f(k)$ is a symmetric monoidal functor.
\end{thm}

\begin{proof}
(i) Choose another preimage  $\tilde M' : \tilde U' \to \tilde V'$ in $\bord^\mathrm{cw}$ of $M : U \to V$, and choose preimages $\tilde C_U : \tilde U \to \tilde U'$ and $\tilde C_V : \tilde V \to \tilde V'$ of the cylinder $C_U$ and $C_V$. Consider the diagram
\be
\raisebox{2.3em}{\xymatrix@R=.5em{
& \Tcw(\tilde U) \ar[dd]^{\zeta_{\tilde U',\tilde U}} \ar[rr]^{\Tcw(\tilde M)}
&& \Tcw(\tilde V) \ar[dd]^{\zeta_{\tilde V',\tilde V}} \ar[dr]^{\pi(\tilde V)} 
\\
\cytens{A_{m,1}} X_{\dwt x(U)}
\ar[ur]^{e(\tilde U)} 
\ar[dr]_{e(\tilde U')} 
&&&&
\cytens{A_{n,1}} X_{\dwt x(V)}
\\
& \Tcw(\tilde U') \ar[rr]^{\Tcw(\tilde M')} &&
\Tcw(\tilde V') \ar[ur]_{ \pi(\tilde V')} 
}}
\quad .
\ee
To see that the left triangle commutes, substitute $\zeta_{\tilde U',\tilde U} = e(\tilde U') \circ \pi(\tilde U)$ and use that $\pi(\tilde U) \circ e(\tilde U) = \id$ (lemma \ref{lem:e-r-and-Tcw-relation}). Commutativity of the right triangle follows analogously. The following chain of equalities shows that also the central square commutes:
\be\label{eq:T-indep-of-M-proof}
\begin{array}{ll}
\zeta_{\tilde V',\tilde V} \circ \Tcw(\tilde M) 
&\overset{(1)}= \Tcw(\tilde C_V) \circ \Tcw(\tilde M)
\overset{(2)}= \Tcw(\tilde C_V \circ  \tilde M)
\\[.5em]
&\overset{(3)}= \Tcw( \tilde M' \circ \tilde C_U )
\overset{(4)}= \Tcw(\tilde M') \circ \Tcw(\tilde C_U)
\overset{(5)}= \Tcw(\tilde M') \circ \zeta_{\tilde U',\tilde U} \ .
\end{array}
\ee
Step (1) is the definition of $\zeta_{\tilde V',\tilde V}$ in \eqref{eq:zeta_UU-def}; step (2) is functoriality of $\Tcw$; step (3) follows from lemma \ref{lem:Tcw-cell-indep-fixed-bnd} since $\tilde C_V \circ  \tilde M$ and $ \tilde M' \circ \tilde C_U$ are just different cell decompositions (but identical on the boundary) of the same bordism $\tilde U \to \tilde V'$; steps (4) and (5) are the same as (2) and (1).
Thus the diagram \eqref{eq:T-indep-of-M-proof} commutes, establishing (i).
\\[.3em]
(ii) By definition \eqref{eq:TM-def} and lemma \eqref{lem:e-r-and-Tcw-relation}, $T(C_U) = \pi(\tilde U') \circ \Tcw(\tilde C_U) \circ e(\tilde U) = \pi(\tilde U') \circ \zeta_{\tilde U',\tilde U} \circ e(\tilde U) =  \pi(\tilde U') \circ e(\tilde U') \circ \pi(\tilde U) \circ e(\tilde U) = \id$.
\\[.3em]
(iii) Let $U \xrightarrow{M} V \xrightarrow{N} W$ be two composable morphisms in $\bord$ and choose a preimage $\tilde U \xrightarrow{\tilde M} \tilde V \xrightarrow{\tilde N} \tilde W$ in $\bord^\mathrm{cw}$. To check compatibility with composition, we need to show $T(N \circ M) = T(N) \circ T(M)$. Inserting the definition, this amounts to
\be
  \pi(\tilde W) \circ \Tcw(\tilde N \circ \tilde M) \circ e(\tilde U)
  =
  \pi(\tilde W) \circ \Tcw(\tilde N) \circ e(\tilde V) \circ \pi(\tilde V) \circ \Tcw(\tilde M) \circ e(\tilde U)
\ee
That the two sides are indeed equal can be seen as follows.  By lemma \ref{lem:e-r-and-Tcw-relation}, $e(\tilde V) \circ \pi(\tilde V) = \zeta_{\tilde V,\tilde V}$, and, if $\tilde C$ is a preimage of $C_V$, by functoriality of $\Tcw$ the rhs is equal to $ \pi(\tilde W) \circ \Tcw(\tilde N \circ \tilde C \circ \tilde M) \circ e(\tilde U)$. But $F(\tilde N \circ \tilde C \circ \tilde M) = N \circ M$, so that by lemma \ref{lem:Tcw-cell-indep-fixed-bnd}, the rhs is indeed equal to the lhs. 

Monoidality and symmetry of $T$ are implied by that of $\Tcw$.
\end{proof}

This concludes our construction of an example of a two-dimensional topological field theory with defects. 

\subsection{Some examples of amplitudes}\label{sec:example-amplitude}

\begin{figure}[tb] 
\raisebox{60pt}{a)} \hspace{-2em}
\raisebox{-38pt}{\begin{picture}(80,80)
  \put(0,0){\scalebox{.45}{\includegraphics{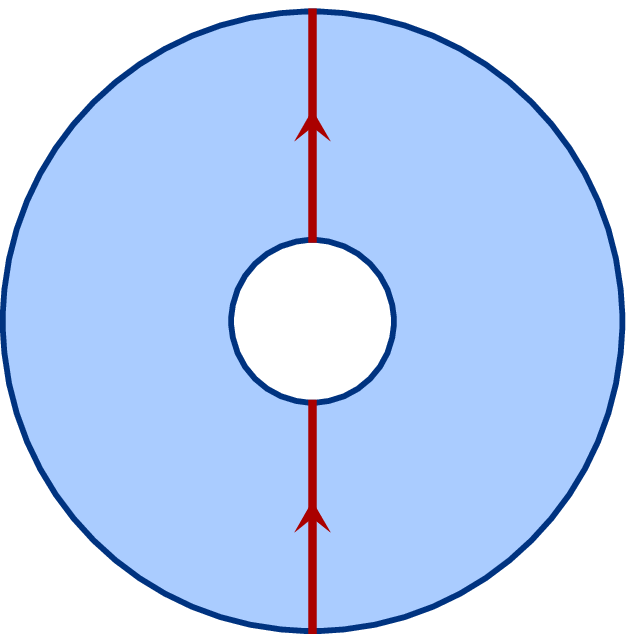}}}
  \put(0,0){
     \setlength{\unitlength}{.45pt}\put(-211,-323){
     \put(240,409)   {\scriptsize $b$ }
     \put(352,409)   {\scriptsize $a$ }
     \put(284,464)   {\scriptsize $y$ }
     \put(284,358)   {\scriptsize $x$ }
  }}
\end{picture}}
\hspace{1em}
\raisebox{-57pt}{\begin{picture}(122,122)
  \put(0,0){\scalebox{.7}{\includegraphics{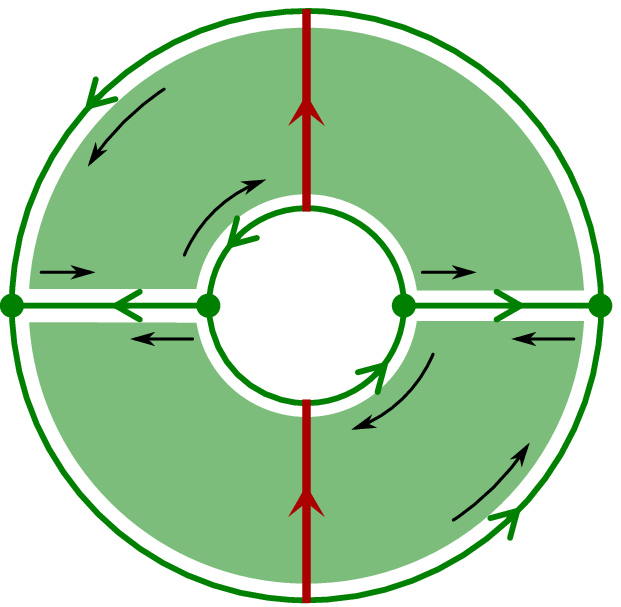}}}
  \put(0,0){
     \setlength{\unitlength}{.7pt}\put(-213,-327){
     \put(258,370)   {\scriptsize $p_1$ }
     \put(330,454)   {\scriptsize $p_2$ }
     \put(305,431)   {\scriptsize $e_1$ }
     \put(284,398)   {\scriptsize $e_2$ }
     \put(249,424)   {\scriptsize $e_3$ }
     \put(359,424)   {\scriptsize $e_4$ }
     \put(222,471)   {\scriptsize $e_5$ }
     \put(367,351)   {\scriptsize $e_6$ }
  }}
\end{picture}}
\hspace{2em}
\raisebox{60pt}{b)} \hspace{-2em}
\raisebox{-38pt}{\begin{picture}(80,80)
  \put(0,0){\scalebox{.45}{\includegraphics{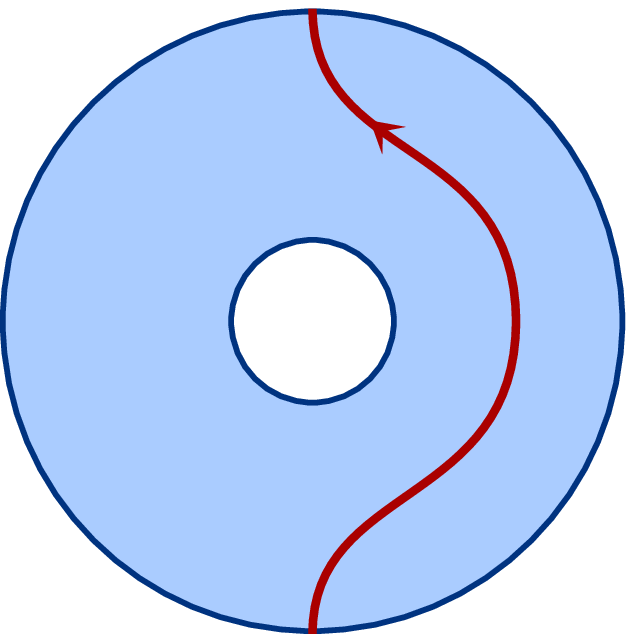}}}
  \put(0,0){
     \setlength{\unitlength}{.45pt}\put(-211,-323){
     \put(305,460)   {\scriptsize $x$ }
     \put(240,409)   {\scriptsize $b$ }
     \put(370,409)   {\scriptsize $a$ }
  }}
\end{picture}}
\hspace{1em}
\raisebox{-57pt}{\begin{picture}(122,122)
  \put(0,0){\scalebox{.7}{\includegraphics{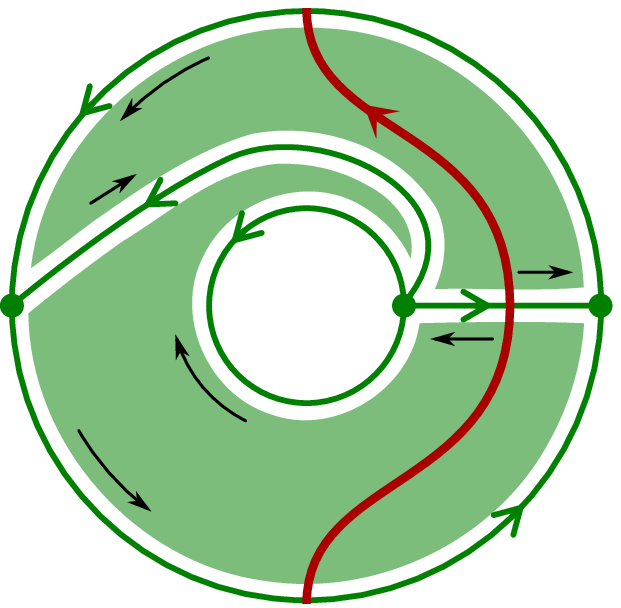}}}
  \put(0,0){
     \setlength{\unitlength}{.7pt}\put(-213,-327){
     \put(253,375)   {\scriptsize $p_1$ }
     \put(282,475)   {\scriptsize $p_2$ }
     \put(286,427)   {\scriptsize $e_1$ }
     \put(244,425)   {\scriptsize $e_2$ }
     \put(343,424)   {\scriptsize $e_3$ }
     \put(220,474)   {\scriptsize $e_4$ }
     \put(368,347)   {\scriptsize $e_5$ }
  }}
\end{picture}}
\caption{Two bordisms with defects together with a choice of cell-decomposition used in the sample computation. As in figure \ref{fig:annulus-cell-decomp}, the orientations of the edges are chosen for convenience and are not part of the data of the cell decomposition.}
\label{fig:amplitude-examples}
\end{figure}

Let us work through two more examples to see how the amplitude of a bordism
\be
  M : U \to V 
\ee
in $\bord \equiv \bord_{2,1}^\mathrm{def,top}$
is computed in lattice TFT. As in section \ref{sec:T^cw-cell-indep}, we denote by $\tilde M : \tilde U \to \tilde V$ a lift to 
$\bord^\mathrm{cw} \equiv \bord_{2,1}^\mathrm{def,top,cw}$.

\medskip

The first example is shown in figure \ref{fig:amplitude-examples}\,a). Using the notation of section \ref{sec:2-cat-from-defect-QFT}, let $U = V = O(y \circ x^*)$ be the object of $\bord$ consisting of a single $S^1$ with two marked points $(x,-)$ and $(y,+)$. For $\tilde U = \tilde V$ we choose a decomposition with two 1-cells. The spaces $R_e$ are:
\begin{quote}
\begin{tabular}{c|cccc}
$e$ & $e_1$  & $e_2$  & $e_5$  & $e_6$ \\
\hline
$R_e$ & $X_y$ & $X_x^*$ & $X_y$ & $X_x^*$ 
\end{tabular}
\end{quote}
The allowed triples are:
\begin{quote}
\begin{tabular}{c|cccc}
$(p,e,\mathrm{or})$ & $(p_1,e_2,-)$ &  $(p_1,e_3,+)$ &  $(p_1,e_6,+)$ &  $(p_1,e_4,-)$ 
\\
\hline
$Q_{p,e,\mathrm{or}}$ & $X_x^*$ & $A_b$ & $X_x$ & $A_a$ 
\end{tabular}
\\[.5em]
\begin{tabular}{c|cccc}
$(p,e,\mathrm{or})$ &  $(p_2,e_1,-)$ &  $(p_2,e_4,+)$ &  $(p_2,e_5,+)$ & $(p_2,e_3,-)$ 
\\
\hline
$Q_{p,e,\mathrm{or}}$ & $X_y$ & $A_a$ & $X_y^*$ & $A_b$ 
\end{tabular}
\end{quote}
We now evaluate $\Tcw(\tilde M)$ as given in \eqref{eq:Tcw-as-two-maps}, which is a linear map from $R_{e_1} \otimes R_{e_2}$ to $R_{e_5} \otimes R_{e_6}$, both of which spaces are equal to $X_y \otimes X_x^*$.
The map $\id \otimes P$ in \eqref{eq:Tcw-as-two-maps} maps the element $w \otimes \varphi \in R_{e_1} \otimes R_{e_2}$ to (not writing all `$\otimes$'-symbols)
$$ 
\begin{array}{l@{}c@{}c@{}c@{}c@{}c@{}c@{}c@{}c@{}c@{}c@{}c@{}c@{}c@{}c@{}c@{}c@{}c@{}c@{}c}
\longrightarrow &(R_{e_1} && R_{e_2}) &&
(Q_{p_1,e_3,+} && Q_{p_1,e_6,+} && Q_{p_1,e_4,-} && 
Q_{p_2,e_4,+} && Q_{p_2,e_5,+} &&  Q_{p_2,e_3,-}) && 
(R_{e_5} && R_{e_6})
\\
\displaystyle
\mapsto \sum_{i,j,k,l} & (w & \otimes& \varphi)&\otimes&( b_i' &\otimes& u_j &\otimes& a_k' &\otimes& a_k &\otimes& v_l^* &\otimes& b_i)&\otimes&(v_l &\otimes& u_j^*)
\end{array}
$$
which in turn gets mapped by $E \otimes \id$ to 
\be
\sum_{i,j,k,l} \varphi(b_i'.u_j.a_k') \, v_l^*(b_i.w.a_k) \, v_l \otimes u_j^*
\ee
in $R_{e_5} \otimes R_{e_6}$. This can be simplified by carrying out the sum over the bases $u_j$, $v_l$ and their duals, resulting in
\be
  \Tcw(\tilde M)(w \otimes \varphi) = \sum_{i,k}(b_i.w.a_k) \otimes (a_k'.\varphi.b_i') \ .
\ee
Comparing to the discussion in section \ref{sec:alg-prel}, we see that this is nothing but the projector $p_\otimes$ on $X_y \otimes X_x^*$ whose image is $\cytens{A_b} X_y \otimes_{A_a} X_x^*$. Combining this with the definition \eqref{eq:TM-def} of $T$, we see that $T(M) = \id$ on $\cytens{A_b} X_y \otimes_{A_a} X_x^*$. This illustrates point (ii) of theorem \ref{thm:lattice-defect-TFT}.

The following lemma provides a different point of view on this result which will be useful in understanding the 2-category of defect conditions. Denote by $\mathrm{ev}_V : V^* \otimes V \to k$ the evaluation of a vector space on its dual and  write $\Hom_{A|B}(X,Y)$ for the space of bimodule maps between two $A$-$B$-bimodules $X,Y$.

\begin{lem} \label{lem:YxX*-Hom(X,Y)}
Let $A,B$ be Frobenius algebras with trace pairing, and let $X,Y$ be finite dimensional $A$-$B$-bimodules. The map $\phi :\, \cytens A Y \otimes_B X^*  \to \Hom_{A|B}(X,Y)$, 
\be
  \phi(\gamma) ~:=~ \Big(
  X 
  \xrightarrow{\gamma \otimes \id_X} 
  (\cytens A Y \otimes_B X^*) \otimes X
  \xrightarrow{e_\otimes \otimes \id_X} 
  Y \otimes X^* \otimes X
  \xrightarrow{\id_X \otimes \mathrm{ev}_X} 
  Y
  \Big)
\ee
is an isomorphism.
\end{lem}

\begin{proof} 
Let $\gamma \in \, \cytens A Y \otimes_B X^*$. That $\phi(\gamma)$ is a bimodule map is a straightforward calculation using $\mathrm{ev}_X(u \otimes (a.v.b)) = \mathrm{ev}_X((b.u.a) \otimes v)$ and $(e_\otimes \circ \gamma).a = a.(e_\otimes \circ \gamma)$. That $\phi$ is an isomorphism follows by the standard argument using the corresponding coevaluation map and the duality properties.
\end{proof}

Via this lemma we can also think of $T(O(y \circ x^*))$ as $\Hom_{A_b|A_a}(X_x,X_y)$. If we identify $X_y \otimes X_x^*$ with $\Hom(X_x,X_y)$ then $\Tcw(\tilde M)$ becomes the projection from general linear maps to bimodule intertwiners.

\medskip

The second example -- figure \ref{fig:amplitude-examples}\,b) -- is again an annulus with a domain wall, but this time the in-going boundary sits entirely in one domain. Here, the source-object is $U = O(b)$, an $S^1$ with no marked points and labelled $b \in D_2$, and the target object is $V = O(x \circ x^*)$. The lift $\tilde U$ we chose contains a single edge, while $\tilde V$ contains two edges. The spaces $R_e$ and $Q_{p,e,\mathrm{or}}$ in this case are:
\begin{quote}
\begin{tabular}{c|ccc}
$e$ & $e_1$  & $e_4$  & $e_5$ \\
\hline
$R_e$ & $A_b$ & $X_x$ & $X_x^*$ 
\end{tabular}
\\[.5em]
\begin{tabular}{c|c@{\,}c@{\,}c@{\,}c@{\,}c@{\,}c@{\,}c}
$(p,e,\mathrm{or})$ & $(p_1,e_1,-)$ &  $(p_1,e_2,+)$ &  $(p_1,e_5,+)$ &  $(p_1,e_3,-)$
 &  $(p_2,e_2,-)$ &  $(p_2,e_3,+)$ &  $(p_2,e_4,+)$  
\\
\hline
$Q_{p,e,\mathrm{or}}$ & $A_b$ & $A_b$ & $X_x$ & $X_x^*$ & $A_b$ & $X_x$ & $X_x^*$ 
\end{tabular}
\end{quote}
The map $\id \otimes P$ takes $w \in A_b$ to
$$ 
\begin{array}{l@{}c@{}c@{}c@{}c@{}c@{}c@{}c@{}c@{}c@{}c@{}c@{}c@{}c@{}c@{}c@{}c@{}c@{}c@{}c}
\longrightarrow &R_{e_1} && 
(Q_{p_1,e_3,-} && Q_{p_1,e_2,+} && Q_{p_1,e_5,+} &&  
Q_{p_2,e_4,+} && Q_{p_2,e_2,-} &&  Q_{p_2,e_3,+}) && 
(R_{e_4} && R_{e_5}) 
\\
\displaystyle
\mapsto \sum_{i,j,k,l} & w &\otimes&( u_i^* &\otimes& b_j'  &\otimes& u_k &\otimes& u_l^* &\otimes& b_j &\otimes& u_i)&\otimes&(u_l &\otimes& u_k^*) ~~.
\end{array}
$$
This in turn is mapped to $\sum_{i,j,k,l} u_i^*((w \cdot b_j').u_k)\,u_l^*(b_j.u_i)\cdot u_l \otimes u_k^*$ by $E \otimes \id$, which simplifies to
\be
  \Tcw(\tilde M)(w) = 
  \sum_{j,k} \big((b_j w b_j').u_k \big) \otimes u_k^* 
  =
  \sum_{k} (p_\otimes(w).u_k) \otimes u_k^*  \ ,
\ee
where $p_\otimes : A_b \to A_b$ is the projection to the centre of $A_b$, see lemma \ref{lem:ho-hom=ho-cohom}.
Accordingly, the resulting map for the bordism $M$ is
\be
\begin{array}{rccc}
  T(M) : &Z(A_b)  &\longrightarrow& \Hom_{A_b|A_a}(X_x,X_x)
\\
& z  &\longmapsto& (q \mapsto z.q) \ ,
\end{array}
\ee
where we have identified $\cytens{A_b} X_x \otimes_{A_a} X_x^* \cong \Hom_{A_b|A_a}(X_x,X_x)$ via lemma \ref{lem:YxX*-Hom(X,Y)}.

\subsection{Bicategory of algebras and 2-category of defect conditions} \label{sec:bicat-alg-2cat-def}

In the construction in section \ref{sec:lattice-TFT-data}--\ref{sec:T^cw-cell-indep}, the domain wall conditions were given by bimodules. Bimodules naturally form a bicategory (see \cite[Sec.\,2.5]{Benabou:1967}, \cite[Sec.\,I.3]{Gray:1974} or \cite[Ch.\,XII.7]{MacLane-book}), and in this subsection we want to compare this bicategory to the 2-category of defect conditions described in section \ref{sec:2-cat-from-defect-QFT}. Our conventions for bicategories can be found in appendix \ref{app:bicategories}.

\begin{defn}
(i) The category $\alg(k)$ has associative unital algebras over $k$ as objects and (unital) algebra homomorphisms as morphisms.
\\[.3em]
(ii) The bicategory $\Alg(k)$ has associative unital algebras over $k$ as objects. The morphism category $\Alg(k)(A,B)$ is given by the category of $B$-$A$-bimodules and bimodule intertwiners. The composition functor $\Alg(k)(B,C) \times \Alg(k)(A,B) \to \Alg(k)(A,C)$ is $(-)\otimes_B(-)$.
\end{defn}

We will start with a small digression which is not restricted to Frobenius algebras with trace pairing. Namely, we will look at some properties of $\Alg(k)$.

\medskip

Given a 1-category $\mathcal{C}$, we denote the bicategory obtained from $\mathcal{C}$ by adding only identity 2-morphisms again by $\mathcal{C}$. When comparing $\alg(k)$ and $\Alg(k)$, we understand $\alg(k)$ as a bicategory in this sense. For an algebra map $f : A \to B$ and a right $B$-module $M$, we denote by $M_f$ the right $A$-module with action $(m,a) \mapsto m.f(a)$. In particular, $B_f$ is a $B$-$A$-bimodule. The next lemma (following \cite[Sec.\,5.7]{Benabou:1967}) makes precise the idea that $\Alg(k)$ contains more 1- and 2-morphisms than $\alg(k)$.\footnote{
Note that, while $\alg(k)(A,B)$ is not additive (since $f+g$ is never an algebra homomorphism if $f$ and $g$ are), the category $\Alg(k)(A,B)$ has direct sums of 1-morphisms, so that we have added enough morphisms to `linearise' $\alg(k)$.}

\begin{lem} \label{lem:alg->Alg}
(i) Let $A \xrightarrow{f} B \xrightarrow{g} C$ be algebra maps. The following map is well-defined and an isomorphism of $C$-$A$-bimodules:
\be \label{eq:m_gf-def-Alg-Alg}
  m_{g,f} : C_g \otimes_B B_f \longrightarrow C_{g \circ f}
  \quad , \quad
  c \otimes_B b \longmapsto c \cdot g(b) \ .
\ee
(ii) The assignment
\be
  i : \alg(k) \longrightarrow \Alg(k) \ ,
\ee   
which is the identity on objects and which maps $A\xrightarrow{f}B$ to $B_f$, is a (non-lax) functor. The unit transformations are identities and the multiplication transformations are given by $m_{g,f}$.
\\[.3em]
(iii) Let $f,g : A \to B$ be algebra maps. Then $i(f)$ and $i(g)$ are 2-isomorphic in $\Alg(k)$ if and only if $f(-) = u \cdot g(-) \cdot u^{-1}$ for some $u \in B^\times$.
\end{lem}

\begin{proof} Abbreviate $m \equiv m_{g,f}$. 
\\[.3em]
(i) To see that $m$ is well-defined, consider the map $\bar m : C_g \otimes B_f \to C_{g \circ f}$ given by $u \otimes v \mapsto ug(v)$. We verify the cokernel condition: for $b \in B$ we have
$\bar m( (u.b) \otimes v) 
= \bar m( (ug(b)) \otimes v) 
= ug(b)g(v) = ug(bv) = \bar m(u \otimes (bv))$. 
Therefore, $\bar m$ induces a map $C_g \otimes_B B_f \to C_{g \circ f}$, which is precisely $m$. Since $c \mapsto c \otimes_B 1_B$ is an isomorphism from $C_{g \circ f}$ to $C_g \otimes_B B_f $, and since by composing with $m$ one obtains the identity on $C_g$, it follows that $m$ is an isomorphism. It is straightforward to check that $m$ intertwines the $C$-$A$-bimodule structures.
\\[.3em]
(ii) We have to verify associativity and unit properties of the functor. We start with associativity. Given algebra maps $A \xrightarrow{f} B \xrightarrow{g} C \xrightarrow{h} D$, we must show commutativity of the diagram
\be
\raisebox{4em}{\xymatrix{
( D_h \otimes_C C_g ) \otimes_B B_f
\ar[r]^\sim \ar[d]_{m_{h,g} \otimes_{B} \id} 
& 
D_h \otimes_C ( C_g \otimes_B B_f )
\ar[d]^{\id \otimes_{C} m_{g,f}} 
\\
D_{h\circ g} \otimes_{B} B_f 
\ar[d]_{m_{h\circ g,f}}
& D_h \otimes_{C} C_{g\circ f} 
\ar[d]^{m_{h ,  g \circ f}}
\\
D_{h\circ g \circ f} \ar[r]^= 
& D_{h\circ g \circ f} 
}}
\ee
Acting on an element $d \otimes_{C} c \otimes_{B} b$, the left brach gives
$d \cdot h(c) \cdot h(g(b))$ and the right branch gives $d \cdot h(c\cdot g(b))$. These are equal as $h$ is an algebra map. The unit properties in turn amount to commutativity of the following two diagrams:
\be
\raisebox{2em}{\xymatrix{
B \otimes_{B} B_f \ar[r]^(.6)\sim \ar[d]^{\id}
&
B_f
\\
B_\id \otimes_{B} B_f \ar[r]^(.6){m_{\id,f}}
&
B_{\id \circ f} \ar[u]_\id
}}
\quad,\quad
\raisebox{2em}{\xymatrix{
B_f \otimes_{A} A \ar[r]^(.6)\sim \ar[d]^{\id}
&
B_f
\\
B_f \otimes_{A} A_\id \ar[r]^(.6){m_{f,\id}}
&
B_{\id \circ f} \ar[u]_\id
}}
\ee
In the left diagram, both branches give the map $b \otimes_{B} b' \mapsto b \cdot b'$, and in the right diagram, both branches give $b \otimes_{B} a \mapsto b \cdot f(a)$.

Since by part (i) the $m$'s are isomorphisms, we do indeed obtain a functor, not just a lax functor. 
\\[.3em]
(iii) `$\Rightarrow$': Suppose that $\psi : B_f \to B_g$ is an isomorphism of $B$-$A$-bimodules. Then for all $x,b \in B$ and $a \in A$ we have $\psi(b\cdot x\cdot f(a)) = b \cdot \psi(x) \cdot g(a)$. From this we conclude that $f(a) \cdot \psi(1) = \psi(f(a)) =\psi(1) \cdot g(a)$. Since $\psi$ is invertible, $\psi(1) \in B^\times$, and so $f(a) = \psi(1) \cdot g(a) \cdot \psi(1)^{-1}$.

`$\Leftarrow$': The isomorphism is given by $b \mapsto b \cdot u$.
\end{proof}

Recall the construction of the 2-category $\bfD[D_2,D_1;T]$ in \eqref{eq:def-defect-2cat}, the assignment of algebras and bimodules to elements of $D_2$ and $D_1$ in the beginning of section \ref{sec:lattice-TFT-data}, and the definition of the defect TFT $T$ in theorem \ref{thm:lattice-defect-TFT}. We want to define a functor
\be \label{eq:Delta-B-Alg-def}
\Delta : \bfD[D_2,D_1;T] \longrightarrow \Alg(k) \ ,
\ee
which on objects $a \in D_2$ acts $\Delta(a) = A_a$ and on 1-morphisms $\dwt x \in \bfD(a,b)$ as $\Delta(\dwt x) = X_{\dwt x}$, using the notation \eqref{eq:X-dwtx-def}. The action on 2-morphisms will be described after the following remark.

\begin{rem} \label{rem:scale-trans-inv-in-TFT}
(i) By \eqref{eq:defect-2-cat-2morphs}, the 2-morphism spaces of $\bfD[D_2,D_1;T]$ are given by $\bfD_2(\dwt x,\dwt y) := H^\mathrm{inv}(\dwt y \circ \dwt x^*)$. One may think that in a TFT all states are scale and translation invariant, and this is true but for one detail. Let $\dwt x : a \to a$ and let $C_{O(\dwt x)}$ be the  cylinder over $O(\dwt x)$. The defining property \eqref{eq:scale-trans-inv-def} of a scale and translation invariant family implies that all vectors $\psi_{\dwt x;r}$ lie in the image of the idempotent $T(C_{O(\dwt x)}) : T(O(\dwt x)) \to T(O(\dwt x))$. Conversely, each vector in the image of $T(C_{O(\dwt x)})$ gives rise to a scale and translation invariant family. Indeed, for TFTs, $T(O(\dwt x)) \equiv T(O(\dwt x;r))$ is independent of $r$, and so is the family $\psi_{\dwt x;r}$. We can therefore identify $H^\mathrm{inv}(\dwt x)$ with the image of the idempotent $T(C_{O(\dwt x)})$. For our lattice TFT construction, by theorem \ref{thm:lattice-defect-TFT}\,(iii) this does not make a difference, but for a general TFT, $T(C_{O(\dwt x)})$ may be different from the identity map on $T(O(\dwt x))$.
\\[.3em]
(ii) Given a TFT for which the idempotents $T(C_U)$ for objects $U \in \bord_{2,1}^\mathrm{def,top}$ are not always identity maps, one can define a new TFT $T'$ in which one replaces all state spaces $T(U)$ by the image of the corresponding idempotent $T(C_U)$. The embedding of the image of $T(C_U)$ into $T(U)$ provides a monoidal natural transformation from $T'$ to $T$. One can think of $T'$ as the `non-degenerate subtheory' of $T$, because an amplitude $T(U \xrightarrow{M} V)$ vanishes if its argument comes from the kernel of $T(C_U)$. In principle, one can always work with non-degenerate TFTs, but in some situations degenerate TFTs are useful as an intermediate step (such as in the orbifold construction of \cite{Frohlich:2009gb}, or in a sense also the construction in section \ref{sec:T^cw-cell-indep}, where $T$ was defined precisely as the restriction of $\Tcw$ to images of idempotents).
\end{rem}

According to part (i) of the above remark, in our lattice TFT example we have $H^\mathrm{inv}(\dwt x) = T(O(\dwt x))$. Substituting the definition of $T$ on objects in \eqref{eq:def-TFT-T-on-S1}, we see that for $\dwt x, \dwt y : a \to b$,
\be
  \bfD_2(\dwt x,\dwt y) = \, \cytens{A_b} X_{\dwt y} \otimes_{A_a} X_{\dwt x^*} \ .
\ee
Using this and lemma \ref{lem:YxX*-Hom(X,Y)}, we can finally state the action of the functor $\Delta$ on morphisms. Namely for $u \in \bfD_2(\dwt x,\dwt y)$ we set $\Delta(u) = \phi(u) : X_{\dwt x} \to Y_{\dwt y}$. 

\medskip

We should now proceed to show that $\Delta$ thus defined is indeed a functor between bicategories, which in addition is locally fully faithful (since $\phi$ is an isomorphism). However, we will not go through these details and instead turn to the next topic, the relation between lattice TFT with defects and the centre of an algebra.

\section{The centre of an algebra} \label{sec:centre}

The map which assigns to an algebra $A$ its centre $Z(A)$ is not functorial, at least not in the obvious sense. Namely, given $A \in \alg(k)$, then also $Z(A) \in \alg(k)$, but for an algebra homomorphism $f : A \to B$ it is in general not true that $f|_{Z(A)}$ lands in $Z(B)$. For example, if $A$ is the algebra of diagonal $2{\times}2$ matrices, if $B$ is all $2{\times}2$-matrices and if $f$ is the embedding map, then $Z(A)=A$, but $Z(B) = k \, \id$ which does not contain $f(Z(A))$.

For Frobenius algebras with trace pairing one could use the maps $e_\otimes$ and $\pi_\otimes$ between $A$ and $Z(A) = \, \cytens{A}A$ (cf.\ lemma \ref{lem:ho-hom=ho-cohom} -- not true for general algebras) to map $f$ to $\pi_\otimes \circ f \circ e_\otimes$, but this would in general not be compatible with composition and multiplication. 

\medskip

The main point of this section is to define a functorial version of the centre. This is done by first constructing a bicategory -- or rather two versions thereof -- whose objects are commutative algebras. The centre is then a lax functor into this bicategory; this functor will also be given in two versions (theorem \ref{thm:alg->CAlg-laxfun} and remark \ref{rem:centre-2nd-version}). These constructions are motivated by 2-dimensional TFT with defects, so we begin the discussion by highlighting the relevant algebras and maps in the defect TFT.

\subsection{Spaces and maps associated to defect TFTs} \label{sec:spaces+maps}

Let $T : \bord_{2,1}^\mathrm{def,top}(D_2,D_1,D_0) \to \vect_f(k)$ be a defect TFT (not necessarily obtained via lattice TFT). The functor $T$ encodes an infinite number of state spaces and linear maps between them. In this subsection we will pick out some of the more fundamental ones and investigate their properties. 

By remark \ref{rem:scale-trans-inv-in-TFT}\,(ii) we are entitled to assume that all idempotents $T(C_U)$ are in fact identity maps, and we will make this assumption for the rest of this subsection.
Recall from section \ref{sec:2-cat-from-defect-QFT} the 2-category $\bfD[D_2,D_1;T]$ associated to a field theory with defects. By remark \ref{rem:scale-trans-inv-in-TFT}\,(i) and because of our assumption that $T(C_U) = \id_{T(U)}$, definition \eqref{eq:defect-2-cat-2morphs} of the 2-morphism spaces becomes
\be
  \bfD_2(\dwt x,\dwt y) = T(O(\dwt y \circ \dwt x^*)) \ .
\ee

Recall that the identity 1-morphism $\one_a : a \to a$, for $a \in D_2$, is the empty tuple $\one_a = ()$. Consider the space of 2-endomorphisms of $\one_a$,  $\bfD_2(\one_a,\one_a) = T(O(a))$. This is an associative, commutative, unital algebra; the bordisms which give the multiplication and unit morphisms are those in figure \ref{fig:id-hor-vert-comp}\,a,b), but without domain walls. Commutativity follows since precomposing the multiplication bordism with a transposition $\sigma : O(a) \sqcup O(a) \to O(a) \sqcup O(a)$ gives a diffeomorphic bordism. In fact, by the usual arguments, it is even a Frobenius algebra, and this Frobenius algebra defines the defect-free TFT given by $D_2 = \{a\}$ and $D_1 = D_0 = \emptyset$.

For an arbitrary 1-morphism $\dwt x : a \to b$, the 2-endomorphisms $\bfD_2(\dwt x,\dwt x)$ do still form an associative, unital algebra (even a Frobenius algebra), but this algebra need not be commutative. The horizontal composition functors for 
$(a \xrightarrow{\one_a} a \xrightarrow{\dwt x} b) =  a \xrightarrow{\dwt x} b$ 
and
$(a \xrightarrow{\dwt x} b \xrightarrow{\one_b} b) =  a \xrightarrow{\dwt x} b$ 
give linear maps
\be
  \hat R : \bfD_2(\dwt x,\dwt x) \otimes \bfD_2(\one_a,\one_a) \longrightarrow \bfD_2(\dwt x,\dwt x) ~~,~~
  \hat L :  \bfD_2(\one_b,\one_b) \otimes \bfD_2(\dwt x,\dwt x) \longrightarrow \bfD_2(\dwt x,\dwt x) \ .
\ee
The bordisms for the maps $\hat L$ and $\hat R$ are as in figure \ref{fig:id-hor-vert-comp}\,c), provided we specialise the latter to the case where only one of the two in-going boundary circles has domain walls attached to it. If we insert the identity 2-morphism $\id_{\dwt x}$, we obtain maps $R := \hat R(\id_{\dwt x} \otimes -) :  \bfD_2(\one_a,\one_a) \to \bfD_2(\dwt x,\dwt x)$ and $L := \hat L(- \otimes \id_{\dwt x}) : \bfD_2(\one_b,\one_b) \to \bfD_2(\dwt x,\dwt x)$. The corresponding bordisms are obtained by gluing a disc as in figure \ref{fig:id-hor-vert-comp}\,a) into the hole which has the domain walls attached. Figure \ref{fig:amplitude-examples}\,b) shows a bordism obtained in this way.

\begin{figure}[tb] 
$$
\begin{array}{l}
T\Bigg(
\raisebox{-26pt}{\begin{picture}(58,60)
  \put(0,0){\scalebox{.5}{\includegraphics{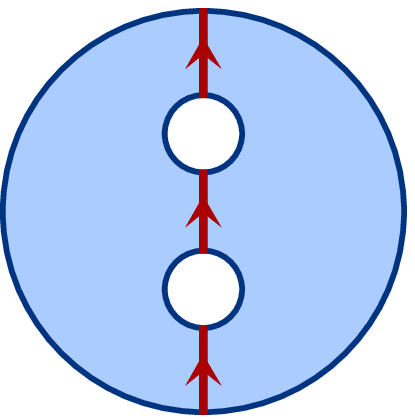}}}
  \put(0,0){
     \setlength{\unitlength}{.5pt}\put(-241,-356){
     \put(253,410)   {\scriptsize $b$ }
     \put(336,410)   {\scriptsize $a$ }
     \put(306,453)   {\scriptsize $x$ }
     \put(306,410)   {\scriptsize $x$ }
     \put(306,368)   {\scriptsize $x$ }
     \put(294,433)   {\scriptsize $1$ }
     \put(294,388)   {\scriptsize $2$ }
  }}
\end{picture}}
\Bigg)\!\Big(L(q),w\Big)
~~ = ~~
T\Bigg(
\raisebox{-26pt}{\begin{picture}(58,60)
  \put(0,0){\scalebox{.5}{\includegraphics{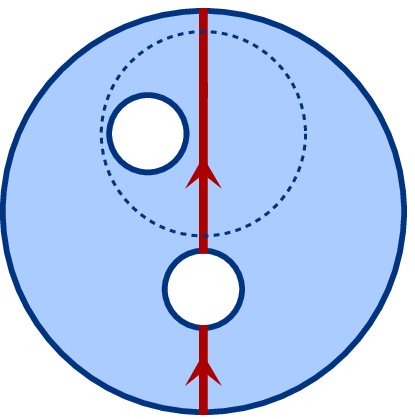}}}
  \put(0,0){
     \setlength{\unitlength}{.5pt}\put(-241,-356){
     \put(253,410)   {\scriptsize $b$ }
     \put(336,410)   {\scriptsize $a$ }
     \put(306,443)   {\scriptsize $x$ }
     \put(306,368)   {\scriptsize $x$ }
     \put(277,433)   {\scriptsize $1$ }
     \put(294,388)   {\scriptsize $2$ }
  }}
\end{picture}}
\Bigg)\!\Big(q,w\Big)
\\[2em]
\hspace{5em}
= ~~
T\Bigg(
\raisebox{-26pt}{\begin{picture}(58,60)
  \put(0,0){\scalebox{.5}{\includegraphics{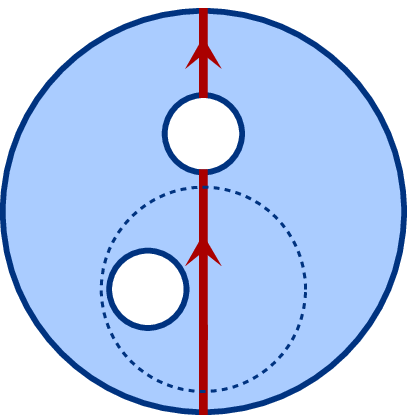}}}
  \put(0,0){
     \setlength{\unitlength}{.5pt}\put(-241,-356){
     \put(306,453)   {\scriptsize $x$ }
     \put(306,378)   {\scriptsize $x$ }
     \put(253,410)   {\scriptsize $b$ }
     \put(336,410)   {\scriptsize $a$ }
     \put(294,433)   {\scriptsize $2$ }
     \put(277,388)   {\scriptsize $1$ }
  }}
\end{picture}}
\Bigg)\!\Big(q,w\Big)
~~ = ~~
T\Bigg(
\raisebox{-26pt}{\begin{picture}(58,60)
  \put(0,0){\scalebox{.5}{\includegraphics{pic16a.eps}}}
  \put(0,0){
     \setlength{\unitlength}{.5pt}\put(-241,-356){
     \put(253,410)   {\scriptsize $b$ }
     \put(336,410)   {\scriptsize $a$ }
     \put(306,453)   {\scriptsize $x$ }
     \put(306,410)   {\scriptsize $x$ }
     \put(306,368)   {\scriptsize $x$ }
     \put(294,433)   {\scriptsize $2$ }
     \put(294,388)   {\scriptsize $1$ }
  }}
\end{picture}}
\Bigg)\!\Big(L(q),w\Big)
\end{array}
$$
\vspace*{-1em}
\caption{Manipulation of bordisms showing that $L \equiv \hat L(- \otimes \id_{\dwt x}): \bfD_2(\one_b,\one_b) \to\bfD_2(\dwt x,\dwt x)$ maps to the centre of $\bfD_2(\dwt x,\dwt x)$. Here $q \in \bfD_2(\one_b,\one_b)$ and $w \in \bfD_2(\dwt x,\dwt x)$.}
\label{fig:L-maps-to-centre}
\end{figure}

With the help of bordisms, it is easy to see that $R$ and $L$ are algebra homomorphisms whose image lies in the centre of $\bfD_2(\dwt x,\dwt x)$. The bordism manipulations showing that the image of $L$ lies in the centre are given in figure \ref{fig:L-maps-to-centre}.\footnote{Manipulating such disk-shaped bordisms reminds one of the string-diagram notation for 2-categories \cite{Street:1996}. Indeed a string-diagram identity implies an identity for defect correlators on disks, but the converse is not true -- the 2-category $\bfD[D_2,D_1;T]$ satisfies more conditions than a generic 2-category.}

\begin{figure}[tb] 
$$
\xymatrix@R=4em@C=8em{& 
T\Big(
\raisebox{-17pt}{\begin{picture}(25,42)
  \put(-1,6){\scalebox{.70}{\includegraphics{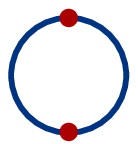}}}
  \put(-1,6){
     \setlength{\unitlength}{.70pt}\put(-280,-396){
     \put(297,422)   {\scriptsize $x$ }
     \put(297,405)   {\scriptsize $x$ }
     \put(300,437)   {\scriptsize $+$ }
     \put(300,390)   {\scriptsize $-$ }
     \put(307,412)   {\scriptsize $a$ }
     \put(287,412)   {\scriptsize $b$ }
  }}
\end{picture}}
\Big)
\ar[d]_(.55){
T\Bigg(
\raisebox{-26pt}{\begin{picture}(58,60)
  \put(0,0){\scalebox{.5}{\includegraphics{pic16a.eps}}}
  \put(0,0){
     \setlength{\unitlength}{.5pt}\put(-241,-356){
     \put(253,410)   {\scriptsize $b$ }
     \put(336,410)   {\scriptsize $a$ }
     \put(306,453)   {\scriptsize $y$ }
     \put(306,410)   {\scriptsize $x$ }
     \put(306,368)   {\scriptsize $x$ }
     \put(294,433)   {\scriptsize $f$ }
  }}
\end{picture}}
\Bigg)
} \\ 
T\Big(
\raisebox{-10pt}{\begin{picture}(26,27)
  \put(0,0){\scalebox{.70}{\includegraphics{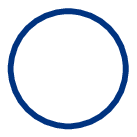}}}
  \put(0,0){
     \setlength{\unitlength}{.70pt}\put(-280,-396){
     \put(288,411)   {\scriptsize $b$ }
  }}
\end{picture}}
\Big)
\ar@/^2em/[ur]^(.4){
T\Bigg(
\raisebox{-17pt}{\begin{picture}(40,42)
  \put(0,0){\scalebox{.35}{\includegraphics{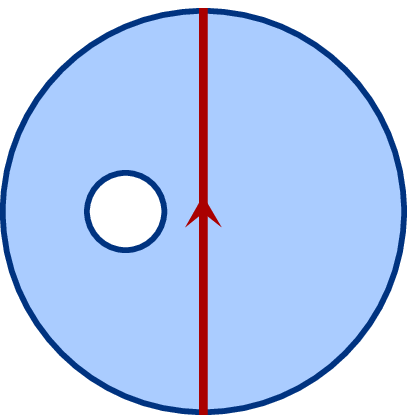}}}
  \put(0,0){
     \setlength{\unitlength}{.35pt}\put(-241,-356){
     \put(248,410)   {\scriptsize $b$ }
     \put(341,410)   {\scriptsize $a$ }
     \put(305,377)   {\scriptsize $x$ }
  }}
\end{picture}}
\Bigg)
} \ar@/_2em/[dr]_(.4){
T\Bigg(
\raisebox{-17pt}{\begin{picture}(40,42)
  \put(0,0){\scalebox{.35}{\includegraphics{pic17a.eps}}}
  \put(0,0){
     \setlength{\unitlength}{.35pt}\put(-241,-356){
     \put(248,410)   {\scriptsize $b$ }
     \put(341,410)   {\scriptsize $a$ }
     \put(305,377)   {\scriptsize $y$ }
  }}
\end{picture}}
\Bigg)
} & 
T\Big(
\raisebox{-17pt}{\begin{picture}(25,42)
  \put(-1,6){\scalebox{.70}{\includegraphics{pic17d.eps}}}
  \put(-1,6){
     \setlength{\unitlength}{.70pt}\put(-280,-396){
     \put(297,422)   {\scriptsize $y$ }
     \put(297,405)   {\scriptsize $x$ }
     \put(300,437)   {\scriptsize $+$ }
     \put(300,390)   {\scriptsize $-$ }
     \put(307,412)   {\scriptsize $a$ }
     \put(287,412)   {\scriptsize $b$ }
  }}
\end{picture}}
\Big)
& 
T\Big(
\raisebox{-10pt}{\begin{picture}(26,27)
  \put(0,0){\scalebox{.70}{\includegraphics{pic17c.eps}}}
  \put(0,0){
     \setlength{\unitlength}{.70pt}\put(-280,-396){
     \put(288,411)   {\scriptsize $a$ }
  }}
\end{picture}}
\Big)
\ar@/_2em/[ul]_(.4){
T\Bigg(
\raisebox{-17pt}{\begin{picture}(40,42)
  \put(0,0){\scalebox{.35}{\includegraphics{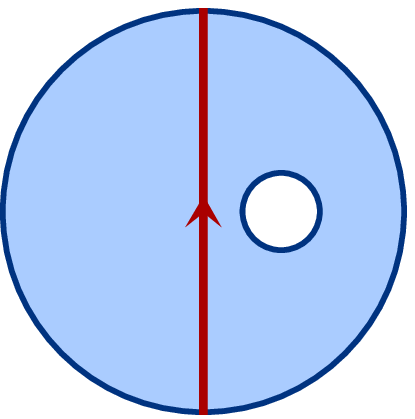}}}
  \put(0,0){
     \setlength{\unitlength}{.35pt}\put(-241,-356){
     \put(248,410)   {\scriptsize $b$ }
     \put(341,410)   {\scriptsize $a$ }
     \put(305,377)   {\scriptsize $x$ }
  }}
\end{picture}}
\Bigg)
} \ar@/^2em/[dl]^(.4){
T\Bigg(
\raisebox{-17pt}{\begin{picture}(40,42)
  \put(0,0){\scalebox{.35}{\includegraphics{pic17b.eps}}}
  \put(0,0){
     \setlength{\unitlength}{.35pt}\put(-241,-356){
     \put(248,410)   {\scriptsize $b$ }
     \put(341,410)   {\scriptsize $a$ }
     \put(305,377)   {\scriptsize $y$ }
  }}
\end{picture}}
\Bigg)
} \\ & 
T\Big(
\raisebox{-17pt}{\begin{picture}(25,42)
  \put(-1,6){\scalebox{.70}{\includegraphics{pic17d.eps}}}
  \put(-1,6){
     \setlength{\unitlength}{.70pt}\put(-280,-396){
     \put(297,422)   {\scriptsize $y$ }
     \put(297,405)   {\scriptsize $y$ }
     \put(300,437)   {\scriptsize $+$ }
     \put(300,390)   {\scriptsize $-$ }
     \put(307,412)   {\scriptsize $a$ }
     \put(287,412)   {\scriptsize $b$ }
  }}
\end{picture}}
\Big)
 \ar[u]_(.55){
T\Bigg(
\raisebox{-26pt}{\begin{picture}(58,60)
  \put(0,0){\scalebox{.5}{\includegraphics{pic16a.eps}}}
  \put(0,0){
     \setlength{\unitlength}{.5pt}\put(-241,-356){
     \put(253,410)   {\scriptsize $b$ }
     \put(336,410)   {\scriptsize $a$ }
     \put(306,453)   {\scriptsize $y$ }
     \put(306,410)   {\scriptsize $y$ }
     \put(306,368)   {\scriptsize $x$ }
     \put(294,388)   {\scriptsize $f$ }
  }}
\end{picture}}
\Bigg)
 } }
$$
\vspace*{-1.5em}
\caption{Summary of the state spaces and maps between them as described in section \ref{sec:spaces+maps}. Here only the case that $\dwt x = ((x,+))$ and $\dwt y = ((y,+))$ is shown. For tuples with more elements, the bordisms involve the corresponding sequences of parallel lines.}
\label{fig:2-digram-from-TFT}
\end{figure}

Next, consider the space  of 2-morphisms $\bfD_2(\dwt x,\dwt y)$ between two 1-morphisms $\dwt x, \dwt y : a \to b$. This is the TFT state space for a circle with sequence of marked points $\dwt y \circ \dwt x^*$. By vertical composition, $\bfD_2(\dwt x,\dwt y)$ is a right $\bfD_2(\dwt x,\dwt x)$-module and a left $\bfD_2(\dwt y,\dwt y)$-module. Using $R$ to map $\bfD_2(\one_a,\one_a)$ into $\bfD_2(\dwt x,\dwt x)$ and $\bfD_2(\dwt y,\dwt y)$, we see that $\bfD_2(\dwt x,\dwt y)$ is also a bimodule for $\bfD_2(\one_a,\one_a)$. However, by an argument analogous to that in figure \ref{fig:L-maps-to-centre} it is easy to check that the left and right action agree. Equally, $L$ turns it into an $\bfD_2(\one_b,\one_b)$ bimodule with identical left and right action.

Let $f \in \bfD_2(\dwt x,\dwt y)$ be a 2-morphism. Pre- and post-composing with $f$ defines maps 
\be
  f \circ (-) : \bfD_2(\dwt x,\dwt x) \to \bfD_2(\dwt x,\dwt y)
  \quad , \quad
  (-) \circ f : \bfD_2(\dwt y,\dwt y) \to \bfD_2(\dwt x,\dwt y)
\ee
Again by manipulating bordisms, one checks that $f \circ (-)$ intertwines the right $\bfD_2(\dwt x,\dwt x)$ action and $(-) \circ f$ intertwines the left $\bfD_2(\dwt y,\dwt y)$-action. All these maps are collected in figure \ref{fig:2-digram-from-TFT}.

\begin{rem}
In conformal (and thus in particular in topological) field theory, one has the state-field correspondence, which says that the space of fields associated to a point on the world sheet is the same as the space of states on a small circle obtained by cutting out a small disc around this point; in fact, one can take this as a definition of what one means by a field. Then $\bfD_2(\one_a,\one_a)$ is the space of `bulk fields' and $\bfD_2(\dwt x,\dwt x)$ is the space of `defect fields' supported on the defect $x$. An important notion in quantum field theory is the short distance expansion or operator product expansion (OPE). In TFT, of course, the distance between insertion points is immaterial. The above considerations isolate the three most important OPEs: the OPE of two bulk fields; the OPE of two defect fields; the expansion of a bulk field close to a defect line in terms of defect fields.
\end{rem}

\subsection{The bicategory of commutative algebras, version 1} \label{sec:comm-alg-v1}

In this subsection we will use some of the structure seen in 2d TFT with defects in the previous subsection to define a bicategory of commutative algebras in terms of cospans. All algebras will be unital, associative algebras over a field $k$.

\begin{defn} \label{def:cospan}
A {\em cospan between commutative algebras}, or {\em cospan} for short, is a tuple $(A,\alpha,T,\beta,B)$, where $A,B$ are commutative algebras, $T$ is an algebra, and $\alpha,\beta$ are algebra homomorphisms
\be\label{eq:cospan-def}
\cospd{A}{\alpha}{T}{\beta}{B}
\ee
such that the images of $\alpha$ and $\beta$ lie in the centre $Z(T)$ of $T$.
\end{defn}

The definition has no preferred `direction', but we will pick one anyway: we will think of \eqref{eq:cospan-def} as going from $B$ to $A$. The reason for this choice is that we will use the maps $\alpha,\beta$ to turn $T$ into an $A$-$B$-bimodule and the composition of cospans (to be defined in more detail below) will be the tensor product of bimodules, just as in the bicategory $\Alg(k)$. In the latter, $A$-$B$-bimodules serve as 1-morphisms $B \to A$. We will write $T : B \to A$, or just $T$, to abbreviate the data in \eqref{eq:cospan-def}. Two different cospans $T,T' : B \to A$ can be compared via algebra homomorphisms $T\to T'$. This leads first to a category of cospans from $B$ to $A$, and then to a bicategory $\CAlg(k)$ of cospans between commutative algebras.
The construction is almost identical to the standard construction of the bicategory of spans for a given category with pullbacks (see \cite[Sec.\,2.6]{Benabou:1967} or \cite{Gray:1974,MacLane-book}), with the exception that not all three objects in \eqref{eq:cospan-def} are taken from the same category (we use commutative algebras for the starting points of the cospan and not necessarily commutative algebras for the middle term).

\begin{defn} \label{def:cosp-1cat}
The category $\cosp(A,B)$ of cospans between commutative algebras from $A$ to $B$ is defined as follows.
\begin{itemize}
\item objects $T \in \cosp(A,B)$ are cospans $T : A \to B$.
\item morphisms $f \in \cosp(A,B)(T,T')$ from $T : A \to B$ to $T' : A \to B$ are algebra maps $f : T \to T'$ such that the following diagram commutes:
\be \label{eq:cospan-morphism-condition}
\cospdm{B}{\beta}{T}{\alpha}{A}{\beta'}{T'}{\alpha'}{f}
\ee
\item the unit morphism in $\cosp(A,B)(T,T)$ is the identity map $\id_T$, and composition of morphisms is composition of algebra maps.
\end{itemize}
\end{defn}

The composition of two cospans $A \xrightarrow{T} B \xrightarrow{S} C$ is defined by the usual pushout square,
\be
\raisebox{3.6em}{\small \xymatrix@C=1.5em@R=1.5em{
&& S \otimes_B T
\\
& S \ar[ur]^{\id \otimes_B 1_T} && T \ar[ul]_{1_S \otimes_B \id} 
\\
C \ar[ur]^{\gamma} && B \ar[ul]_\beta \ar[ur]^{\beta'} && A \ar[ul]_\alpha
}}
\ee
The algebra homomorphism $\beta$ turns $S$ into a right $B$-module, and $\beta'$ turns $T$ into a left $B$-module; these module structures are implied when writing $S \otimes_B T$. 
We still need to turn $S \otimes_B T$ into an algebra and show that $A$ and $C$ get mapped to the centre of this algebra.
The multiplication on $S \otimes_B T$ is given by 
\be
  (s\otimes_B t) \cdot (s' \otimes_B t') := (ss') \otimes_B (tt') \ .
\ee  
To check that this is well-defined, we start with the map $\bar m : (S \otimes T) \otimes (S \otimes T) \to S \otimes_B T$, which takes $(s\otimes t) \otimes (s' \otimes t')$ to $(ss') \otimes_B (tt')$ and verify the cokernel condition. We present the calculation for the first factor, the one for the second factor is similar. With $b \in B$,
\be\begin{array}{l}
  \bar m(s.b \otimes t \otimes s' \otimes t') 
  = (s \beta(b) s') \otimes_B (tt')
  \overset{(*)}= (s s' \beta(b)) \otimes_B (tt')
\\[.5em]
\quad  = (s s').b \otimes_B (tt')
  = (s s') \otimes_B b.(tt')
 = (s s') \otimes_B (\beta'(b)tt')
  = \bar m(s \otimes b.t \otimes s' \otimes t') 
\end{array}
\ee
The step marked `$(*)$' uses that the image of the algebra homomorphism $\beta$ is in the centre of $S$. This, by the way, is the reason not to allow general algebra homomorphisms in definition \ref{def:cospan}: we want the tensor product over $B$ to carry an induced algebra structure.
Finally, it is straightforward to check that for all $a \in A$, $c \in C$, the elements $1_S \otimes_B \alpha(a)$ and $\gamma(c) \otimes_B 1_T$ are in the centre of $S \otimes_B T$. 

The composition of cospans defined above forms part of a functor:

\begin{lem} \label{lem:cosp-composition}
The assignment
\be \label{eq:cosp-compos-functor}
\begin{array}{rccccc}
\circledcirc_{C,B,A} :& \cosp(B,C) &\hspace{-.5em}\times\hspace{-.5em}& \cosp(A,B) &\longrightarrow& \cosp(A,C)
\\
&\Bigg(
\cospdm{C}{\gamma}{T}{\beta_2}{B}{\gamma'}{T'}{\beta_2'}{g}
&\hspace{-.5em},\hspace{-.5em}&
\cospdm{B}{\beta_1}{S}{\alpha}{A}{\beta_1'}{S'}{\alpha'}{f}
\Bigg)
&\longmapsto&
\cospdm{C}{\gamma \otimes_B 1}{T \otimes_B S}{1 \otimes_B \alpha}{A}{\gamma' \otimes_B 1}{T' \otimes_B S'}{1 \otimes_B \alpha'}{g \otimes_B f}
\end{array}
\ee
defines a functor.
\end{lem}

\begin{proof}
Note that the algebra map $g$ is automatically a $C$-$B$-bimodule map. For example, $g(t.b) = g(t \cdot \beta_2(b)) = g(t) \cdot g(\beta_2(b)) = g(t) \cdot \beta_2'(b) = g(t).b$. Similarly, $f$ is a $B$-$A$-bimodule map. Thus $g \otimes_B f$ is well-defined. That the two triangles on the rhs of \eqref{eq:cosp-compos-functor} commute is immediate. Finally, functoriality of $\circledcirc_{C,B,A}$ amounts to the statement that
\be
  (g_2 \circ g_1) \otimes_B (f_2 \circ f_1)
  = (g_2 \otimes_B f_2) \circ (g_1 \otimes_B f_1) \ ,
\ee
which is a property of the tensor product over $B$.
\end{proof}

We have now gathered the ingredients to define the first version of the bicategory of commutative algebras, which we denote by 
\be \label{eq:CAlg-def}
  \CAlg(k) \ .
\ee
Its objects are commutative algebras over $k$. Given two such algebras $A,B$, the category of morphisms from $A$ to $B$ is $\cosp(A,B)$. The identity in $\cosp(A,A)$ is
\be\label{eq:identity-cospan}
\cospd A \id A \id A
 \qquad .
 \ee
 The composition functor is $\circledcirc_{C,B,A}$ from lemma \ref{lem:cosp-composition}. The associativity and unit isomorphisms are just the natural isomorphisms $T \otimes_C (S \otimes_B R) \cong (T \otimes_C S) \otimes_B R$ and $R \otimes_A A \cong R \cong B \otimes_B R$, which we will not write out in the following. It is then clear that the coherence conditions of a bicategory -- as listed in appendix \ref{app:bicategories} -- are satisfied.

\begin{rem}
For a category $\mathcal{C}$, we denote by $\mathcal{C}^{1,0}$ the subcategory containing only invertible morphisms. Similarly,
given a bicategory $\bfB$, denote by $\bfB^{2,1}$ the bicategory obtained from $\bfB$ by restricting to invertible 2-morphisms, and by $\bfB^{2,0}$ the bicategory consisting only of invertible 1- and 2-morphisms (and $\bfB \equiv \bfB^{2,2}$; see \cite{Lurie:2009aa} for more on $(m,n)$-categories). If $\bfB$ is a $k$-linear bicategory, we obtain a lax functor
\be \label{eq:E-B-CAlg-def}
  \bfE \,:\, \bfB^{2,1} \longrightarrow \CAlg(k) \ ,
\ee
where `$\bfE$' stands for endomorphism. We will illustrate this functor in the case of the 2-category $\bfB \equiv \bfD[D_2,D_1;T]$ for a fixed 2d TFT $T$ (which need not come from the lattice construction). On objects and 1-morphisms we set
\be
  \bfE(a) = \bfB(\one_a,\one_a) \quad , \qquad
  \bfE(b \xleftarrow{\dwt x} a) =  \Bigg(
  \cospd{\bfB(\one_b,\one_b)}{L}{\bfB(\dwt x,\dwt x)}{R}{\bfB(\one_a,\one_a)}  \Bigg)
  \quad ;
\ee
the maps $L$ and $R$ have been given in section \ref{sec:spaces+maps}. To an invertible 2-morphism $u : \dwt x \to \dwt y$ we assign the algebra map $\bfE(u) : \bfE(\dwt x) \to \bfE(\dwt y)$ given by conjugation with $u$. That is, $f : \dwt x \to \dwt x$ gets mapped to
\be \label{eq:E-functor-on-2morph}
  \bfE(u)(f) = 
  \Big( \dwt y \xrightarrow{u^{-1}} \dwt x \xrightarrow{~f~} \dwt x \xrightarrow{~u~} \dwt y \Big) \ .
\ee
We will omit the details of the proof that $\bfE$ is a lax functor. Note that, because \eqref{eq:E-functor-on-2morph} involves an inverse, $\bfE$ is only defined on $\bfB^{2,1}$. Nonetheless, $\bfB^{2,1}$ is not enough to define $\bfE$, instead one requires all of $\bfB$ so that $\bfB(\one_a,\one_a)$, etc., are indeed $k$-algebras. Also, even though the image $\bfE(u)$ of a 2-morphism is always invertible, it is not true that $\bfE$ is a functor to $\CAlg(k)^{2,1}$, because the associativity 2-morphism $\bfE(\dwt y) \circledcirc \bfE(\dwt x) \to \bfE(\dwt y \circ \dwt x)$ is not necessarily invertible.
\end{rem} 

Denote by $\alg(k)_\mathrm{com}$ the full subcategory of commutative algebras in $\alg(k)$. In the remainder of this subsection we illustrate that $\CAlg(k)$ enlarges the morphism spaces of $\alg(k)_\mathrm{com}$, but that it does not add new invertible morphisms.

\begin{lem} \label{lem:commutative-algebra-embedding}
The assignment
\be
\begin{array}{rccc}
I :& \alg(k)_\mathrm{com} &\longrightarrow&\CAlg(k)^{2,1}
\\[.5em]
&
B \xleftarrow{f} A &\longmapsto& \cospd B \id B f A
\end{array}
\ee
defines a (non-lax) functor.
\end{lem}

\begin{proof}
Clearly, the identity gets mapped to the identity. Given two algebra homomorphisms $A \xrightarrow{f} B \xrightarrow{g} C$, one verifies that the map $m_{g,f}$ from lemma \ref{lem:alg->Alg} defines an isomorphism of cospans
\be
  \cospdm C{\id \otimes_B 1}{C \otimes_B B}{1 \otimes_B f}A \id C {g \circ f}{m_{g,f}}
  \quad .
\ee
The verification of the associativity condition works along the same lines as the proof of lemma \ref{lem:alg->Alg}\,(ii).
\end{proof}

A 2-morphism between the cospans $I(f)$ and $I(g)$ would necessarily have to be the identity map in oder to make the left triangle in the condition \eqref{eq:cospan-morphism-condition} commute. This then implies $f=g$. In particular, different algebra maps get mapped to non-2-isomorphic cospans. In this sense, $I$ is faithful.

\begin{lem} \label{lem:invertible-cospan-1}
A cospan 
\be \label{eq:invertible-cospan-1-T}
\cospd B \beta T \alpha A
\ee
is invertible in $\CAlg(k)$ if and only if $\alpha$ and $\beta$ are isomorphisms.
\end{lem}

\begin{proof}
`$\Leftarrow$': Suppose $\alpha$ and $\beta$ are isomorphisms. Then 
\be\label{eq:invertible-cospan-1-aux2}
\cospdm B \beta T \alpha A \id B {\beta^{-1} \circ \alpha} {\beta^{-1}}
\ee
is an isomorphism of cospans, i.e.\ $T \cong I(\beta^{-1} \circ \alpha)$. The latter cospan has inverse $I(\alpha^{-1} \circ \beta)$ by lemma \ref{lem:commutative-algebra-embedding}. Thus also $T$ is invertible.
\\[.3em]
`$\Rightarrow$': Suppose $(A,\alpha',S,\beta',B)$ is a two-sided inverse of $T$. This means that there are isomorphisms $f,g$ of cospans
\be
\cospdm A{\alpha' \otimes_B 1}{S \otimes_B T}{1 \otimes_B \alpha}A\id A \id f
\qquad \text{and} \qquad
\cospdm B{\beta \otimes_A 1}{T \otimes_A S}{1 \otimes_A \beta'}B \id B \id g
\qquad .
\ee
Consider the left diagram. Since $f$ is an isomorphism, it implies that also $f^{-1} = 1_S \otimes_B \alpha$ is an isomorphism. Thus we have the identities
\be\label{eq:invertible-cospan-1-aux1}
  f \circ (1_S \otimes_B \alpha) = \id_A 
  \qquad , \qquad (1_S \otimes_B \alpha) \circ f = \id_{S \otimes_B T} \ .
\ee  
The first of these can be rewritten as $f \circ (1_S \otimes_B \id_T) \circ \alpha = \id_A$, showing that $\alpha$ 
has left-inverse $\hat f := f \circ (1_S \otimes_B \id_T) : T \to A$. An analogous argument gives the left inverse $\hat g := g \circ (1_T \otimes_A \id_S) : S \to B$ of $\beta'$.

Since $\beta' : B \to S$ is an algebra map and an intertwiner of right $B$-modules (by definition of the right $B$-action on $S$), we can write $1_S \otimes_B \alpha : A \to S \otimes_B T$ as $(\beta' \otimes_B \id_T) \circ (1_B \otimes_B \alpha)$. Inserting this into the second identity in \eqref{eq:invertible-cospan-1-aux1} gives
\be
  \id_{S \otimes_B T}
  = \Big(
  S \otimes_B T \xrightarrow{f} A \xrightarrow{~\alpha~} T \xrightarrow{1_B \otimes_B \id_T} B \otimes_B T
  \xrightarrow{\beta' \otimes_B \id_T} S \otimes_B T \Big)
\ee
We compose both sides with $\hat g \otimes_B \id_T$ and use that $\hat g$ is left-inverse to $\beta'$. This results in $\hat g \otimes_B \id_T = (1_B \otimes_B \id_T) \circ \alpha \circ f$. Finally, composing with $1_S \otimes_B \id_T$ from the left and using that $\hat g(1_S) = g(1_T \otimes_A 1_S) = 1_B$, shows that $1_B \otimes_B \id_T = (1_B \otimes_B \id_T) \circ \alpha \circ \hat f$. Since $1_B \otimes_B \id_T$ is an isomorphism, we see that $\hat f$ is also a right-inverse for $\alpha$, and hence $\alpha$ is an isomorphism. That $\beta$ is an isomorphism follows along the same lines.
\end{proof}

From the proof we see that the cospan $T$ in \eqref{eq:invertible-cospan-1-T} is 2-isomorphic to $I(\beta^{-1} \circ \alpha)$. Thus every 1-isomorphism lies in the essential image of $I$. 

\begin{rem} \label{rem:algkx->CAlgux-functor}
An algebra isomorphism $f : A \to B$, when restricted to $Z(A)$, provides an isomorphism $f|_{Z(A)} : Z(A) \to Z(B)$. This gives a functor from $\alg(k)^{1,0}$ to $\alg(k)_\mathrm{com}^{1,0}$, which, when composed with $I$, gives a functor
\be \label{eq:Z-alg^x-CAlg^x}
  \alg(k)^{1,0} \xrightarrow{~Z~} \alg(k)_\mathrm{com}^{1,0} \xrightarrow{~I~} \CAlg(k)^{2,0} \ .
\ee
In theorem \ref{thm:alg->CAlg-laxfun} below, we will extend this beyond the groupoid case to a lax functor $Z : \alg(k) \to \CAlg(k)$. 
As an aside, note that the composed functor in \eqref{eq:Z-alg^x-CAlg^x} is neither full nor faithful (isomorphic centres do not imply isomorphic algebras, and different algebra isomorphisms may restrict to the same map on the centre). However, $I :  \alg(k)_\mathrm{com}^{1,0} \to \CAlg(k)^{2,0}$ {\em is} an equivalence since for invertible 1- and 2-morphisms, $I$ is an equivalence on the morphism categories (as those morphism categories which lie in the image of $I$ only contain identity 2-morphisms), and on objects it is just the identity.
\end{rem}

\subsection{Functorial centre, version 1}

Given two not necessarily commutative algebras $A$, $B$ and an algebra homomorphism $f : A \to B$, we define the {\em centraliser} $Z_{A,B}(f)$ to be the centraliser of the image of $f$ in $B$,
\be \label{eq:ZAB-def}
  Z_{A,B}(f) = \big\{\,b \in B \,\big|\, f(a)\,b = b\,f(a) \, \text{ for all } a \in A \, \big\} \ .
\ee 
Let $\iota : Z(B) \to B$ be the embedding map and denote the restriction of $f$ to $Z(A)$ also by $f$.

\begin{lem}
Let $A,B$ be algebras and $f : A \to B$ an algebra homomorphism. Then
\be \label{eq:ZAB-cospan}
\cospd{Z(B)}\iota{Z_{A,B}(f)}f{Z(A)}
\ee
is a cospan of commutative algebras.
\end{lem}

\begin{proof}
Since $Z(B) \subset Z_{A,B}(f)$, $\iota$ is an algebra map which maps to the centre of $Z_{A,B}(f)$. Next we check that the image of $Z(A)$ under $f : A \to B$ lies in $Z_{A,B}(f)$. Given $z \in Z(A)$ set $b = f(z)$. For all $a \in A$ we have $f(a)b = f(a)f(z) = f(az) = f(za) = f(z)f(a) = bf(a)$. Thus $f(z) \in Z_{A,B}(f)$. It is then immediate that $f(Z(A))$ lies in the centre of $Z_{A,B}(f)$.
\end{proof}

\begin{rem} \label{rem:ZAB(f)-HomBfBf}
In the more restrictive setting of Frobenius algebras with trace pairing, the cospan \eqref{eq:ZAB-cospan} has actually already appeared in disguise in the lattice TFT construction. Consider the cospan forming the top of the diamond of maps in figure \ref{fig:2-digram-from-TFT}. By definition \eqref{eq:def-TFT-T-on-S1} and lemma \ref{lem:ho-hom=ho-cohom}, $T(O(a)) = \cytens{A_a} A_a = Z(A_a)$ and $T(O(b)) = Z(A_b)$. For the top entry we have $T(O(x^* \circ x)) = \cytens{A_b} X_x \otimes_{A_a} X_x^* \cong \Hom_{A_b|A_a}(X_x,X_x)$, where we used lemma \ref{lem:YxX*-Hom(X,Y)}. To make the connection to \eqref{eq:ZAB-cospan}, take $A = A_a$, $B = A_b$ and let $f : A \to B$ be an algebra map. For $X_x$ take the bimodule $B_f$ defined in section \ref{sec:bicat-alg-2cat-def}. The map $\phi$ in
\be
  \cospdm{Z(B)}\iota{Z_{A,B}(f)}f{Z(A)}{\mathrm{act}}{\Hom_{B|A}(B_f,B_f)}{\mathrm{act}}{\phi}
  \quad ,
\ee
defined as $\phi(b) := (u \mapsto u\cdot b)$, provides an isomorphism of cospans. Here `act' refers to the map that takes $b \in Z(B)$ to the bimodule map $u \mapsto b.u$; $u \in B_f$ (resp.\ $a \in Z(A)$ to $u \mapsto u.a$). We omit the details.
\end{rem}

\begin{lem}\label{lem:m_fg-def}
Let $A \xrightarrow{f} B \xrightarrow{g} C$ be algebra maps. 
The map $m_{g,f} : Z_{B,C}(g) \otimes_{Z(B)} Z_{A,B}(f) \to Z_{A,C}(g \circ f)$ given by $u \otimes_{Z(B)} v \mapsto u \cdot g(v)$ is a morphism of cospans
\be \label{eq:m_gf-def}
\raisebox{2.1em}{\small \xymatrix@R=0.8em@C=1.5em{
& Z_{B,C}(g) \otimes_{Z(B)} Z_{A,B}(f) \ar[dd]^{m_{g,f}} \\
Z(C) \ar[ur]^(.3){\iota \otimes_B 1} \ar[dr]_(.35){\iota} && Z(A) \ar[ul]_(.3){1 \otimes_B f} \ar[dl]^(.35){g \circ f} \\
& Z_{A,C}(g \circ f) }}
\ee
\end{lem}

\begin{proof}
Abbreviate $m \equiv m_{g,f}$ and $Y \equiv Z(B)$. 
\\[.3em]
\nxt {\em $m$ is well-defined:} 
That $m$ gives a well-defined map to $C$ is the same argument as in the proof of lemma \ref{lem:alg->Alg}\,(i). That the image of $m$ lies in the centraliser $Z_{A,C}(g \circ f)$ amounts to, for all $a \in A$,
\be\begin{array}{ll}
g(f(a)) \cdot m(u \otimes_Y v) &\!\!= g(f(a)) \cdot u \cdot g(v) \overset{(1)}{=} u \cdot g(f(a)) \cdot g(v) =  u \cdot g(f(a)v) 
\\[.5em]
  &\!\! \overset{(2)}= u \cdot g(vf(a)) = u \cdot g(v) \cdot g(f(a)) = m(u \otimes_Y v) \cdot g(f(a)) \ , 
\end{array}
\ee
where (1) follows as $u \in Z_{B,C}(g)$ commutes with anything in the image of $g$, and (2) follows analogously from $v \in Z_{A,B}(f)$.
\\[.3em]
\nxt {\em $m$ is an algebra map:}  We have
\be\begin{array}{ll}
  m\big( (u \otimes_Y v) \cdot (u' \otimes_Y v') \big)
  &\!\!= m\big( (uu') \otimes_Y (vv') \big)
  = uu'g(vv')
  = uu'g(v)g(v')
\\[.5em]
  &\!\!\overset{(*)}= ug(v)u'g(v')
  = m(u \otimes_Y v) \cdot m(u' \otimes_Y v') \ .
\end{array}
\ee
The only perhaps not immediately obvious step is $(*)$, which follows since by definition for all $u' \in Z_{B,C}(g)$ and $v \in B$ we have $u' g(v) = g(v) u'$.
\\[.3em]
\nxt {\em The triangles commute:} Acting on arbitrary elements $c \in Z(C)$ and $a\in Z(A)$, commutativity of the two triangles amounts to the identities $c = m(c \otimes_Y 1_B)$ and $g(f(a)) = m(1_B \otimes_Y f(a))$, both of which are immediate upon substituting the definition of $m$.
\end{proof}

We have now collected the ingredients to state the second main result of this note.

\begin{thm} \label{thm:alg->CAlg-laxfun}
The assignment
\be \label{eq:alg->CAlg-defn}
\begin{array}{rrcl}
Z :& \alg(k) &~~\longrightarrow~~&\CAlg(k)
\\[.5em]
&
B \xleftarrow{f} A &\longmapsto& 
\cospd{Z(B)}{\iota}{Z_{A,B}(f)}{f}{Z(A)}
\end{array}
\ee
defines a lax functor. The unit transformations are identities and the multiplication transformations are given by $m_{g,f}$.
\end{thm}

In other words, on objects the lax functor acts as $A \mapsto Z(A)$, on 1-morphisms $A \to B$ as $f \mapsto Z_{A,B}(f)$, and all 2-morphisms in $\alg(k)$ are identities, which get mapped to identity 2-morphisms in $\CAlg(k)$. 

\begin{proof}
It remains to verify the associativity and unit properties. The argument is identical to that in the proof of lemma \ref{lem:alg->Alg}\,(ii).
\end{proof}

\begin{rem}
(i) The map $m_{g,f}$ in lemma \ref{lem:m_fg-def} is typically not an isomorphism. For example, take $A = C = k \oplus k$ and $B$ to be upper triangular $2{\times}2$-matrices. For the map $f$ we take the diagonal embedding and for $g$ the projection onto the diagonal part. By commutativity of the underlying algebra we see $Z(A)=Z(C)=Z_{B,C}(g) = Z_{A,C}(g \circ f) = k \oplus k$. The remaining algebras are $Z(B) \cong k$ (multiples of the identity matrix) and $Z_{A,B}(f) \cong k \oplus k$, the diagonal $2{\times}2$-matrices. Thus $Z_{B,C}(g) \otimes_{Z(B)} Z_{A,B}(f) \cong (k \oplus k) \otimes_k (k \oplus k)$ while $Z_{A,C}(g \circ f) = k \oplus k$.
Therefore, we only have a lax functor.
\\[.3em]
(ii) If we restrict $Z$ to commutative algebras $\alg(k)_\mathrm{com}$ we obtain the functor $I$ from lemma \ref{lem:commutative-algebra-embedding} (since then $Z(B)=B$, $Z(A)=A$ and $Z_{A,B}(f) = B$ for all $f$). In this sense, $Z$ is an extension of $I$ to all algebras; the price to pay is that we have to work with bicategories and the functor becomes lax.
\end{rem}

\subsection{The bicategory of commutative algebras, version 2}

Figure \ref{fig:2-digram-from-TFT} suggests that there is an enlargement of $\CAlg(k)$, where the 2-morphisms are also replaced by cospans. The lattice TFT construction suggests that this enlargement becomes relevant if one wants to extend the centre functor from $\alg(k)$ to $\Alg(k)$. This is the topic of the present subsection, as well as of the next one.

\begin{defn} \label{def:2-diagram+3-cell}  
(i) A {\em 2-diagram} from a cospan $S : A \to B$ to $T : A \to B$ is 
a triple $(g,M,f)$, where $M$ is
a $T$-$S$-bimodule, $f$ is a right $S$-module map and $g$ is a left $T$-module map such that
the two squares in the diagram
\be  \label{eq:2diagram-cond-1}
\cospdd{B}{\beta_1}{{S}{f}}{\alpha_1}{A}{\beta_2}{{T}{g}}{\alpha_2}{M}
\ee
commute, and such that the induced left and right action of $A$ on $M$ agree, and those of $B$ on $M$ agree, i.e.\ that for all $a \in A$, $b \in B$, $m \in M$
\be  \label{eq:2diagram-cond-2}
\alpha_2(a).m = m.\alpha_1(a) \quad , \quad
\beta_2(b).m = m.\beta_1(b) \ .
\ee
We will also abbreviate $M : S \to T$.
\\[.3em]
(ii) A {\em 3-cell} between two 2-diagrams $(g,M,f)$ and $(g',M',f')$ is a $T$-$S$-bimodule map $\delta : M \to M'$ such that the following diagram commutes:
\be
\raisebox{3em}{\small\xymatrix@R=1.5em@C=1em{
& S \ar[ld]_f \ar[rd]^{f'} & \\ M \ar[rr]^\delta && M' \\
& T \ar[lu]^g \ar[ru]_{g'} }}
\ee
(iii) The {\em category of 2-diagrams} $\tdiag_{AB}(S,T)$ has 2-diagrams as objects and 3-cells as morphisms. The identity 3-cell for the object $(g,M,f)$ is the identity map on $M$, the composition of 3-cells is given by composition of bimodule maps.
\end{defn}

\begin{rem} \label{rem:tricat-conj}
(i) 
The category $\tdiag_{AB}(S,T)$ can be used to define a bicategory $\Cosp(A,B)$ whose objects are cospans of commutative algebras from $A$ to $B$ and whose morphism categories are $\tdiag_{AB}(S,T)$; this will be done in detail in \cite{in-prep}. Conjecturally, there is a tricategory $\CALG(k)$ whose objects are commutative algebras and whose morphism bicategories are given by $\Cosp(A,B)$; we hope to return to this in the future.
\\[.3em]
(ii)
The conjectural tricategory $\CALG(k)$ of part (i) is similar (but not equal) in structure to the tricategory of conformal nets described in \cite{Bartels:2009ts}. 
In \cite{Bartels:2009ts}, objects are conformal nets. An object in $\CALG(k)$, i.e.\ a commutative algebra, could be thought of as a `topological net' which assigns the same algebra to every interval; this algebra is then necessarily commutative (but it is not a conformal net: the algebra does not have to be von Neumann and in general it violates the split property).
According to \cite[Def.\,3]{Bartels:2009ts}, 1-morphisms are `defects', i.e.\ conformal nets for bicoloured intervals. In the present language, this corresponds to the data of a cospan $(A,\alpha,T,\beta,B)$: evaluating the net on mono-coloured subintervals produces the two commutative algebras $A,B$, for a bicoloured interval one obtains $T$, and the inclusion of a mono-- into a bicoloured interval provides the two maps $\alpha,\beta$.
In \cite[Def.\,4]{Bartels:2009ts}, 2-morphisms are sectors between the defect nets -- this means a Hilbert space with compatible actions of all four conformal nets involved: the two defect nets and the two conformal nets which the defects go between. This provides all the data and constraints of a 2-diagram as in \eqref{eq:2diagram-cond-1} {\em except} for the maps $f$ and $g$, which are not part of the setting of \cite{Bartels:2009ts}. We will need these two maps for the centre functor, see lemma \ref{lem:centre-2-diagram} below.
The third level of categorical structure is constructed in \cite{Bartels:2009ts} by making conformal nets a bicategory internal to symmetric monoidal categories. 
\end{rem}

Part (i) of the above remark motivates the notation $\Cospu(A,B)$ for the category whose objects are cospans from $A$ to $B$ and whose morphisms are isomorphism classes of 2-diagrams. Composition in $\Cospu(A,B)$ is given by
\be \label{eq:2-diag-vertical}
\begin{array}{rccccc}
\odot :& \Cospu(A,B)(S,T) &\hspace{-.5em}\times\hspace{-.5em}& \Cospu(A,B)(R,S) &\longrightarrow& \Cospu(A,B)(R,T)
\\
&\Bigg(
\cospdd B{\beta_2}{S g}{\alpha_2}A{\beta_3}{T{g'}}{\alpha_3}N
&\hspace{-.5em},\hspace{-.5em}&
\cospdd B{\beta_1}{R f}{\alpha_1}A{\beta_2}{S{f'}}{\alpha_2}M
\Bigg)
&\longmapsto&
\raisebox{2.8em}{\small \xymatrix@R=1.5em@C=2em{
& R \ar[d]_{g(1) \otimes_S f} \\
B \ar@/^10pt/[ur]^{\beta_1} \ar@/_10pt/[dr]_{\beta_3} & N \otimes_S M & A \ar@/_10pt/[ul]_{\alpha_1} \ar@/^10pt/[dl]^{\alpha_3} \\
& T \ar[u]_{g' \otimes_S f'(1)}  }}
\end{array}
\quad .
\ee
The right hand side is again a 2-diagram. Let us check explicitly the left square in \eqref{eq:2diagram-cond-1} and the first condition in \eqref{eq:2diagram-cond-2}. For $b \in B$,
\be\begin{array}{ll}
g(1) \otimes_S f(\beta_1(b))
&\overset{(1)}=
g(1) \otimes_S f'(\beta_2(b))
\overset{(2)}=
g(1) \otimes_S \beta_2(b).f'(1)
\overset{(3)}=
g(1).\beta_2(b) \otimes_S f'(1)
\\[.5em] &
\overset{(4)}=
g(\beta_2(b)) \otimes_S f'(1)
\overset{(5)}=
g'(\beta_3(b)) \otimes_S f'(1) \ .
\end{array}
\ee
Here, (1) is commutativity of the left square in the 2-diagram $(f',M,f)$, (2) the fact that $f'$ is left $S$-module map, (3) the property of $\otimes_S$, (4) follows since $g$ is a right $S$-module map, and finally (5) is commutativity of the left square in the 2-diagram $(g',N,g)$. Next, that the left action of $B$ on $N \otimes_S M$ agrees with the right action of $B$ follows from
\be\begin{array}{ll}
  b.(n \otimes_S m) 
  &=   (\beta_3(b).n) \otimes_S m
  =   (n.\beta_2(b)) \otimes_S m
\\[.5em] &
  =   n \otimes_S (\beta_2(b).m)
  =   n \otimes_S (m.\beta_1(b))
  =   (n \otimes_S m).b
\end{array}
\ee
The unit morphism in $\Cospu(A,B)(T,T)$ is
\be \label{eq:unit-cospan}
  \cospdd{B}{\beta}{T\id}{\alpha}{A}{\beta}{T\id}{\alpha}{T}
\ee

\begin{lem} \label{lem:2-diag-inv}
A 2-diagram $(g,M,f)$ between $S,T : A \to B$ is invertible if and only if both $f$ and $g$ are invertible.
\end{lem}

The proof is similar to that of lemma \ref{lem:invertible-cospan-1} and we omit the details.

\begin{lem} \label{lem:Cospu-composition}
The assignment
\be \label{eq:cospu-compos-functor} \small
\begin{array}{rccccc}
\circledcirc_{C,B,A} :& \Cospu(B,C) &\hspace{-.5em}\times\hspace{-.5em}& \Cospu(A,B) &\longrightarrow& \Cospu(A,C)
\\
&\Bigg(
\cospdd{C}{\gamma}{T g}{\beta_2}{B}{\gamma'}{{T'}{g'}}{\beta_2'}{N}
&\hspace{-.5em},\hspace{-.5em}&
\cospdd{B}{\beta_1}{S f}{\alpha}{A}{\beta_1'}{{S'}{f'}}{\alpha'}{M}
\Bigg)
&\longmapsto&
\cospdd{C}{\gamma \otimes_B 1}{ {T \otimes_B S}{g \otimes_B f} }{1 \otimes_B \alpha}{A}{\gamma' \otimes_B 1}{ {T' \otimes_B S'}{g' \otimes_B f'} }{1 \otimes_B \alpha'}{N \otimes_B M}
\end{array}
\ee
defines a functor.
\end{lem}

\begin{proof}
It is evident that the functor maps a pair of identity cospans \eqref{eq:unit-cospan} to the identity cospan. To verify functoriality, choose another pair of 2-diagrams $N' : T' \to T''$ and $M' : S' \to S''$ in the product category. First composing in the product with $\odot \times \odot$ gives the pair $(N' \otimes_{T'} N , M' \otimes_{S'} M)$. Applying $\circledcirc$ yields $X := N' \otimes_{T'} N \otimes_B M' \otimes_{S'} M$. In the other order, by first applying $\circledcirc$, we obtain the two 2-diagrams $N' \otimes_B M' : T' \otimes_B S' \to T'' \otimes_B S''$ and $N \otimes_B M : T \otimes_B S \to T' \otimes_B S'$. Applying $\odot$ to this gives $Y := (N' \otimes_B M') \otimes_{T' \otimes_B S'} (N \otimes_B M)$. We claim that the isomorphism $N' \otimes N \otimes M' \otimes M \to N' \otimes M' \otimes N \otimes M$ given by permuting factors induces maps $X \to Y$ and $Y \to X$; these are then automatically inverse to each other. For example, the condition that the map $\phi : N' \otimes N \otimes M' \otimes M \to Y$ respects the tensor product over $T'$ amounts to, for $t' \in T$,
\be
\begin{array}{ll}
  \phi\big((n'.t') \otimes n \otimes m' \otimes m\big)
  &= \big((n'.t') \otimes_B m'\big) \otimes_{T' \otimes_B S'} (n \otimes_B m)
  \\[.5em]
  &= \big((n' \otimes_B m').(t' \otimes_B 1)\big) \otimes_{T' \otimes_B S'} \big(n \otimes_B m\big)
  \\[.5em]
  &= \big(n' \otimes_B m'\big) \otimes_{T' \otimes_B S'} \big((t' \otimes_B 1).(n \otimes_B m)\big)
  \\[.5em]
  &= \big(n' \otimes_B m'\big) \otimes_{T' \otimes_B S'} \big((t'.n) \otimes_B m\big)
  \\[.5em]
  &  = \phi\big(n' \otimes (t'.n) \otimes m' \otimes m\big) \ .
\end{array}
\ee
The other tensor product cokernel conditions are checked similarly. That the induced isomorphism $X \to Y$ is a 3-cell is equally straightforward.
\end{proof}

We can now define the second version of the bicategory of commutative algebras, which we denote by 
\be
  \CALGu(k) \ ;
\ee
Objects and 1-morphisms are as in $\CAlg(k)$, but for 2-morphisms we take equivalence classes of 2-diagrams. In other words, the morphism category $A \to B$ is $\Cospu(A,B)$. The composition functor is given in lemma \ref{lem:Cospu-composition}. The notation $\CALGu(k)$ is motivated by the conjectural tricategory of remark \ref{rem:tricat-conj}\,(i). The associativity and unit isomorphisms of $\CALGu(k)$ are just those of bimodules, and the required coherence conditions are satisfied for the same reason.

As compared to $\CAlg(k)$, the category $\CALGu(k)$ has more 2-morphisms. This is made precise by the observation that
\be \label{eq:I-CAlg-CALG-def}
\begin{array}{rccc}
\bfI :& \CAlg(k) &\longrightarrow&\CALGu(k)
\\[.5em]
&
\cospdm B \beta S \alpha A {\beta'} T {\alpha'} f
&\longmapsto& 
\cospdd B \beta {S f} \alpha A {\beta'} {T \id} {\alpha'} T
\end{array}
\ee
is a locally faithful functor from $\CAlg(k)$ to $\CALGu(k)$; we skip the details. By lemma \ref{lem:2-diag-inv}, each invertible morphism in $\Cospu(A,B)(S,T)$ lies in the image of $\bfI$ (we again skip the details). Therefore, the restriction
\be
\bfI : \CAlg(k)^{2,1} \xrightarrow{~~\sim~~} \CALGu(k)^{2,1}
\ee
is an equivalence of bicategories. In this sense, passing from $\CAlg(k)$ to $\CALGu(k)$ adds more non-invertible 2-morphisms.

\subsection{Functorial centre, version 2} \label{sec:fun-cent-v2}

In this subsection we will try to extend the centre functor to $\Alg(k)$. We will see that $\CALGu(k)$ is not quite good enough as a target category, and we have to restrict ourselves to appropriate subcategories of $\Alg(k)$. 

Let $A,B$ be algebras and let $X$ be a $B$-$A$-bimodule. Then
\be
\cospd{Z(B)}{\mathrm{act}}{\Hom_{B|A}(X,X)}{\mathrm{act}}{Z(A)}
\ee
is a cospan of commutative algebras. As in remark \ref{rem:ZAB(f)-HomBfBf}, `act' refers to the map that takes $b \in Z(B)$ to the bimodule map $x \mapsto b.x$ (resp.\ $a \in Z(A)$ to $x \mapsto x.a$).

\begin{lem} \label{lem:centre-2-diagram}
Let $A,B$ be algebras, let $X, Y$ be $A$-$B$-bimodules and let $f : X \to Y$ be a bimodule homomorphism. Then
\be \label{eq:centre-2-diagram}
\cospdd{Z(A)}{\mathrm{act}}{{\Hom_{A|B}(X,X)}{f \circ (-)}}{\mathrm{act}}{Z(B)}{\mathrm{act}}{{\Hom_{A|B}(Y,Y)}{(-)\circ f}}{\mathrm{act}}{\Hom_{A|B}(X,Y)}
\ee
is a 2-diagram.
\end{lem}

\begin{proof}
We need to verify the conditions in definition \ref{def:2-diagram+3-cell}.
The composition of bimodule maps turns $\Hom_{A|B}(X,Y)$ into a right module over $\Hom_{A|B}(X,X)$ and a left module over $\Hom_{A|B}(Y,Y)$. The map $h \mapsto f \circ h$ from $\Hom_{A|B}(X,X)$ to $\Hom_{A|B}(X,Y)$ is a right module map (this translates into $f \circ (h \circ h') = (f \circ h) \circ h'$). Similarly, $(-) \circ f$ is a left module map. Commutativity of the left square amounts to equality of the two maps $x \mapsto f(a.x)$ and $x \mapsto a.f(x)$ for all $a \in Z(A)$, which follows since $f$ is a bimodule map. That the right square commutes follows analogously. Finally, consider the two conditions in \eqref{eq:2diagram-cond-2}. The first condition amounts to equality of the two maps $x \mapsto g(x).b$ and $x \mapsto g(x.b)$ for all $b \in Z(B)$ and $g \in \Hom_{A|B}(X,Y)$, which holds since $g$ is a bimodule map. The second condition can be checked similarly.
\end{proof}

As the constructions will now get somewhat technical, let us just outline in the remark below how the discussion continues from here, leaving the details to \cite{in-prep}.

\begin{rem} \label{rem:centre-2nd-version}
(i) The 2-diagram in \eqref{eq:centre-2-diagram} provides a lax functor
\be
  \bfZ_{A,B} : \Alg(k)(A,B) \longrightarrow \Cosp(Z(A),Z(B)) \ .
\ee
This functor is indeed lax for the following reason: The vertical composition \eqref{eq:2-diag-vertical} of two 2-diagrams of the form \eqref{eq:centre-2-diagram} belonging to bimodule maps $f : X \to Y$ and $g : Y \to Z$ yields a 2-diagram with central term
\be \label{eq:Hom-x_Hom-Hom}
  \Hom_{A|B}(Y,Z) \otimes_H \Hom_{A|B}(X,Y)
  \qquad \text{where} ~~ H\equiv\Hom_{A|B}(Y,Y) \ .
\ee
This space is in general not isomorphic to $\Hom_{A|B}(X,Z)$. So we cannot obtain a functor $\Alg(k)(A,B) \longrightarrow \Cospu(Z(A),Z(B))$ in this way, and consequently not a -- lax or otherwise -- functor from $\Alg(k)$ to $\CALGu(k)$. However, we conjecture that the 2-diagram \eqref{eq:centre-2-diagram} does give rise to a lax functor $\bfZ$ from $\Alg(k)$ to the (conjectural) tricategory $\CALG(k)$.
\\[.3em]
(ii)
If the maps $f,g$ above are isomorphisms, the space \eqref{eq:Hom-x_Hom-Hom} is isomorphic to $\Hom_{A|B}(X,Z)$. In this way, we at least obtain a functor $\bfZ_{A,B} : \Alg(k)(A,B)^{1,0} \longrightarrow \Cospu(A,B)$ and with this also a lax functor
\be \label{eq:Z:Alg21-CALGu}
  \bfZ : \Alg(k)^{2,1} \longrightarrow \CALGu(k) \ .
\ee
(iii)
Denote by $\mathbf{F}$ the subcategory of $\Alg(k)$ consisting of Frobenius algebras with trace-pairing and finite-dimensional bimodules. One can show \cite{in-prep} that the restriction
\be \label{eq:bicat-groupoid-equiv}
  \bfZ : \mathbf{F}^{2,0} \longrightarrow \CALGu(k)^{2,0} \cong \CAlg(k)^{2,0} \cong \alg(k)^{1,0}_\mathrm{com}
\ee
is locally fully faithful. This has the interpretation that all isomorphisms of lattice TFTs without defects (i.e.\ isomorphisms of Frobenius algebras with trace pairing) are implemented by invertible domain walls (i.e.\ bimodules inducing Morita equivalences).
\end{rem}

\begin{rem}
(i) There is a close link between the lattice TFTs with defects and the centre functor just defined. Let
$T : \bord_{2,1}^\mathrm{def,top}(D_2,D_1,D_0) \to \vect_f(k)$ be a lattice TFT with defects as in theorem \ref{thm:lattice-defect-TFT}, and let $\bfD \equiv \bfD[D_2,D_1;T]$ be the 2-category of defect conditions defined in section \ref{sec:2-cat-from-defect-QFT}. Then we have the commuting square
\be \label{eq:lattice-TFT-commuting}
\raisebox{2em}{\xymatrix{
  \bfD^{2,1} \ar[d]^{\Delta} \ar[r]^{\bfE} & \CAlg(k) \ar[d]^{\bfI}
  \\
  \Alg(k)^{2,1} \ar[r]^{\bf Z} & \CALGu(k)
}}
\ee
where the functor $\Delta$ was given in \eqref{eq:Delta-B-Alg-def}, $\bfE$ in \eqref{eq:E-B-CAlg-def}, $\bfI$ in \eqref{eq:I-CAlg-CALG-def}, and $\bfZ$ in \eqref{eq:Z:Alg21-CALGu}. Indeed, evaluating the diagram on an invertible 2-morphism $f : \dwt x \to \dwt y$ for $\dwt x, \dwt y : a \to b$ gives for the upper path and lower path, in this order,
\be \small
\cospdd{Z(A)}{\mathrm{act}}{{\Hom_{A|B}(X,X)}{f \circ (-) \circ f^{-1}}}{\mathrm{act}}{Z(B)}{\mathrm{act}}{{\Hom_{A|B}(Y,Y)}{\id}}{\mathrm{act}}{\Hom_{A|B}(Y,Y)}
~~,~~
\cospdd{Z(A)}{\mathrm{act}}{{\Hom_{A|B}(X,X)}{f \circ (-)}}{\mathrm{act}}{Z(B)}{\mathrm{act}}{{\Hom_{A|B}(Y,Y)}{(-)\circ f}}{\mathrm{act}}{\Hom_{A|B}(X,Y)}
~~,
\ee
These are isomorphic 2-diagrams, and thus equal in $\CALGu(k)$. 
\\[.3em]
(ii) 
The commuting square \eqref{eq:lattice-TFT-commuting} shows that the lattice construction of defect TFTs is an implementation of the centre functor. Conjecturally, the restriction to invertible 2-morphisms can be dropped if one replaces both bicategories on the right hand side with the (equally conjectural) tricategory $\CALG(k)$.
\end{rem}

\subsection{Generalisation motivated by 2d conformal field theory}\label{sec:rCFT-def}

Rational conformal field theories can be build in two steps. In the first step one starts from a rational vertex operator algebra $V$ and finds its modules and the corresponding spaces of conformal blocks. The category $\mathcal{R}ep(V)$ of $V$-modules is a modular category in this case \cite{Huang:1994,Huang2005}.

The second step is combinatorial and consists of assigning a correlator to each world sheet, i.e.\ choosing a particular vector in the space of conformal blocks corresponding to the world sheet, such that the factorisation and locality constraints are satisfied. In the context of vertex operator algebras and for world sheets of genus zero and certain world sheets of genus one, such correlators were constructed in \cite{Huang:2005gz} -- see also the overview \cite{Kong:2009qw}.

The second step can also be solved elegantly for world sheets of arbitrary genus with the help of three-dimensional topological field theory -- provided one assumes that this 3d TFT correctly encodes the factorisation and monodromy properties of conformal blocks at arbitrary genus. The 3d TFT in question is  the Reshitikhin-Turaev 3d TFT obtained from the modular category $\mathcal{R}ep(V)$. This combinatorial construction of CFT correlators in terms of 3d TFT was carried out in \cite{tft1,tft5} -- see also the overview  \cite{Runkel:2005qw} -- and in particular allows for a description of CFT correlators for world sheets with defects \cite{tft1,Frohlich:2006ch}.

\medskip

Generalising the considerations in section \ref{sec:spaces+maps} from 2d TFT to 2d CFT suggests an interesting generalisation of the centre construction, which we now sketch. 

Let us start with the (bi)categories $\alg(k)$ and $\Alg(k)$. Instead of working with algebras and bimodules over a field $k$, that is, with algebras in the symmetric monoidal category of $k$-vector spaces, one considers algebras and bimodules in a general monoidal category $\mathcal{C}$ (in the CFT-context, this is the category $\mathcal{R}ep(V)$). In particular, we do not demand that $\mathcal{C}$ is symmetric or braided (though in the CFT context it is braided).

To generalise $\CAlg(k)$ and $\CALGu(k)$, we need to be able to talk about commutative algebras, so here we consider cospans of commutative algebras in a braided monoidal category. 

There is one major new ingredient when passing from vector spaces to more general categories, which is based on the following observation. For an algebra $A$ in a general monoidal category $\mathcal{C}$ it makes no sense to talk about its centre as a subalgebra commuting with the entire algebra, because the formulation of this condition needs a braiding. A natural candidate to take the role of the centre in the case of general monoidal categories is the so-called {\em full centre} $Z(A)$ of $A$ \cite{Fjelstad:2006aw,Davydov:2009}. This is a commutative algebra which lives in the monoidal centre $\mathcal{Z}(\mathcal{C})$ (see \cite{Joyal:1993}) of the category $\mathcal{C}$. Since $\mathcal{Z}(\mathcal{C})$ is braided, we can talk about commutative algebras there. If $\mathcal{C}$ is the category of $k$-vector spaces, one has the degenerate situation that $\mathcal{Z}(\mathcal{C}) \cong \mathcal{C}$, and so many of the subtleties of the centre-construction are not visible. (In the context of rational CFT, and for modular categories in general, one has $\mathcal{Z}(\mathcal{C}) = \mathcal{C} \boxtimes \bar{\mathcal{C}}$, see \cite[Thm.\,7.10]{Muger2001b}.)

The constructions and results of sections \ref{sec:comm-alg-v1}--\ref{sec:fun-cent-v2} all have analogues in the more general setting of algebras in monoidal categories. For example, an instance of the equivalence \eqref{eq:bicat-groupoid-equiv}, with the corresponding interpretation in terms of domain walls implementing equivalences of CFTs, has been found in \cite[Thm.\,3.14]{Davydov:2010rm}. More details will appear in \cite{in-prep}.

\begin{rem}
As an aside, let us recall an observation from \cite{Schweigert:2006af} which illustrates the usefulness of the 2-category of defect conditions defined in section \ref{sec:2-cat-from-defect-QFT} in the context of rational CFT. Namely, consider a fixed rational CFT (i.e.\ restrict your attention to only one world sheet phase) and consider only topological defects from this world sheet phase to itself, which in addition commute with the holomorphic and anti-holomorphic copy of the rational vertex operator algebra $V$. Then the 2-category $\bfD$ from section \ref{sec:2-cat-from-defect-QFT} has only one object (and so is a monoidal category). It turns out that $\bfD$ is Morita equivalent (in the sense of \cite[Def.\,4.2]{Muger2001a}) to $\mathcal{R}ep(V)$; this follows since $\bfD$ is monoidally equivalent to the category of $A$-$A$-bimodules in $\mathcal{R}ep(V)$ for an appropriate Frobenius algebra $A$ with trace pairing \cite[Sec.\,2]{Frohlich:2006ch}. It also follows (from \cite[Thm.\,3.3]{Schauenburg:2001}) that the monoidal centre $\mathcal Z(\bfD)$ is braided monoidally equivalent to the monoidal centre $\mathcal Z(\mathcal{R}ep(V)) = \mathcal{R}ep(V) \boxtimes \overline{\mathcal{R}ep(V)}$. Thus, quite remarkably, if one knows the one-object 2-category of chiral symmetry preserving topological defects in a rational CFT, one obtains for free the braided monoidal category of representations of its chiral symmetry $V \otimes \bar V$.
\end{rem}

\section{Outlook} \label{sec:outlook}

In this final section we would like to show some further directions that we find interesting and point out some open questions. From the perspective of this article, there are two evident problems which we left untouched:
\begin{enumerate}
\item 
In the introduction we claimed that there are two natural ways in which higher categories arise in field theory: by demanding that the functor defining the field theory assigns data to manifolds of codimension larger than one, or by working with defects of various dimensions. Clearly, one should study these two constructions in unison. We are aware of three works in this direction: one in 2d TFT \cite{Schommer-Pries:2009}, and two in arbitrary dimension -- \cite[Sec.\,4.3]{Lurie:2009aa} and \cite[Sec.\,6.7]{Morrison:2009a} -- both `extended down to points'. A better understanding of the relation between the two appearances of higher categories should allow one to make precise the idea that $n$-dimensional TFT extended down to points is in some sense dual to $n$-dimensional TFT which has defects in all dimensions. 
\item
A symmetric monoidal functor defining a 2d TFT or 2d CFT without defects has a well-known presentation in terms of generators and relations which provides a link with Frobenius algebras \cite{Sonoda:1988,Dijkgraaf:1989,Abrams:1996ty, tft1,Huang:2005gz}. This connection has been useful in the  construction of examples and in classification questions. For field theories with defects in 2d (let alone higher dimensions) such a generators and relations presentation is presently not known. Nonetheless, progress has been made in related questions: an algebraic description of 2d TFT with defects which extends down to points was presented in \cite{Schommer-Pries:2009}, and for planar algebras, generators are given in \cite{Kodiyalam:2004a} and a construction in terms of a 1-morphism in a pivotal strict 2-category is presented in \cite{Ghosh:2008a}. For 2d homotopy TFTs over spaces with at least one of $\pi_1$ or $\pi_2$ trivial, a classification in terms of Frobenius algebras with extra structure is given in \cite{Turaev:1999yf,Brightwell:1999a}.
\end{enumerate}
Apart from these two points, let us list some further miscellaneous points to complement the material presented in this note.

\medskip

\nxt One nice application of quantum field theories with topological defects is the orbifold construction. Here, one introduces a domain wall which implements the `averaging over the orbifold group', together with a selection of lower-dimensional junctions which allow one to glue these domain walls together. The orbifold theory is then defined in terms of a cell-decomposition of the original theory with the `averaging domain wall' placed on the codimension-1 cells. The advantage of this point of view is that the `averaging domain wall' need not actually be given by a sum over group-symmetries, giving rise to a generalisation of the orbifold construction. In the case of 2d rational CFTs, this is described in \cite{Frohlich:2009gb}. It is proved there that any two rational CFTs over the same left/right chiral symmetry algebra can be written as a generalised orbifold of one another. 

\medskip

\nxt In the application of field theories to questions in cohomology one considers field theories `over a space $X$', see e.g.\ \cite{Turaev:1999yf,Brightwell:1999a,Stolz:unpubl}. This means that objects and morphisms of the bordism category are in addition equipped with continuous maps to $X$. For each point $x \in X$, a field theory over $X$ gives a field theory for undecorated bordisms by choosing these continuous maps to be constant with value $x$. The role of $X$ is reminiscent of our $D_n$, the set labelling the top-dimensional domains $M_n$ for an $n$-dimensional field theory with defects. However, in our setting the $D_n$ label attached to a point in $M_n$ is locally constant and may change only across $M_{n-1}$, and each such change has to be accompanied by specifying a domain wall which mediates this change. It would be interesting to have a continuous formulation of the framework presented here to be able to incorporate continuously changing domain conditions via `smeared-out' domain walls and junctions.

\medskip

\nxt While the general setup in section \ref{sec:bord-with-def} allows for non-topological defects, in this note we only studied the topological case. Theories with non-topological defects are much harder to treat and are much less studied. We mention here four examples in 2d CFT:
\\[-1.5em]
\begin{itemize}
\itemsep 0.1em
\item[-] There are only two 2d CFTs in which all conformally invariant domain walls (this includes the topological ones) from the CFT to itself are known\footnote{
  Here `known' means that one has a list of defect operators satisfying a selection of consistency conditions. Conjecturally, this uniquely specifies all conformally invariant defects, at least in `semi-simple theories'. Only defects whose field content (the space $Q(O(x \circ x^*))$) has discrete $(L_0+\bar L_0)$-spectrum are considered.
}: 
the Lee-Yang model and the Ising model, see \cite{Oshikawa:1996dj,Quella:2006de}.
\item[-] 
In \cite{Quella:2006de}, a transmission coefficient was introduced which measures the `non-topo\-logicality' of a domain wall.
\item[-] 
In \cite{Bachas:2007td}, the fusion product of certain non-topological domain walls in the free boson CFT (found in \cite{Bachas:2001vj}) was computed, showing that at least in these theories the notion of fusion makes sense for non-topological domain walls despite the short-distance singularities.
\item[-] 
In the operator-algebraic approach of \cite{Bartels:2009ts}, non-topological domain walls are included from their start and also their fusion is defined. The definition is via Connes' fusion of bimodules and does not involve a short-distance limit.
\end{itemize}

\nxt The centre of an algebra can be interpreted as a `boundary-bulk map' in the following sense. When considering 2d TFT on surfaces with (unparametrised) boundaries, in addition to `closed states' associated to circles, there are `open states' associated to intervals. The open states form a non-commutative Frobenius algebra and the closed states form a commutative Frobenius algebra, which one can take to be the centre of $A$ (see e.g.\ \cite{Lazaroiu:2000rk,Moore:2006dw}). Thus, the centre defines a theory in one dimension higher (here in dimension two) for which the starting theory is a boundary theory (and the boundary is one-dimensional).
The construction in section \ref{sec:centre} can be understood as turning this boundary-bulk map into a functor. There are a number of situations in which such a boundary-bulk map occurs: In 2d rational CFT one finds that the boundary theory determines a unique bulk theory \cite{Fjelstad:2006aw}; algebraically this amounts to the construction of the full centre of an algebra in a monoidal category as briefly mentioned in section \ref{sec:rCFT-def}. In \cite{in-prep} we will show that the bulk-boundary map is functorial also in the rational CFT case. For certain two-dimensional quantum spin-lattices (which are three-dimensional models because of the time direction), edge excitations determine the bulk excitations, see \cite{Kitaev:2011a}. Finally, an analogous (but much more general) result exists for $\mathbb{E}[k]$-algebras (related to algebras over the little-discs operad in $k$-dimensions). Namely, in \cite{Lurie:2009b} a construction is presented which assigns to an $\mathbb{E}[k]$-algebra (in a symmetric monoidal $\infty$-category) its centre, which is an $\mathbb{E}[k{+}1]$-algebra in the same category, see \cite[Cor.\,2.5.13]{Lurie:2009b}.\footnote{IR would like to thank Owen Gwilliam for discussions on this point.}
It would be interesting to understand the precise relation to the constructions presented here.

\bigskip\noindent
{\bf Acknowledgements}: 
The authors would like to thank 
Nils Carqueville,
Jens Fjelstad,
J\"urgen Fuchs,
Andr\'e Henriques,
Chris Schommer-Pries,
and
especially Sebastian Novak
for helpful discussions and comments on a draft of this article. 
AD would like to thank Tsinghua University, and IR the Beijing International Center for Mathematical Research, for hospitality while part of this work was completed. 
LK is supported by the Basic Research Young Scholars Program of Tsinghua University, Tsinghua University independent research Grant No. 20101081762 and by NSFC Grant No.\ 20101301479. 
IR is supported in part by the SFB 676 `Particles, Strings and the Early Universe' of the DFG.

\appendix

\section{Appendix: Bicategories and lax functors} \label{app:bicategories}

In this appendix we recall the definition of bicategories and related notions, see \cite{Benabou:1967} or \cite{Gray:1974,Leinster:1998}.

\begin{defn} \label{def:bicat} 
A bicategory $\mathbf{S}$ consists of a set of objects (in a given universe) and a category of morphisms $\Mor(A,B)$ for each pair of objects $A$ and $B$ together with 
\begin{enumerate}
\item {\it identity morphism:} $\one_A:  \one \to \Mor(A,A)$ for all $A\in \mathbf{S}$, where $\one$ is a category with only one object and only the identity morphism. We will abbreviate $\one_A \equiv \one_A(\one) \in  \Mor(A,A)$,
\item {\it composition functor:} 
$$
\circledcirc_{C,B,A} \,:\, \Mor(B,C) \times \Mor(A, B) \,\longrightarrow \, \Mor(A, C) 
\quad , \quad
(T, S) \longmapsto T\circ S \ ,
$$
\item {\it associativity isomorphisms:} for $A, B, C, D\in \mathbf{S}$, there is a natural isomorphism between 
functors $\Mor(C,D) \times \Mor(B,C) \times \Mor(A,B) \to \Mor(A,D)$:
$$
\alpha \,:\,  \circledcirc_{D,B,A} \circ (\circledcirc_{D,C,B} \times \id) \,\longrightarrow \,
\circledcirc_{D,C,A} \circ (\id \times \circledcirc_{C,B,A})\ ,
$$
\item {\it left and right unit isomorphisms:}
for $A,B \in \mathbf{S}$ there are natural transformations between functors
$\one \times \Mor(A,B) \to \Mor(A,B)$ and $\Mor(A,B) \times \one \to \Mor(A,B)$:
$$
 l\,:\, \circledcirc_{B,B,A} \circ (\one_B \times \id )  \,\longrightarrow \, \id
 \quad , \quad
 r\,:\, \circledcirc_{B,A,A} \circ (\id \times \one_A)  \,\longrightarrow \,  \id \ ,
$$
\end{enumerate}
satisfying the following coherence conditions:
\begin{enumerate}
\item 
{\it associativity coherence:} 
\be  \label{diag:asso-bicat}
\xymatrix{
((S \circ T)\circ U) \circ V \ar[rr]^{\alpha(S,T,U) \circ \id_V}  
\ar[d]_{\alpha(S \circ T,U,V)}  & & (S\circ (T \circ U)) \circ V \ar[d]^{\alpha(S, T \circ U, V)} \\
(S\circ T) \circ (U\circ V) \ar[rd]_(.4){\alpha(S,T,U \circ V)~~~~} &  & S\circ ((T\circ U) \circ V) \ar[ld]^(.4){~~~~\id_S \circ \alpha(T,U,V)} \\
& S\circ (T \circ (U\circ V)) &  
} 
\ee
\item 
{\it identity coherence:} 
\be  \label{diag:triangle-bicat}
\xymatrix{
(S \circ \one_B) \circ T  \ar[rd]_{r(S) \circ \id_T} \ar[rr]^{\alpha(S,\one_B, T)} & & S\circ (\one_B \circ T) \ar[ld]^{~~\id_S \circ l(T)}  \\
& S\circ T & 
}
\ee
\end{enumerate}
\end{defn}

\begin{defn} \label{def:lax-functor}
Let $\mathbf{C}$ and $\bfD$ be two bicategories. A lax functor $\mathbf{F}: \mathbf{C} \rightarrow \bfD$ is a quadruple $\mathbf{F}=(F, \{\mathbf{F}_{(A,B)} \}_{A,B\in \mathbf{C}}, i, m)$ where 
\begin{enumerate}
\item $F$ is a map of objects $X \mapsto F(X)$ for each object $X$ in $\mathbf{C}$,
\item $\mathbf{F}_{(A,B)}: \Mor_\mathbf{C}(A, B)\to \Mor_\bfD(F(A), F(B))$ is a functor for each pair of objects $A,B\in \mathbf{C}$,
\item {\it unit transformation:} natural transformations $i_A:  \one_{F(A)} \to \mathbf{F}_{(A,A)} \circ \one_A$ between two functors $\one \to \Mor_\bfD(F(A), F(A))$ for all $A$,
\item {\it multiplication transformation:} 
$m: \circledcirc_\bfD \circ (\mathbf{F}_{(B,C)} \times \mathbf{F}_{(A,B)}) \to \mathbf{F}_{(A,C)} \circ \circledcirc_\mathbf{C}$, i.e.\ a collection of morphisms $m_{S,T}: \mathbf{F}_{(B,C)}(S) \circ \mathbf{F}_{(A,B)}(T) \to \mathbf{F}_{(A,C)}(S \circ T)$ natural in $S \in \Mor_\mathbf{C}(B,C), T\in \Mor_\mathbf{C}(A, B)$,
\end{enumerate}
satisfying the following commutative diagrams: 
\begin{enumerate}
\item {\it associativity:} for $S \in \Mor_\mathbf{C}(C,D), T\in \Mor_\mathbf{C}(B,C), U\in \Mor_\mathbf{C}(A,B)$, 
$$
\xymatrix{
(\mathbf{F}_{(C,D)}(S) \circ \mathbf{F}_{(B,C)}(T)) \circ \mathbf{F}_{(A,B)}(U) \ar[r]^{\alpha_{\mathbf{D}}} \ar[d]_{m \circ \id} 
& \mathbf{F}_{(C,D)}(S) \circ (\mathbf{F}_{(B,C)}(T) \circ \mathbf{F}_{(A,B)}(U)) \ar[d]^{\id \circ m} \\
\mathbf{F}_{(B,D)}(S\circ T) \circ \mathbf{F}_{(A,B)}(U) \ar[d]_m & \mathbf{F}_{(C,D)}(S) \circ \mathbf{F}_{(A,C)}(T\circ U) \ar[d]^m\\
\mathbf{F}_{(A,D)}((S\circ T) \circ U) \ar[r]^{\mathbf{F}_{(A,D)}(\alpha_{\mathbf{C}})} & \mathbf{F}_{(A,D)}(S\circ (T\circ U)) \, ,
}
$$
\item {\it unit properties:} for $S\in \Mor_\mathbf{C}(A, B)$, 
$$
\xymatrix{
\one_{F(B)} \circ \mathbf{F}_{(A,B)}(S) \ar[r]^{\hspace{0.5cm}l(F(S))} \ar[d]_{i_B\circ \id}  & \mathbf{F}_{(A,B)}(S) \\
\mathbf{F}_{(B,B)}(\one_B) \circ \mathbf{F}_{(A,B)}(S) \ar[r]^(.55)m & \mathbf{F}_{(A,B)}(\one_B \circ S)\, , \ar[u]_{\mathbf{F}_{(A,B)}(l(S))} 
}
$$
$$
\xymatrix{
\mathbf{F}_{(A,B)}(S) \circ \one_{F(A)} \ar[r]^{\hspace{0.5cm}r(F(S))} \ar[d]_{\id \circ i_B}  & \mathbf{F}_{(A,B)}(S) \\
 \mathbf{F}_{(A,B)}(S) \circ \mathbf{F}_{(A,A)}(\one_A)  \ar[r]^(.55)m & \mathbf{F}_{(A,B)}(S\circ \one_A)\, . \ar[u]_{\mathbf{F}_{(A,B)}(r(S))} 
}
$$
\end{enumerate}
\end{defn}

If we reverse all arrows, we obtain the notion of {\it oplax functor}. Given a lax functor $\mathbf{F}$, if the natural transformations $i$ and $m$ are actually isomorphisms, then $\mathbf{F}$ is called a functor. 

Let $P$ be a property of a functor between 1-categories like full, faithful, essentially surjective, etc. We say that a (lax, oplax or neither) functor is {\em locally} $P$, if for all objects $A,B$ the functors $\mathbf{F}_{(A,B)}$ have property $P$. 

\newcommand\arxiv[2]      {\href{http://arXiv.org/abs/#1}{#2}}
\newcommand\doi[2]        {\href{http://dx.doi.org/#1}{#2}}
\newcommand\httpurl[2]    {\href{http://#1}{#2}}

\end{document}